\documentclass[11pt]{amsart}

\usepackage{amsmath,amsthm,amsfonts}
\usepackage[text={6in,8.5in},centering,letterpaper,dvips]{geometry}
\usepackage{amssymb,latexsym}
\usepackage{cite}
\usepackage{hyperref}
\usepackage{mathrsfs}
\usepackage{enumerate}
\usepackage{graphicx}

\usepackage[all,arc]{xy}

\title{Unoriented Cobordism Maps on Link Floer Homology}
\author{Haofei Fan}
\date{}

\newtheorem{thm}{Theorem}[section]
\newtheorem{cor}[thm]{Corollary}
\newtheorem{lem}[thm]{Lemma}
\newtheorem{prop}[thm]{Proposition}

\newtheorem{quest}[thm]{Question}

\theoremstyle{definition}
\newtheorem{defn}[thm]{Definition}
\newtheorem{exmp}[thm]{Example}

\theoremstyle{remark}
\newtheorem{rem}[thm]{Remark}

\makeatletter
\let\c@equation\c@thm
\makeatother
\numberwithin{equation}{section}

\newcommand{\x}{\mathbf{x}}
\newcommand{\y}{\mathbf{y}}

\newcommand{\cal}{\mathcal}
\newcommand{\bs}{\boldsymbol}

\newcommand{\lk}{\mathbb{L}}

\begin{document}
\maketitle

\begin{abstract}
We study the problem of defining maps on link Floer homology induced
by unoriented link cobordisms. We provide a natural notion of link
cobordism, disoriented link 
cobordism, which tracks the motion of index zero and index three critical
points. Then we construct a map on unoriented link Floer homology
associated to
a disoriented link cobordism. Furthermore, we give a
comparison with
Oszv\'{a}th-Stipsicz-Szab\'{o}'s and Manolescu's constructions of link
cobordism maps for an
unoriented band move.
\end{abstract}

\tableofcontents{}

\section{Introduction}

Heegaard Floer homology is an invariant of three-manifolds, introduced by
Ozsv\'{a}th and Szab\'{o} in \cite{Ozsvath2004c}. 
Knot
Floer homology is a variation of Heegaard Floer homology, which was
discovered by Ozsv\'{a}th and Szab\'{o}\cite{Ozsvath2004a} and independently
by Rassmussen \cite{Rasmussen2003}. Later, it was generalized to  link Floer
homology in \cite{Ozsvath2008b} and proved to be a powerful and successful tool for studying knots and links in
three-manifolds: it detects the Thurston norm of the link complement
\cite{Ozsvath2008a}; it
detects the fiberedness of a knot \cite{Ni2007}; one can extract a tau-invariant from it,
and get a lower bound of four-ball genus of a knot or link \cite{Ozsvath2003c}. There is
also a version called unoriented link Floer homology that is independent of the
link orientation, see \cite{Ozsvath2015} and \cite{Ozsvath2014}.

While searching for applications of link Floer Homology, a natural
question arises: whether an oriented (or unoriented resp.) link
cobordism induces a map on link Floer
homology (or unoriented link Floer homology resp.).

In \cite{Juhasz2006}, Juh\'{a}sz introduced  sutured Floer Homology. Later he
provided a way to construct cobordism map for sutured manifolds, see
\cite{Juhasz2014}. In particular, he built a notion of cobordism, the
decorated link cobordism, which contains not only the link cobordism
surface but also a family of one-manifolds on the surface. These
one-manifolds provide extra data for defining the cobordism 
map. Juh\'{a}sz and Thurston also proved the naturality of link Floer homology in
\cite{Juhasz2012}. In \cite{Juhasz2016}, Juh\'{a}sz showed that a decorated link
cobordism induces a map on sutured Floer homology.
In \cite{Zemke2016}, Zemke generalized the idea in \cite{Sarkar2015} and
established a way to study basepoints moving maps by using
quasi-stabilization. Later, following Juh\'{a}sz's framework, Zemke
constructed link cobordism maps on link Floer homology and showed the
invariance of this construction, see \cite{Zemke2016a}. 

Notice that all of the above works are for oriented link
cobordisms. For unoriented link cobordisms in \cite{Ozsvath2015},
Ozsv\'{a}th, Stipsicz and Szab\'{o} construct maps
on unoriented grid homology, which is the grid diagram version of unoriented
link Floer homology. However, they did not prove the invariance of the
maps. Inspired by these works, we will introduce a
natural link
cobordism notion and construct cobordism maps on
 unoriented link Floer homology.

The link category we study is made of the disoriented links. The
objects of this category, disoriented links, are
constructed as follows. Let $L$ be a link (with no assigned
orientation) in
a closed oriented three-manifold $Y$. Suppose that $\mathbf{p}$ and
$\mathbf{q}$ are two sets of points which appear alternatively on each
component of $L$. The set $L\backslash(\mathbf{p}\cup\mathbf{q})$
consists of $2n$-arcs $\mathbf{l}=\{l_1,\cdots,l_{2n}\}$, which we
orient from $q$ to $p$, where $n$ is the number of $p$ points. 
A \textbf{\textit{disoriented link}}  is the quadruple
$\cal{L}=(L,\mathbf{p},\mathbf{q},\mathbf{l})$. The definition of
disoriented links is inspired by the Morse function compatible with $L$. 
An example of
disoriented link is shown in Figure \ref{fig:intro-01}.

\begin{figure}
  \centering
  \includegraphics[scale=0.8]{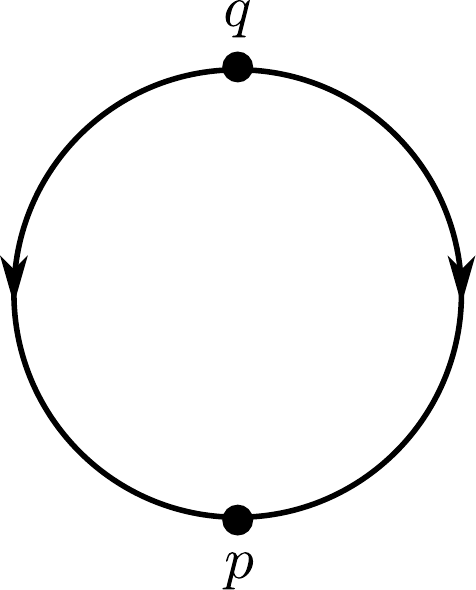}
  \caption{A disoriented link.}
  \label{fig:intro-01}
\end{figure}

Let $H=(\Sigma,\bs{\alpha},\bs{\beta},\mathbf{O})$ be the
pointed Heegaard diagram induced by a Morse function $f$ compatible
with $\cal{L}$, where
$\mathbf{O}$ is the basepoints set $(\Sigma\cap L)$. As in 
\cite{Ozsvath2015}, we can construct a $\delta$-graded, unoriented link Floer chain complex $CFL'(H)$ over the ring $\mathbb{F}_2[U]$. 
 The differential $\partial$ acting on a generator $\x\in\mathbb{T}_\alpha\cap\mathbb{T}_\beta$ is given by:
\[\partial \x
  =\sum_{\y\in\mathbb{T}_\alpha\cap\mathbb{T}_\beta}\sum_{\phi\in\pi_2(x,y),\mu(\phi)=1}\#(\cal{M}(\phi)/\mathbb{R})U^{n_{\mathbf{O}}(\phi)}\y,\]
  where $n_{\mathbf{O}}$ is equal to the sum $\sum_{o_i\in \mathbf{O}}n_{o_i}$. 
 The relative $\delta$-grading
between two generators $\x$ and $\y$ is given by
\[
  \delta(\x,\y)=\mu(\phi)-n_{\mathbf{O}}(\phi).
\]

If the link $L$ is \textit{\textbf{homologically even}}, which
means $[L]=2 a\text{, for some } a\in H_1(Y;\mathbb{Z})$, then the $\delta$-grading is a
$\mathbb{Z}$-grading. By tracking the proof of the naturality of
link Floer homology in \cite{Juhasz2012}, we know that the homology
$HFL'$ of $CFL'$ is an invariant of the
disoriented link $\cal{L}$.

A \textit{\textbf{disoriented link cobordism}} $\mathfrak{W}=(\cal W,\cal{F},\cal{A})$ from  $\cal{L}^0=(L^0,\mathbf{p}^0,\mathbf{q}^0,\mathbf{l}^0)$ in $Y^0$ to
  $\cal{L}^1=(L^1,\mathbf{p}^1,\mathbf{q}^1,\mathbf{l}^1)$ in $Y^1$ contains two groups of data:
\begin{enumerate}[(D1)]
\item The data of the link cobordism, denoted by $(\cal{W},\cal{F})$: The manifold
  $\cal{W}=(W,\partial W)$ is a cobordism from $Y^0$ to $Y^1$. The
  surface $\cal{F}=(F,\partial F)$ is embedded in $(W,\partial W)$ with its
  boundary $\partial W$ identified with the links determined by
  $\cal{L}^0$ in $Y^0$ and $\cal{L}^1$ in $Y^1$.
\item The motion of the $\mathbf{p}$ and $\mathbf{q}$ points 
  $\cal{A}=(A,\partial A)$: This is an
  oriented one-manifold $\cal{A}=(A,\partial A)$ properly embedded in
  $(F,\partial F)$. The boundary $\partial A$ identified with the
  zero manifold
  $\mathbf{q}^0-\mathbf{p}^0+\mathbf{p}^1-\mathbf{q}^1$.
\end{enumerate}
An example of a disoriented link cobordism is shown in
Figure \ref{fig:intro-02}.  A similar
construction also appears in Khovanov homology, for details see Remark \ref{thm:khovanov}.

\begin{figure}
  \centering
  \includegraphics[scale=1]{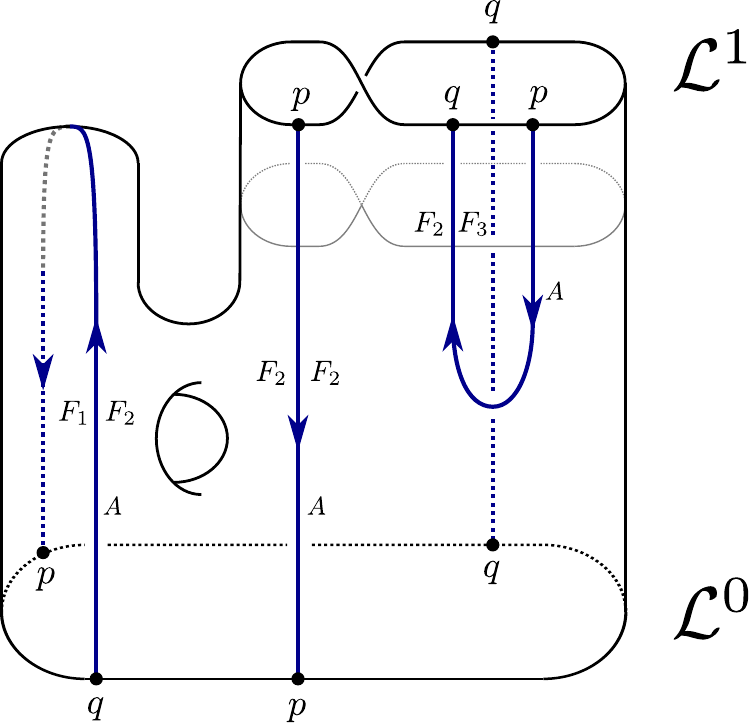}
  \caption{A disoriented link cobordism from $\cal{L}^0$ to $\cal{L}^1$. The motion of $p$ and $q$ is marked as blue
curves.}
  \label{fig:intro-02}

\end{figure}

Similar to what is in \cite{Zemke2016a}, one can find a parametrized
Kirby decomposition of a disoriented link cobordism. However, when
defining maps induced by a four-dimensional two-handle attachment, we can not
establish the correspondence between the $\text{Spin}^c$-structure
for the four-manifold and the equivlance class of triangles
which come from the Heegaard triple subordinate to the
two-handle. For details, see Section \ref{sec:holotri}.

To avoid this issue, we only consider surfaces inside $Y\times I$ and
use the language of ambient isotopy of surfaces in four-manifolds
instead of handle decompositions of the four-manifold. Our main result
is the following: 

\begin{thm}\label{sec:introduction} Suppose that the cobordism
  $\mathcal{W}$ is a product $(Y\times
  I,\partial(Y\times I))$.
Let  $\mathfrak{W}=(\cal
W,\cal{F},\cal{A})$ be a
disoriented link cobordism from $( Y, \cal{L}^0)$ to
$(Y,\cal{L}^1)$, such that the embedding $\mathcal{F}$
induces a trivial map $\mathcal{F}_*: H_2(F,\partial F)\rightarrow
H_2(Y\times I,\partial(Y\times I))$. Then for a torsion
$\textnormal{Spin}^c$-structure $\mathfrak{s}$ of $Y\times I$, we can
define a map:
 \[F_{\mathfrak{W},\mathfrak{s}}:HFL'( Y,
 \cal{L}^0,\mathfrak{s}|_{Y\times\{0\}})\rightarrow HFL'(Y,\cal{L}^1,\mathfrak{s}|_{Y\times\{1\}}), \]
 which is an invariant of $(\mathfrak{W},\mathfrak{s})$. Furthermore, 
 the map $F_{\mathfrak{W},\mathfrak{s}}$ satisfies the composition law,
 i.e. if $\mathfrak{W}=\mathfrak{W}^1\cup\mathfrak{W}^2$, where the
 disoriented link cobordisms $\mathfrak{W}^1$ and $\mathfrak{W}^2$
 satisfy the same conditions as $\mathfrak{W}$, then 
\[F_{\mathfrak{W}^2,\mathfrak{s}|_{\mathfrak{W}^2}}\circ
F_{\mathfrak{W}^1,\mathfrak{s}|_{\mathfrak{W}^1}}=F_{\mathfrak{W},\mathfrak{s}}.\] 
\end{thm}

Note that a similar statement appears in \cite[Theorem
A,B]{Zemke2016a} for arbitrary oriented link cobordisms (not only for
cylinders $Y\times I$). 

In order to construct the cobordism map $F_{\mathfrak{W}}$, we
introduce another group of data (D3) on the surface $(F,\partial F)$,
which allows us to extract Heegaard data. In detail, the data (D3)
tracks the motion of basepoints: this is a one-manifold
$\cal{A}_{\Sigma}=(\cal{A}_{\Sigma},\partial\cal{A}_{\Sigma})$ embeded
in $(F,\partial F)$. The boundary $\partial A _{\Sigma}$ are
basepoints $O^0$ in $Y^0$ and $O^1$ in $Y^1$. Furthermore the
one-manifold
 $\cal{A}_{\Sigma}$ cut the surface $F$ into two parts,
$F_{\alpha}$ and $F_\beta$, each of which is a collection of surfaces
(can be non-orientable) embeded in $Y$. 

The workflow of our construction is shown in Figure \ref{fig:intro-workflow}. 

\begin{enumerate}[\textbf{Step} 1:]
\item we lift the disoriented link cobordism (containing data (D1) and
  (D2)) to a bipartite disoriented link cobordism (containing data
  (D1),(D2) and (D3)). For the definition of bipartite disoriented
  link cobordism see Section \ref{sec:categ-3:-bipart}. This lifting is not unique, see
  Section \ref{sec:relat-betw-three}.
\item we construct cobordism maps for the bipartite
  disoriented link cobordism. In fact, we extract Heegaard data from
  (D1)+(D3). The groups of data (D1)+(D2) help us choose generators
  when defining maps induced by band moves and quasi-stabilizations.  
\item we show that the cobordism maps defined in
  Step 2 are independent of liftings (or the data (D3) in other words),
  at the level of $HFL'$.  
\end{enumerate}

\begin{figure}
  \centering
  \includegraphics[scale=0.6]{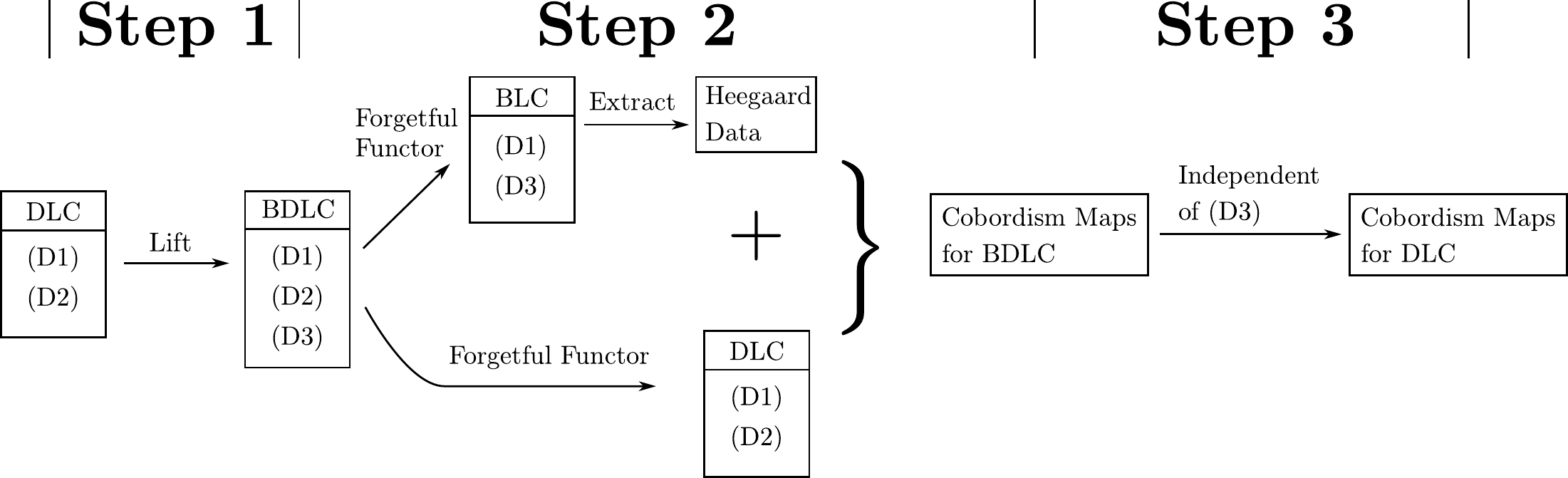}
  \caption{Workflow of the construction: DLC (disoriented link
    cobordism), BLC (bipartite link cobordism), BDLC (bipartite
    disoriented link cobordism); the data (D1) is the link cobordism
    surface; the data (D2) is the motion of index zero/three critical
    points; the data (D3) is the motion of basepoints.}
  \label{fig:intro-workflow}
\end{figure}

\subsection{The difference between unoriented cobordism and oriented
  cobordism}

For an oriented band move, the number of link components will be
changed. On the other hand, there is at least one pair of basepoints for
each link component. Therefore, we deduce that, for a pointed Heegaard
triple subordinate to an oriented band move, we need at least four
basepoints on the Heegaard triple. 

However, when the cobordism surface $(F,\partial F)$ is
non-orientable, a band move may not change the number of link
components. Furthermore, there exists a two-pointed Heegaard
triple subordinate to an unoriented band move between two knots. One
example of such a cobordism is given below.

We consider a band move from the trefoil to the unknot shown in Figure
\ref{fig:intro-03}. This is an example of a band move of type I (see
section \ref{sec:categ-2:-bipart} for the definition of type I and type II band move). For the
Heegaard triple subordinate to a type I band move, the induced
diagram $H_{\beta\gamma}$ no longer represents an unlink, but it is still
a homologically even link in
$\#^n(S^1\times S^2)$. One can compare this fact for type I band moves
with the fact about oriented band moves in \cite[Lemma
6.6]{Zemke2016a}. To deal with the link represented by $H_{\beta\gamma}$, it is necessary to build the unoriented link
Floer homology theory for homologically even links.

\begin{figure}
  \centering
  \includegraphics[scale=0.8]{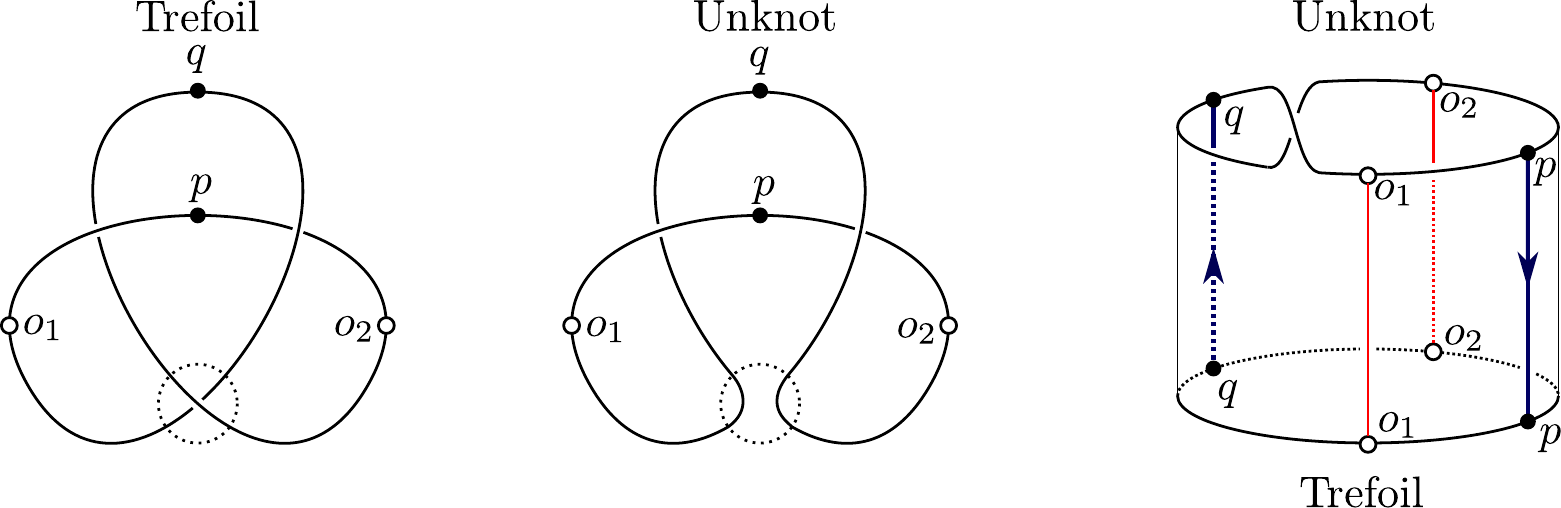}
  \caption{Band move from trefoil to unknot.}
  \label{fig:intro-03}
\end{figure}

\subsection{Organization}

In Section 2, we recall some notions of link Floer homology. In
Section 3, we introduce three link categories: disoriented links,
bipartite links and bipartite disoriented links. We will also discuss
the relation between the three categories. In Section 4, we will
construct the Heegaard triple subordinate to a band move of bipartite
links. In Section
5, we will construct the bipartite link Floer curved chain complex. We
focus our discussions on $\text{Spin}^c$-structure, admissibility and
associativity. In Section 6, we will construct the cobordism map for
band moves of bipartite disoriented links on unoriented link Floer chain complex. Particularly,
we will compare our construction with the band move maps
defined by
Ozsv\'{a}th, Stipsicz and Szab\'{o} in \cite{Ozsvath2015} and the version of cobordism maps
defined by Manolescu in \cite{Manolescu2007}. In Section 7, we define
the bipartite link cobordism maps induced by quasi-stabilizations/destabilizations and
disk-stabilizations/destabilizations. In
Section 8, we will prove the commutation between certain cobordism
maps. We also provide relations between the cobordism maps induced by
band moves ( and quasi-stabilization/destabilizations resp.). 
 In Section 9, we will prove Theorem \ref{sec:introduction}.

\subsection{Further developments and possible applications}

As an application of the  disoriented link cobordism theory,
we will extend the involutive upsilon invariant defined by Hogancamp
and Livingston \cite{Hogancamp2017} from
knots to links. Furthermore, we will study the relation between
involutive upsilon invariant and the unoriented four-ball genus for
 disoriented link cobordism. These discussions will appear in
an upcoming paper \cite{Fan2018}.

Another question one can think about is the following:

\begin{quest}
  Suppose we have a disoriented link cobordism
  $\mathfrak{W}=(\mathcal{W},\mathcal{F},\mathcal{A})$ from
  $(Y^0,\mathcal{L}^0)$ to $(Y^1,\mathcal{L}^1)$. Here we no longer require
   the four-manifold $W$ is cylindrical. Given a torsion
  $\textnormal{Spin}^c$-structure $\mathfrak{s}$, can we still define
  a map $F_{\mathfrak{W},\mathfrak{s}}:
  HFL'(Y^0,\mathcal{L}^0,\mathfrak{s}|_{Y^0})\rightarrow HFL'(Y^1,\mathcal{L}^1,\mathfrak{s}|_{Y^1})$? If so,
  can we get a $\delta$-grading shifts formula for the map $F_{\mathfrak{W},\mathfrak{s}}$?
\end{quest}

\subsection{Acknowledgement}

I am grateful to my advisor, Ciprian Manolescu for the helpful suggestions and
especially for the comments on bipartite disoriented link cobordism
and quasi-stabilization. I am also grateful to Ian Zemke for the discussion on the band moves for unoriented link cobordism.

The author was partially supported by the FRG grant DMS-1563615 from
the NSF. 

\section{Preliminaries}\label{sec:pre}

\subsection{Multi-pointed links}\label{sec:multi}

In this subsection, we recall the Heegaard diagrams for oriented
multi-pointed links. For details, see \cite{Ozsvath2008b}.

\begin{defn}\label{def:multi}
  An \textit{\textbf{oriented multi-pointed link}} is a triple $(Y,L,\mathbf{w},\mathbf{z})$,
  such that:
\begin{itemize}
\item $Y$ is a closed three-manifold, $L$ is an oriented link in $Y$.
\item The set
  $\mathbf{w}=\{w_1,\cdots,w_n\},\mathbf{z}=\{z_1,\cdots,z_n\}$ are
  collection of basepoints on 
  $L$.
\item On each component $L_i$ of $L$, there are basepoints
  $w_{ij},z_{ij}$ appear alternatively on $L_i$, where $j=1,\cdots,n_i$. 
\end{itemize}
\end{defn}

We can construct a
self-indexed Morse function $f$ \textit{\textbf{compatible}} with the triple
$(Y,L,\mathbf{w},\mathbf{z})$, such that, the link $L$ is a union of trajectories
connecting index three and zero critical point of $f$.
A trajectory on $L$
intersect with the surface $\Sigma= f^{-1}(\frac{3}{2})$ at a point $o$. If
the direction of $L$ agrees with the direction of this trajectory, we mark the
intersection $w$, and denote this trajectory a $w$-arc. Otherwise, we 
mark the intersection $z$, and denote this trajectory a
$z$-arc. 

We say that a \textit{\textbf{Heegaard diagram}} $H=(\Sigma,\alpha,\beta, \mathbf{w},\mathbf{z})$ is
\textit{\textbf{compatible}} with a multi-pointed link $(Y,L,\mathbf{w},\mathbf{z})$, if it comes from a
self-indexed Morse function $f$ compatible with this multi-pointed link
$(Y,L,\mathbf{w},\mathbf{z})$. In detail, this Heegaard diagram $H$ should satisfy:
\begin{itemize}
\item The surface $\Sigma$ is identified with $f^{-1}(\frac{3}{2})$.
\item The $\boldsymbol{\alpha}$-curves on $\Sigma$ are the intersections
  of $\Sigma$ with the stable manifolds of all index one critical
  points of $f$. 
\item The $\boldsymbol{\beta}$-curves on $\Sigma$ are the intersections
  of $\Sigma$ and the unstable manifolds of all index two critical
  points of $f$. 
\item The link $L$ as a union of trajectories of $f$ intersects with
  $\Sigma$ at a $w$ basepoint if this trajectory has the same
  direction with $L$, at a $z$ basepoint if it has the opposite
  direction. 
\end{itemize}

Therefore, we implant the oriented link data
$(L,\mathbf{w},\mathbf{z})$
 into a 2n-pointed Heegaard
diagram. Such a diagram $H$ is a surface $\Sigma$ of genus $g$
  equipped with two families of pairwise disjoint simple closed
  curves, 
  $\{\alpha_1,\cdots,\alpha_{g+n-1}\}$ and
  $\{\beta_1,\cdots,\beta_{g+n-1}\}$ and two sets of basepoints
  $\mathbf{w}=\{w_1,\cdots,w_n\}$ and $\mathbf{z}=\{z_1,\cdots,z_n\}$,
  satisfying:

  \begin{itemize}
  \item  The vector spaces
    $\text{Span}(\alpha_1,\cdots,\alpha_{g+n-1})$ and
    $\text{Span}(\beta_1,\cdots,\beta_{g+n-1})$ in $H_1(\Sigma)$ is
    g-dimensional. 

  \item  The space $\Sigma\backslash \bs{\alpha} $ has $n$ connected sets
    $A_1,\cdots,A_n$. Similarly, the space $\Sigma\backslash
    \bs{\beta}$ also has $n$ connected sets $B_1,\cdots,B_n$.
  \item Each component $A_i$ (or $B_j$ resp.) contains exactly
    one pair  of basepoints
    $(w,z)$ (or $(z,w)$ respectively) in
    $(\{\mathbf{w}\},\{\mathbf{z}\})$ (or in
    $(\{\mathbf{z}\},\{\mathbf{w}\})$ resp.).
\end{itemize}

\subsection{Link Floer homology}\label{sec:pre:lf}

In this section we recall the construction of various versions of the link Floer
chain complex $CFL^\circ$ from a $2n$-pointed Heegaard diagram. For
details, see \cite{Ozsvath2008b},\cite{Manolescu2010},\cite{Zemke2016}.

Consider a $2n$-pointed Heegaard diagram $H$ representing the oriented link
$(L,\mathbf{w},\mathbf{z})$. Suppose $L$ is null-homologous in
$Y$. 
The generators of $CFL^\circ$ are the intersections of the two
Lagrangian submanifolds $\mathbb{T}_\alpha$ and $\mathbb{T}_\beta$
inside $\text{Sym}^{g+n-1}(\Sigma)$. We view a generator as a
$(g+n-1)$-tuple of points $\x=(x_1,\cdots,x_{g+n-1})$, where $x_i$'s
are the intersections 
 between $\bs{\alpha}$ and 
$\bs{\beta}$ curves on $\Sigma$.

For each of the generator $\x\in
\mathbb{T}_\alpha\cap \mathbb{T}_\beta$, we can construct a
$\text{Spin}^c$-structure $\mathfrak{s}_{\mathbf{w}}(\x)$ by
picking a non-varnishing vector field on the complement of $Y$ by removing
neighborhoods of flow lines through basepoints $\mathbf{w}$ and the
generator $\x$. Similarly, we can construct $\text{Spin}^c$-structure
map $\mathfrak{s}_{\mathbf{z}}$ from
$\mathbb{T}_\alpha\cap\mathbb{T}_\beta$ to
$\text{Spin}^c(Y)$. 
 Furthermore, the difference 
$\mathfrak{s}_{\mathbf{w}}(\x)-\mathfrak{s}_{\mathbf{z}}(\x)$ is equal
to the Poincar\'{e} dual $\text{PD}[L]$ of the class $[L]\in H_1(Y)$.
As the link $L$ we consider is null-homologous, there is no difference
between the two $\text{Spin}^c$-structure
$\mathfrak{s}_{\mathbf{w}}(\x)$ and $\mathfrak{s}_{\mathbf{z}}(\x)$.
(For details of the construction, see \cite[Section 2.6]{Ozsvath2004c})

Suppose $\mathfrak{s}$ is an torsion $\text{Spin}^c$-structure of
$Y$. We can associate a Maslov $\mathbb{Z}$-grading $gr_{\mathbf{w}}$
for all generators $\x$ with
$\mathfrak{s}_{\mathbf{w}}(x)=\mathfrak{s}$. As $L$ is
null-homologous, we also have a Maslov $\mathbb{Z}$-grading
$gr_{\mathbf{z}}$ for generators in class $\mathfrak{s}$. The
difference $gr_{\mathbf{w}}(\x)-gr_{\mathbf{z}}(\x)$ is equal to twice
of the Alexander grading $\mathbf{A}(\x)$.   

A \textit{\textbf{Heegaard Data}} $\cal{H}$ is a pair $(H,J_t)$, where $J_t$ is a generic one parameter
family of almost complex structure on $\text{Sym}^{g+n-1}(\Sigma)$. 
We assign to each basepoint $w_i$ a variable $U_i$ and
to each basepoint $z_j$ a variable $V_j$. 
Given a
strongly $\mathfrak{s}$-admissible Heegaard data for the link
$(Y,L,\mathbf{w},\mathbf{z})$, we can define the a free
$\mathbb{F}_2[U_{1},\cdots,U_{n},V_{1}\cdots, V_n]$-module $CFL^-(\cal{H},\mathfrak{s})$
with the generators $\x\in \mathbb{T}_\alpha\cap\mathbb{T}_\beta$ with
$\mathfrak{s}_{\mathbf{w}}(\x)$ equal the given
$\text{Spin}^c$-structure $\mathfrak{s}$. The module
$CFL^\infty(\cal{H},\mathfrak{s})$, which contains elements of the
form $\prod_{i,j} U^{k_i}_i V^{l_j}_j\x$, is the localization of
$CFL^-(\cal{H},\mathfrak{s})$. 
We denote by $CFL^+(\cal{H},\mathfrak{s})$
the quotient module
$CFL^\infty(\cal{H},\mathfrak{s})\slash CFL^-(\cal{H},\mathfrak{s})$.

There is an endomorphism 
\[\partial: CFL^\circ(\cal{H},\mathfrak{s})\rightarrow
  CFL^\circ(\cal{H},\mathfrak{s}),\]
which makes $CFL^\circ(Y,\lk,\mathfrak{s})$ into a curved chain
complex (see \cite{Zemke2016} and \cite{Zemke2016a}). 
In detail, the endomorphism acts on a generator $\x$ is given
by:
\begin{equation}\label{eq:1}
  \partial \x = \sum_{\y\in
    \mathbb{T}_\alpha\cap\mathbb{T}_\beta}\sum_{\phi\in
    \pi_2(\x,\y),\mu(\phi)=1}\#(\cal{M}(\phi)/\mathbb{R})\prod_{i,j} U_i^{n_{w_i}(\phi)}V_j^{n_{z_j}(\phi)}\y
\end{equation}

From \cite[Lemma 2.1]{Zemke2016}, we have 

\begin{equation}\label{eq:endocurve}
    \partial^2 = \sum_{i}(U_{i,1}V_{i,2}+V_{i,2}U_{i,2}+\cdots+U_{i,n_i}V_{i,1}+U_{i,1,}V_{i,1}).
\end{equation}

Here $i$ refers the $i$-th component $L_i$ of $L$. We assign to each
basepoint $w_{i,j}$ a variable $U_{i,j}$ and to each basepoint
$z_{i,j}$ a variable $V_{i,j}$. The basepoints
$w_{i,1},z_{i,2},\cdots,w_{i,n_i},z_{i,1}$ appears clockwise on $L_i$.

\begin{rem}
  Usually, the chain complex $CFL^-(\cal{H},\mathfrak{s})$ refers to
  the complex defined by setting all $V_j=1$. When we say a curved
  chain complex 
  $CFL^-(\cal{H,\mathfrak{s}})$, we mean the module together with the
  endomorphism defined in \eqref{eq:1}.
\end{rem}

\subsection{Unoriented link Floer homology}\label{sec:pre:ulfh}

The unoriented link Floer chain complex $CFL'$ was first introduced in 
\cite{Ozsvath2015} as a special case of the $t$-modified link Floer chain complex
$tCFK$ by setting $t=1$. 

We let the complex 
$CFL'(\cal{H},\mathfrak{s})$ be the tensor product of the curved chain
complex $CFL^-(\cal{H},\mathfrak{s})$ with the quotient ring $ R =
\mathbb{F}_2[U_1,\cdots,U_n,V_1,\cdots,V_n,U]/I$, where the ideal 
$I$ is generated by all $U_i-U$ and $V_j-U$. The endomorphism 
\[ \partial \x = \sum_{\y\in
  \mathbb{T}_{\alpha}\cap\mathbb{T}_{\beta}}\sum_{\phi\in
  \pi_2(\x,\y),\mu(\phi)=1}\#(\cal{M}(\phi)/\mathbb{R})U^{n_{\mathbf{w}}(\phi)+n_{\mathbf{z}}(\phi)}\y,
\]
becomes a differential. Here $n_{\mathbf{w}}=\sum_i n_{w_i} \text{,}
  n_{\mathbf{z}}=\sum_i n_{z_i}$.
  
  As before, we assume the link $L$ is null-homologous and the
$\text{Spin}^c$-structure $\mathfrak{s}$ is torsion. Then the two
 $\mathbb{Z}$-grading $gr_{\mathbf{w}}(\x)$ and $gr_{\mathbf{z}}(\x)$ are well-defined.
 The \textbf{$\delta$-grading} of the generator $\x$ in
 $CFL'(\cal{H},\mathfrak{s})$ is defined to be: 
\begin{equation}
\delta(\x)=\frac{1}{2}(gr_{\mathbf{w}}(\x)+gr_{\mathbf{z}}(\x)).
\end{equation}
 We call the $\delta$-graded chain complex
$CFL'(\cal{H},\mathfrak{s})$ the \textbf{\textit{unoriented chain
    complex}}. 
The homology group $H_*(CFL'(\cal{H},\mathfrak{s}))$ is
called the \textit{\textbf{unoriented link Floer homology}}. As we
assign all basepoints $\mathbf{w}$ and $\mathbf{z}$ the same variable
$U$, we lose the infomation of the orientation of $L$. Hence $H_*(CFL'(\cal{H},\mathfrak{s}))$ is an invariant for unoriented
link.   

\subsection{Unoriented grid homology}\label{sec:pre:ugh}  

Grid diagrams provide us with a way to describe the Floer homology of
 oriented links in 
 $S^3$ combinatorially, see 
\cite{Manolescu2007a},\cite{Manolescu2009}. In this subsection, we briefly talk about
how to calculate the unoriented link Floer homology from a grid diagram.
For details, see \cite{Ozsvath2015}.

Suppose $\mathbb{G}$ is a special $2n$-pointed toroidal Heegaard
diagram such that each connected component of $\Sigma - \bs{\alpha}$
or $\Sigma - \bs{\beta}$ is an annulus. Conventionally, we denote the
$w$-basepoint as $X$ and $z$-basepoint as $O$ in the grid diagram.

Given such a grid diagram $\mathbb{G}$, we can define \textit{grid
  chain complex} 
$GC^\circ(\mathbb{G})$ and its \textit{grid homology} $H_*(GC^\circ(\mathbb{G}))$. A generator of grid chain complex
$\mathbb{G}$ is an $n$-tuple $\x=\{x_1,\cdots,x_n\}$, such that each
$\alpha$ and $\beta$-curves contain exactly one of the $x_i$'s.We call these generators \textit{grid states} and denote the
set of all grid states $\mathbf{S}(\mathbf{G})$.

Let Rect$(\x,\y)$ be the set of rectangles from $\x$ to $\y$ and the
weight $\mathbf{W}(r)$ be equal to 
$\#(r \cap (\mathbb{X}\cup\mathbb{O}))$. The 
unoriented grid chain complex $GC'(\mathbb{G})$ is a $\delta$-graded
$\mathbb{F}_2[U]$-module freely generated by grid states $\x$, with
differential
\[\partial \x=\sum_{\y\in\mathbf{S}(\mathbb{G})}\sum_{r\in
    \text{Rect}^0(\x,\y)}U^{\mathbf{W}(r)}\y.\]

Given a grid diagram $\cal G$, the delta grading of a generator $\x$
can be calculated as follows.
Let $P,Q$ be two finite subset of $\mathbb{R}^2$. The function $\cal
I(P,Q)$ count the number of pairs $\mathbf{p}\in P$ and $\mathbf{q}\in
Q$, such that the vector $\mathbf{q}-\mathbf{p}$ lie in the first
quadrant. We define a symmetric function:
\[\mathcal{J}(P,Q)=\frac{\mathcal{I}(P,Q)+\mathcal{I}(Q,P)}{2}\].

Then the $\delta$-grading is given by the formula:
\begin{equation}
  \delta(\x)=\frac{1}{2}(\mathcal{J}(\x-\mathbb{O},\x-\mathbb{O})+\mathcal{J}(\x-\mathbb{X},\x-\mathbb{X}))+\frac{n-l}{2}+1
\end{equation}

We know that 
if $\cal{H}$ is a Heegaard diagram induced from a grid diagram
$\mathbb{G}$, then the chain complex $(CFK^{-}(\cal{H},\partial^-_K)$
is isomorphic to $(GC^-(\mathbb{G}),\partial_{\mathbb{X}}^-)$. In
fact, there is an identification of the bigrading in grid homology and
link Floer homology. Similar result holds for unoriented chain
complex, i.e. there is an isomorphism between the $\delta$-graded
chain complexes $CFL'(\cal{H},\partial)$ and $GC^-(\mathbb{G},\partial_{\mathbb{X}})$.

\section{Link categories}
In this section we introduce three link categories: disoriented links,
bipartite links, and bipartite disoriented links. 

\subsection{Category 1: disoriented links}\label{sec:lc:dl}
The idea of this
category comes from the Morse theory for links. In a disoriented link
cobordism, we keep track of the motion of the index zero and index
three critical points of the disoriented links.

\begin{defn}
  A \textbf{\textit{disoriented link}} is a link $L$ in a closed
  oriented three-manifold $Y$, together with two sets of
points $\mathbf{p}=\{p_1,\cdots,p_n\}$ and $\mathbf{q}=\{q_1,\cdots,q_n\}$ on $L$ such that $p_i$ and
$q_j$ appear alternatively on each component of $L$. These points
cut the link $L$ into $2n$-arcs $\mathbf{l}=\{l_1,\cdots,l_{2n}\}$,
which we orient from $q$ to $p$ such that:
\[\partial \mathbf{l} = \partial l_1+\cdots+\partial l_{2n}= 2
  (p_1+\cdots+p_{n})-2(q_1+\cdots+q_n).\]
We denote a disoriented link by
$\cal{L}=(L,\mathbf{p},\mathbf{q},\mathbf{l})$. We call the points
$\mathbf{p}$ and $\mathbf{q}$ the \textbf{\textit{dividing set}} of
 the disoriented link $\cal{L}$. See Figure \ref{fig:intro-01} for an
 example of disoriented link.

\end{defn}

\begin{rem}
  
The idea of disoriented link comes from the construction of a Morse
function $f$
compatible with a given oriented link $(Y,L)$. We think of $L$ as a
union of trajectories 
$\mathbf{l}$ and forget the $w$ markings and $z$ markings (hence
forget the orientation) of $L$. The points $\mathbf{p}$ play the role of index zero
critical points of $f$. The points $\mathbf{q}$ play the role of index
three critical points of $f$. 

\end{rem}

\begin{defn}
  A \textit{\textbf{surface with divides}} is an embedding: 
\[\cal{A}:(A,\partial A)\hookrightarrow (F,\partial F)\]
such that,
\begin{itemize}
\item The pair $(A,\partial A)$ is a compact, oriented one-manifold. 
\item The pair $(F,\partial F)$ is a compact surface and does not need to be
  orientable. 
\item The components of $F\backslash A=\{F_1,\cdots,F_k\}$ are
  compact oriented surfaces with orientation induced from the
  one-manifold $(A,\partial A)$. 
\end{itemize}
\end{defn}

\begin{rem}
Our definition of surface with divides is a generalization of that is in
\cite{Juhasz2016}. In \cite{Juhasz2016}, for an oriented
link cobordism, the 
orientation of the surface $F$ induces the orientation on each piece
$F_i$. For $F_i$ with $w$ basepoints, the orientation agrees with the
orienation induced by the oriented one manifold $\cal{A}$. For $F_i$
with $z$ basepoints, the orientation agrees with the opposite of the
orienation induced by the oriented one manifold $\cal{A}$.
\end{rem}

\begin{defn} Suppose we have a disoriented link
  $\cal{L}^0=(L^0,\mathbf{p}^0,\mathbf{q}^0,\mathbf{l}^0)$ in a closed oriented three-manifold
  $Y^0$, and a disoriented link
  $\cal{L}^1=(L^1,\mathbf{p}^1,\mathbf{q}^1,\mathbf{l}^1)$ in a closed oriented three-manifold $Y^1$.
  A \textit{\textbf{disoriented link cobordism}} from the disoriented link
  $(Y^0,\cal{L}^0)$ to $(Y^1,\cal{L}^1)$
is a triple $\mathfrak{W}=(\cal{W,F,A})$ such that:
\begin{itemize}
\item The pair $\cal{W}=(W,\partial W)$ is an oriented cobordism from 
  $Y^0$ to $Y^1$.
\item  The map $\cal{F}:(F,\partial F)\rightarrow (W,\partial
W)$ is a smooth embedding of surface $(F,\partial F)$.
\item The embedding $\cal A:(A,\partial A)\hookrightarrow (F,\partial
  F)$ is a surface with divides.
\item The boundary $\partial A$ is the union of points
  $\mathbf{q}^0-\mathbf{p}^0+\mathbf{p}^1-\mathbf{q}^1$. Furthermore, the intersection $(\partial A\cap Y^0)$ is
  $\mathbf{q}^0-\mathbf{p}^0$ and the intersection $(\partial A\cap Y^1)$ is
  $\mathbf{p}^1-\mathbf{q}^1$.  

\item The boundary $\partial (F\backslash A) =\partial
  F_1+\cdots+\partial F_n$ is 
  $-\mathbf{l}^0+\mathbf{l}^1+ 2A $. Furthermore, the intersection $\partial
  (F\backslash A)\cap -Y^0$ is $-\mathbf{l}^0$ and  $\partial
  (F\backslash A)\cap Y^1$ is $\mathbf{l}^1$.
\end{itemize}
\end{defn}

\begin{rem}
  The orientation for the components of $F\backslash A$ is unique and
  determined by the disoriented link $\cal{L}^0$ 
  and $\cal{L}^1$. If $F$ is orientable, then a disoriented
  link cobordism is equivalent to the \textit{decorated link cobordism}
  in \cite{Juhasz2009} and \cite{Zemke2016a}.
\end{rem}

\begin{rem}\label{thm:khovanov}
  In \cite{Clark2009a}, Clark, Morrison and Walker introduced
  disorientation in the link cobordism to make Khovanov homology
  functional with respect to link cobordisms. The disoriented links in our
  definition correpond to a special case of the `disoriented circle'
  with the `disorientation number' 
  equal to zero, see \cite[Lemma 4.4]{Clark2009a}. 
\end{rem}

\begin{exmp}
  In Figure \ref{fig:intro-02}, we show a disoriented link cobordism between two disoriented
  link $L^0$ and $L^1$. The dividing set $A$ (bold blue line) cuts the 
  surface $F$, which is non-orientable in this example, into 
  three components $F_1,F_2,F_3$. Each surface $F_i$ is an
  oriented surface with orientation compatible with the orientation of
  $A$ (marked as the blue arrow). In the three-manifold $Y^1$, the
  orientation of $F_i$ agrees with the orientation of the oriented arcs
  $\mathbf{l}^1$ in $Y^1$. In the three-manifold $Y^0$, the
  orientation of $F_i$ is the opposite of the orientation of
  $\mathbf{l}^0$ in $Y^0$.  
\end{exmp}

\begin{defn}\label{sec:elem-dior-link}
  Suppose the four manifold $W$ is a product $Y\times I$.
  We call a disoriented link cobordism $(\cal{W},\cal{F},\cal{A})$
  from $(Y,\cal{L}^0)$ to $(Y,\cal{L}^1)$
  \textit{\textbf{regular}}, if there exists a projection map
  $\pi:Y\times I\rightarrow  I$ such that:

\begin{itemize}
  \item  The map $\pi|_F$ and $\pi|_{A}$ is a Morse function.
  \item  If $a$ is a regular value for both $\pi|_F$ and $\pi_A$, then
    the triple
    $(\cal{W},\cal{F},\cal{A})\cap \pi^{-1}([0,a])$ is a disoriented
    link cobordism from the disoriented link $(Y,\cal{L}^0)$ to a
    disoriented link $(Y,\cal{L}^a)$.
  \item The index one critical points of $\pi|_F$ do not lie on $A$.
  \item  There is a sequence $\{a_1,\cdots,a_m\}$ of regular values for
  both $\pi|_F$ and 
  $\pi|_A$ such that, there is only one critical points of $\pi|_A$ or
  index one critical point of $\pi|_F$ with its value in
  $(a_i,a_{i+1})$.
\end{itemize}
\end{defn}

\begin{rem}
  The condition (2) in Definition \ref{sec:elem-dior-link} guarantees
  the index two/zero critical points of $\pi|_F$ is included in the
  index one/zero critical points of $\pi|_A$.
\end{rem}

We can decompose a regular cobordism into a composition of four types of
\textit{\textbf{elementary disoriented link cobordisms}}:
\begin{enumerate}
\item \textit{\textbf{Isotopy}} of disoriented links.
\item  \textit{\textbf{Band move}} (saddle move) of disoriented link.
\item \textit{\textbf{Disk-stabilization/destabilization}} of disoriented link.
\item  \textit{\textbf{Quasi-stabilization/destabilization}} of disoriented link.

\end{enumerate}

\begin{rem}
By definition, it is easy to see:
\begin{enumerate}
\item An isotopy contains no critical points of $\pi|_{F}$ or
  $\pi|_A$.
\item A band move contains an
index one critical point of $\pi|_F$.
\item  A
disk stabilization/destabilization contains an index two/zero critical point
of $\pi|_F$.
\item  A quasi-stabilization/destabilization contains an index
one/zero critical points of $\pi|_A$.

\end{enumerate} 
\end{rem}

\begin{figure}
  \centering
  \includegraphics[scale=0.6]{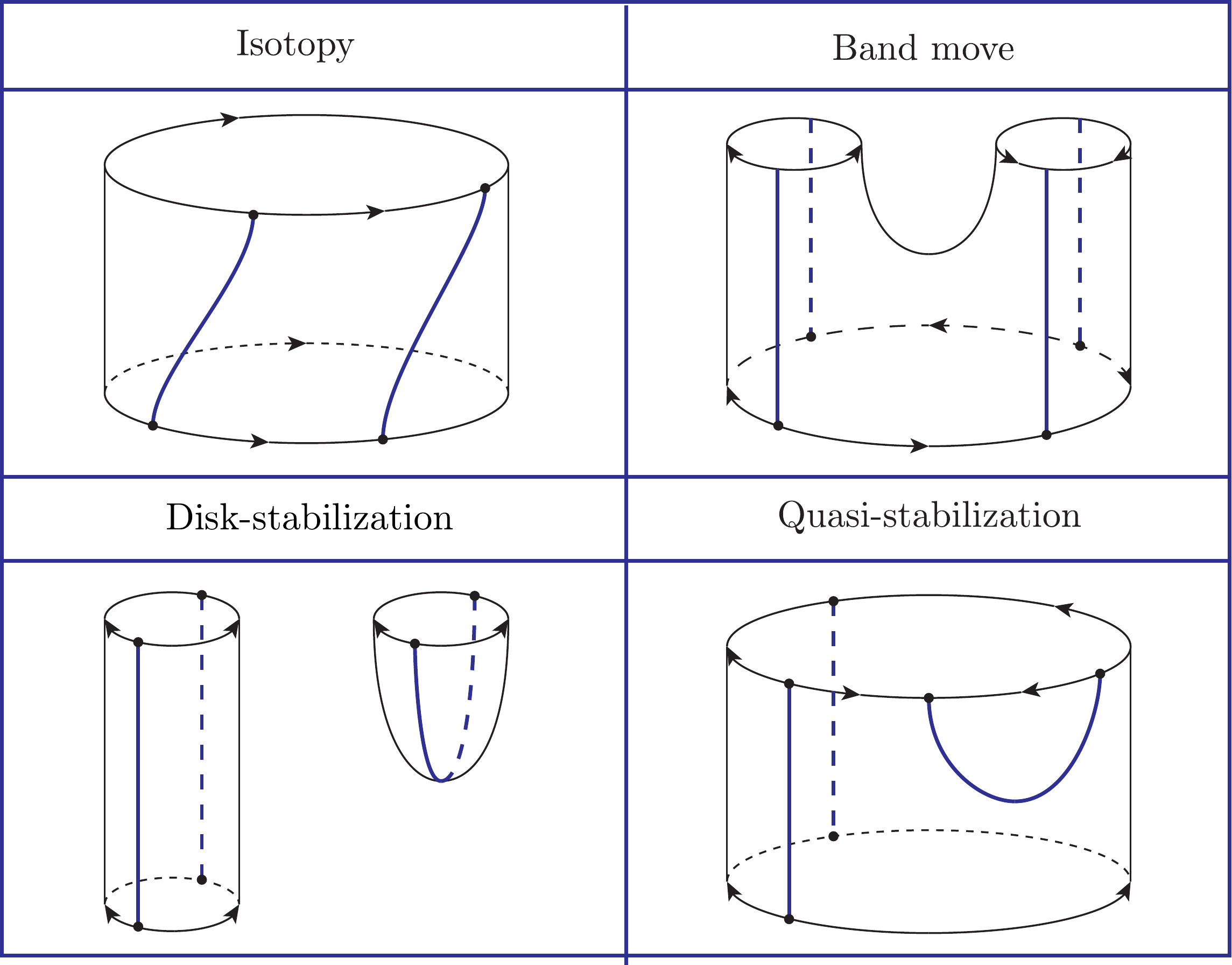}
  \caption{Four types of elementary cobordism.}
  \label{fig:elecob}
\end{figure}

\subsection{Category 2: bipartite links}
\label{sec:categ-2:-bipart}
 In a bipartite link
cobordism, we keep track of the motion of the basepoints of the links.

\begin{defn}
  A \textit{\textbf{bipartite link}} is a link $L$ in a closed
  oriented three-manifold $Y$, together with
  $2n$-basepoints $\mathbf{O}=\{o_1,\cdots,o_{2n}\}$ and two
  $n$-tuples of disjoint embedded arcs
  $L_{\alpha}=\{L_{\alpha,1},\cdots,L_{\alpha,n}\}$ 
  and $L_{\beta}=\{L_{\beta,1},\cdots,L_{\beta,n}\}$ on $L$, such that:
\begin{itemize}
\item The ends $\partial
  L_\alpha=\partial (L_{\alpha,1}\cup\cdots\cup L_{\alpha,n})$ are
  identified with the ends $\partial
  L_\beta=\partial (L_{\beta,1}\cup\cdots\cup
  L_{\beta,n})$. Furthermore, the ends $\partial L_\alpha=\partial
  L_\beta$ are exactly the basepoints $\mathbf{O}$ on $L$. 
\item The union of the two $n$-tuples of arcs $L_\alpha\cup L_\beta$
  is the link $L$.
\end{itemize}
We denote a bipartite link by
$L_{\alpha\beta}=(L,L_\alpha,L_\beta,\mathbf{O})$. See Figure
\ref{fig:bpl} for an example of bipartite link.
\end{defn}

\begin{figure}
  \centering
  \includegraphics[scale=0.8]{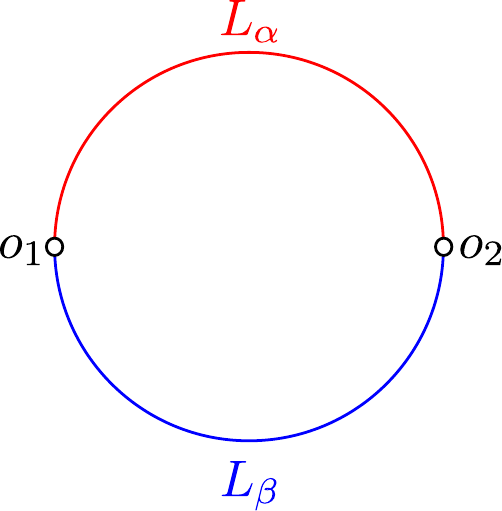}
  \caption{Bipartite link}
  \label{fig:bpl}
\end{figure}

\begin{rem}
  Let $U_\alpha\cup_{\Sigma} U_\beta$ be a Heegaard splitting of a
  three-manifold $Y$. Here $U_\alpha$ and $U_\beta$ are two
  handlebodies, $\Sigma$ is a Heegaard surface. Suppose $L$ is a link
  in $Y$ and intersects $\Sigma$ transversely. Moreover, suppose the
  intersections of $L$ and $U_\alpha$ (or $U_\beta$ resp.) bound compressing
  disks to $\Sigma$. The $2n$-basepoints
  $\mathbf{O}$ play the role of the intersections $L\cap \Sigma$. The $n$-tuple of disjoint embedded arcs $L_\alpha$
  (or $L_\beta$ resp.) play the role of the intersection
  $U_\alpha\cap L$ (or $U_\beta\cap L$ resp.). We do not color the
  basepoints into $w$'s and $z$'s.
\end{rem}

\begin{defn}
  A \textit{\textbf{bipartite link cobordism}} from a bipartite link
  $L_{\alpha\beta}^0$ in $Y^0$ to a bipartite link
  $L_{\alpha\beta}^1$ in $Y^1$ is a quintuple
  $(\cal{W},\cal{F},F_\alpha,F_\beta,\cal{A}_\Sigma)$, such that:
\begin{itemize}
\item The manifold $\cal{W}=(W,\partial W)$ is an oriented cobordism from
  three-manifold $Y^0$ to $Y^1$.
\item The map $\cal{F}:(F,\partial F)\rightarrow (W,\partial W)$ is an
  embedding of a compact surface $(F,\partial F)$ in $(W,\partial W)$.
\item The map $\cal{A}_\Sigma:(A_\Sigma,\partial A_\Sigma)\rightarrow
  (F,\partial F)$ is an embedding of a one-manifold
  $(A_\Sigma,\partial A_\Sigma)$ in $(F,\partial F)$.
\item The surface $F$ is decomposed along $A_\Sigma$ into two
  compact surfaces $F_\alpha$ and $F_\beta$. One side of $A_\Sigma$ on $F$ is
  belong to the interior of $F_\alpha$, the other side is belong to
  the interior of $F_\beta$.
\item The bipartite link $(Y^i,L^i,L^i_\alpha,L^i_\beta,\mathbf{O}^i)$
  is identified with $(Y^i,Y^i\cap\partial F_\alpha,Y^i\cap\partial
  F_\beta,Y^i\cap A_\Sigma)$, where $i=0,1$.
\end{itemize}

\end{defn}

\begin{exmp}
 Figure \ref{fig:bpcob} shows a bipartite link
 cobordism from a bipartite link
 $L_{\alpha\beta}^0$ to a bipartite link
 $L_{\alpha\beta}^1$. The red curves forming
 $A_\Sigma$ cut
 the surface $F$ into four 
 components. Two of the components are belong to $F_\alpha$, the other
 two are belong to $F_\beta$. One of the component of $F_\alpha$ is
 non-orientable. 
\end{exmp}

\begin{figure}
  \centering
  \includegraphics[scale=1]{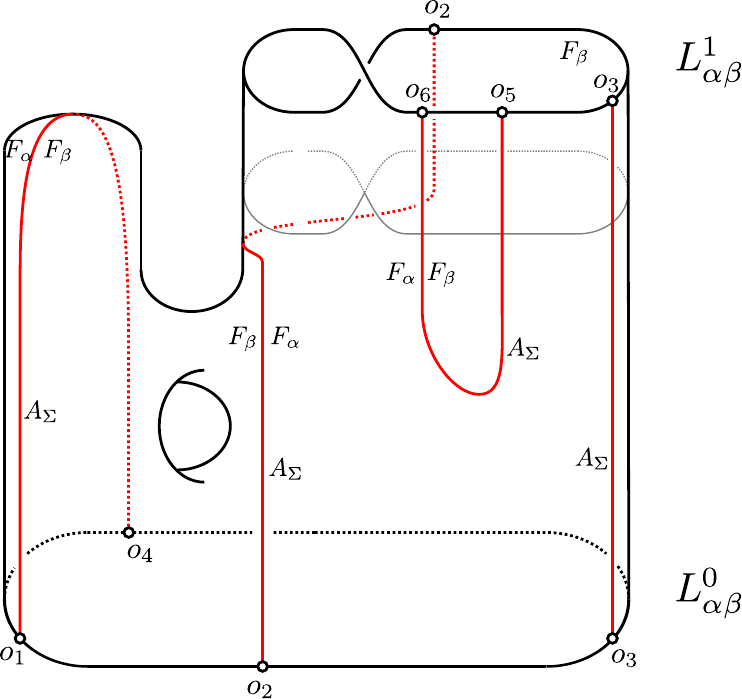}
  \caption{Bipartite link cobordism}
  \label{fig:bpcob}
\end{figure}

We say that a bipartite link cobordism
$(\cal{W},\cal{F},F_\alpha,F_\beta,\cal{A}_\Sigma)$ is
\textit{\textbf{regular}}, if 
it satisfies the same
condition as in Definition \ref{sec:elem-dior-link} with
$\cal{A}_\Sigma$ plays the role of $\cal{A}$ and bipartite links
$L^i_{\alpha\beta}$ play the role of disoriented links
$\cal{L}^i$. Similarly, we can still classify the elementary
bipartite link cobordisms into four types: isotopies, band moves, disk-stabilizations/destabilizations,
quasi-stabilizations/destabilizations.

 Furthermore, we say that
a critical point $p$ of $\pi|_{A_\Sigma}$ or a saddle point of $\pi|_F$ is of \textit{\textbf{$\alpha$-type}}
if $(N_p\backslash p)\cap \pi|_F^{-1}(\pi(p))\subset F_\alpha$,
otherwise; we say that the critical point $p$ is of
\textit{\textbf{$\beta$-type}}. Here $N_p$ is a small neighborhood of
$p$ in $W$.

Suppose $\pi^{-1}[-a,a]$ contains only a saddle critical point
$p$. Without loss of generality, suppose $p$ lies in a component
$F^i_\alpha$ of $F_\alpha$. We call $p$ of \textit{\textbf{Type
    I}}, if $\chi(F^i_\alpha\cap \pi^{-1}[c-\epsilon,c+\epsilon])=0$. If
$\chi(F^{i}_\alpha\cap \pi^{-1}[c-\epsilon,c+\epsilon])=1$, we call $p$ of
\textit{\textbf{Type II}}, as shown in Figure \ref{fig:type12}.

\begin{figure}
  \centering
  \includegraphics[scale=0.5]{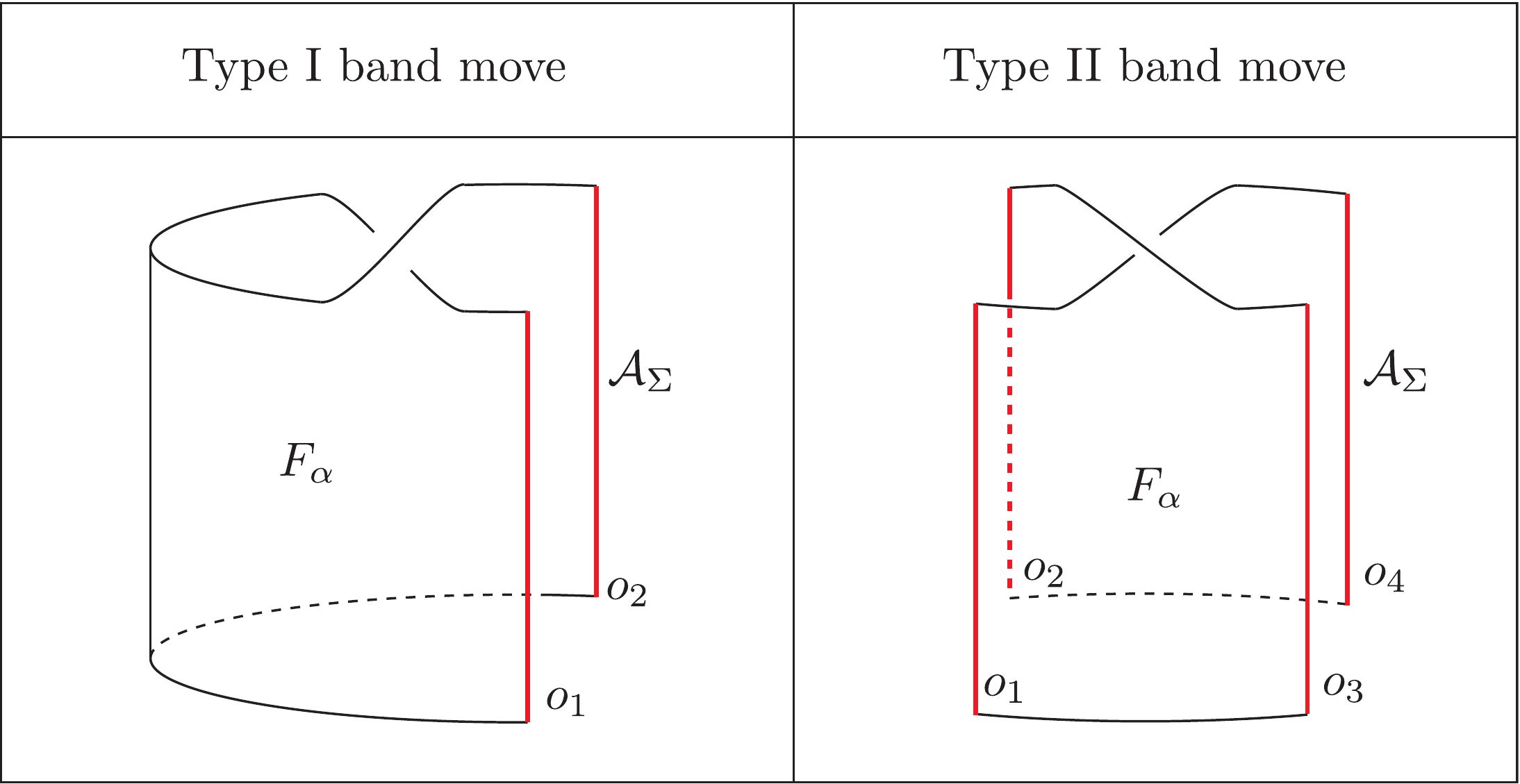}
  \caption{}
  \label{fig:type12}
\end{figure}

\begin{rem}
  Here, the surface $F_\alpha$ and $F_\beta$ may have some orientable
  components and non-orientable components. Consequently, the one-manifold
  $A_\Sigma$ has no canonical orientation. Notice that the bipartite
  link cobordism is different from the 
  decorated link cobordism 
  \cite{Juhasz2009}.
\end{rem}

\subsection{Category 3: bipartite disoriented links}
\label{sec:categ-3:-bipart}
 Bipartite disoriented links combine the data from bipartite links and disoriented links. In a bipartite disoriented link cobordism, we keep track of the motion of index zero/three critical points and the basepoints. 

\begin{defn}\label{sec:categ-point-disor}
  A \textit{\textbf{bipartite disoriented link}} $(\cal{L},\mathbf{O})$ is a disoriented link $\cal{L}=(L,\mathbf{p},\mathbf{q},\mathbf{l})$
together with a set of basepoints $\mathbf{O}$, consisting of a unique basepoint $o_i$ on the interior of each
oriented arc
$l_i\in\mathbf{l}$. See Figure \ref{fig:bdl} for an example of bipartite
disoriented link.
\end{defn}
\begin{figure}
  \centering
  \includegraphics[scale=0.8]{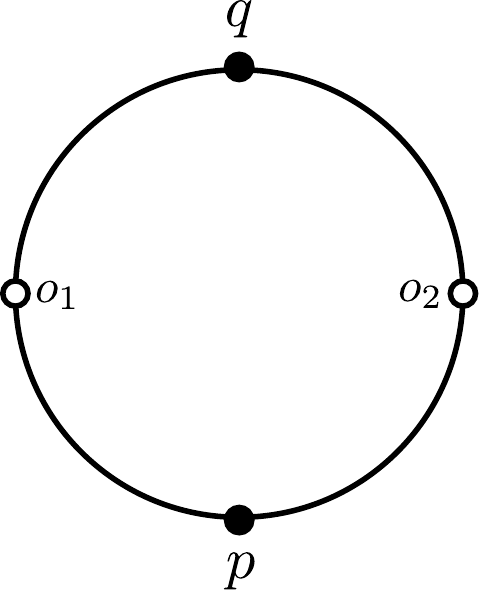}
  \caption{Bipartite disoriented link.}
  \label{fig:bdl}
\end{figure}
\begin{rem}\label{sec:category-3:-pointed-1}
A bipartite disoriented link  $(\cal{L},\mathbf{O})$ determines a
bipartite link as follows. The
basepoints $\mathbf{O}=\{o_1,\cdots,o_{2n}\}$ cut the link $L$ into
$2n$-arcs. Let $L_\alpha$ be the collection of arcs which contain 
$\mathbf{p}$-points and $L_{\beta}$ be the collection of arcs which
contain $\mathbf{q}$-points. Then $L_{\alpha\beta}=(L,L_\alpha,L_\beta,\mathbf{O})$
is a bipartite link. 
\end{rem}

\begin{defn}\label{sec:category-3:-pointed}
  A \textit{\textbf{bipartite disoriented link cobordism}} from
  bipartite disoriented link $(\cal{L}^0,\mathbf{O}^0)$ in $Y^0$ to
  $(\cal{L}^1,\mathbf{O}^1)$ in $Y^1$ is a sextuple
  $\mathbb{W}=(\cal{W},\cal{F},F_\alpha,F_\beta,\cal{A},\cal{A}_\Sigma)$ such that:
\begin{itemize}
\item The triple $(\cal{W},\cal{F},\cal{A})$ is a disoriented link
  cobordism from the disoriented link $(Y^0,\cal{L}^0)$ to $(Y^1,\cal{L}^1)$.
\item The quintuple $(\cal{W},\cal{F},F_\alpha,F_\beta,\cal{A}_\Sigma)$ is
  a bipartite
  link cobordism from bipartite link $(Y^0,L^0_{\alpha\beta})$ to
  $(Y^1,L^1_{\alpha\beta})$. Here $L^i_{\alpha\beta}$ is the bipartite
  link determined by the bipartite disoriented link
  $(\cal{L}^i,\mathbf{O}^i)$, $i=0,1$.
\item The intersection $F_\alpha \cap Y_0$ is a union of arcs
  containing all the $\mathbf{p}$ points. 
\item The intersection $F_\beta \cap Y_0$ is a union of arcs
  containing all the $\mathbf{q}$ points.  
\end{itemize}
Furthermore, if $W=Y\times I$, we call a bipartite disoriented link
cobordism $\mathbb{W}=(\cal{W},\cal{F},F_\alpha,F_\beta,\cal{A},\cal{A}_\Sigma)$ \textit{\textbf{regular}}, if it
satisfies the following condition:
\begin{itemize}
\item The triple $(\cal{W},\cal{F},\cal{A})$ and the quintuple
  $(\cal{W},\cal{F},F_\alpha,F_\beta,\cal{A}_\Sigma)$ are regular.
\item The critical points of $\pi|_{A}$ are exactly the critical points
  of $\pi|_{A_\Sigma}$.
\item The one-manifold $A$ intersect $A_\Sigma$ transversely. The
  intersection points $A\cap A_{\Sigma}$ are exactly the critical
  points of $\pi|_A$
\end{itemize}
 
\end{defn}

\begin{exmp} 
Figure \ref{fig:decdiscob} shows an example of a 
 bipartite disoriented link cobordism $(\cal{W},\cal{F},F_\alpha,F_\beta,\cal{A},\cal{A}_\Sigma)$. The blue curves with arrows are the
  components of 
  oriented one-manifold $A$. The red curves without orientation are
  the components of $A_\Sigma$. Clearly, the triple $(\cal{W},\cal{F},\cal{A})$
  is the disoriented link cobordism shown in Figure
  \ref{fig:intro-02}. The quintuple $(\cal{W},\cal{F},F_\alpha,F_\beta,\cal{A}_\Sigma)$ is the
  bipartite link cobordism shown in Figure \ref{fig:bpcob}.
  
\end{exmp}

\begin{figure}
  \centering
  \includegraphics[scale=1]{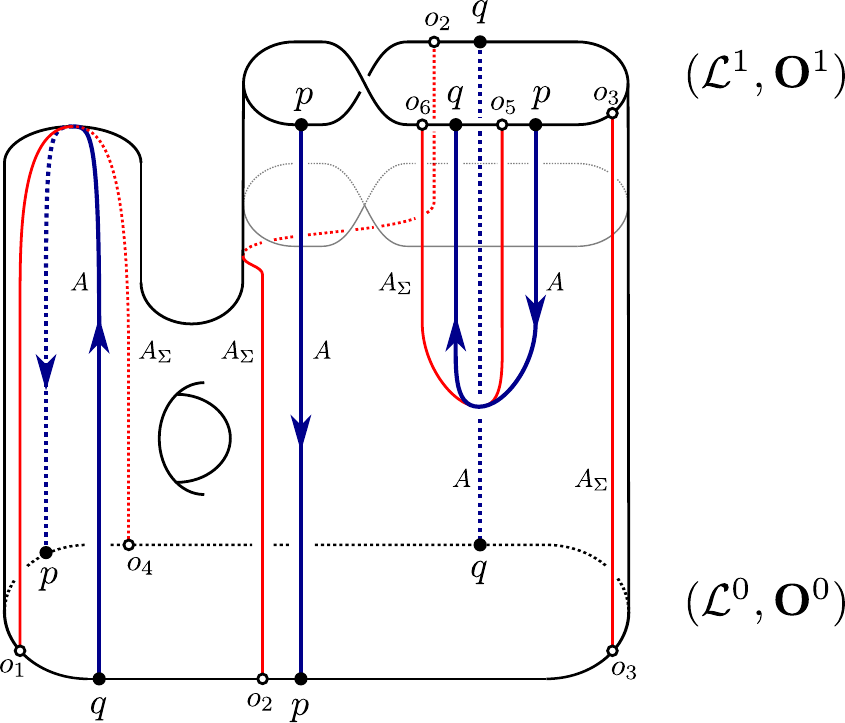}
  \caption{Bipartite disoriented link cobordism.}
  \label{fig:decdiscob}
\end{figure}

If a decorated disoriented link cobordism is regular, we can decompose
it into four types of elementary cobordism as regular disoriented link
cobordisms. Suppose  $(\cal{W},\cal{F},F_\alpha,F_\beta,\cal{A},\cal{A}_\Sigma)$ is
elementary.  
Furthermore, as we have defined the $\alpha$-type and
$\beta$-type of the critical points of $\pi|_A$ or saddle points of
$\pi|_F$, we have the following:

\begin{itemize}
\item If the elementary cobordism contains a saddle point
  $p$ of $\pi|_F$ in $F_\alpha$, we call this elementary cobordism an
\textit{\textbf{$\alpha$-band move}} (or $\alpha$-saddle move),
otherwise we call it
a \textit{\textbf{$\beta$-band move}} (or $\beta$-saddle move). See
Figure \ref{fig:absaddle}. 
\item  If the elementary cobordism
 contains a critical point
  $p$ of $\pi|_F$ and satisfies  $(N_p\backslash p)\cap
  \pi|_F^{-1}(\pi(p))\subset F_\alpha$, we call it an
  \textit{\textbf{$\alpha$-quasi-stabilization/destabilization}}. Here $N_p$ is a small neighborhood of $p$
in $W$. If
the critical point $p$ of $\pi_{A_{\Sigma}}$ satisfying $(N_p\backslash p)\cap \pi|_F^{-1}(\pi(p))\subset F_\beta$, we call
  it a
\textit{\textbf{$\beta$-quasi-stabilization/destabilization}}. See
Figure \ref{fig:abquasi}. 
\end{itemize}

\begin{figure}
  \centering
  \includegraphics[scale=0.6]{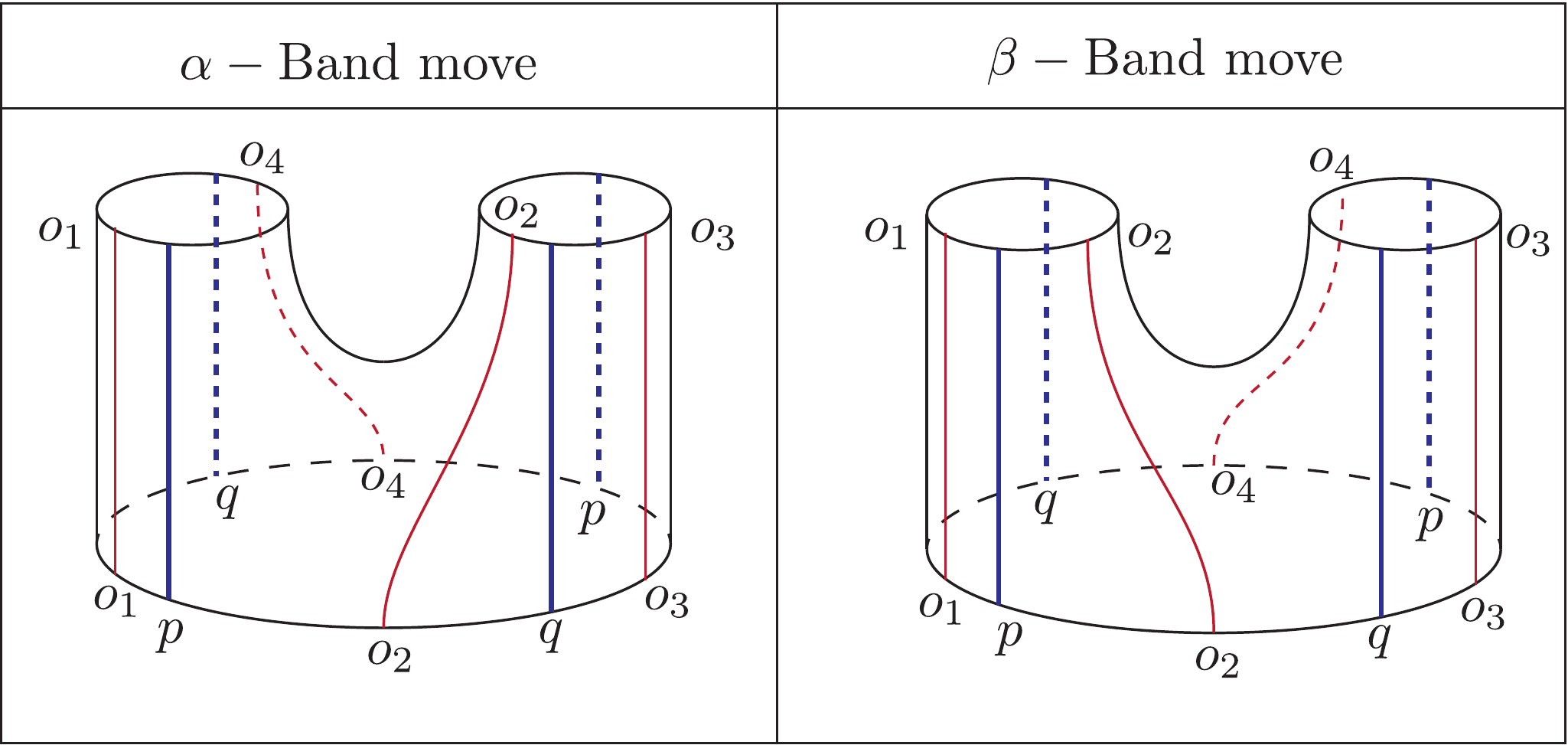}
  \caption{Two types of band move.}
  \label{fig:absaddle}
\end{figure}

\begin{figure}
  \centering
  \includegraphics[scale=0.6]{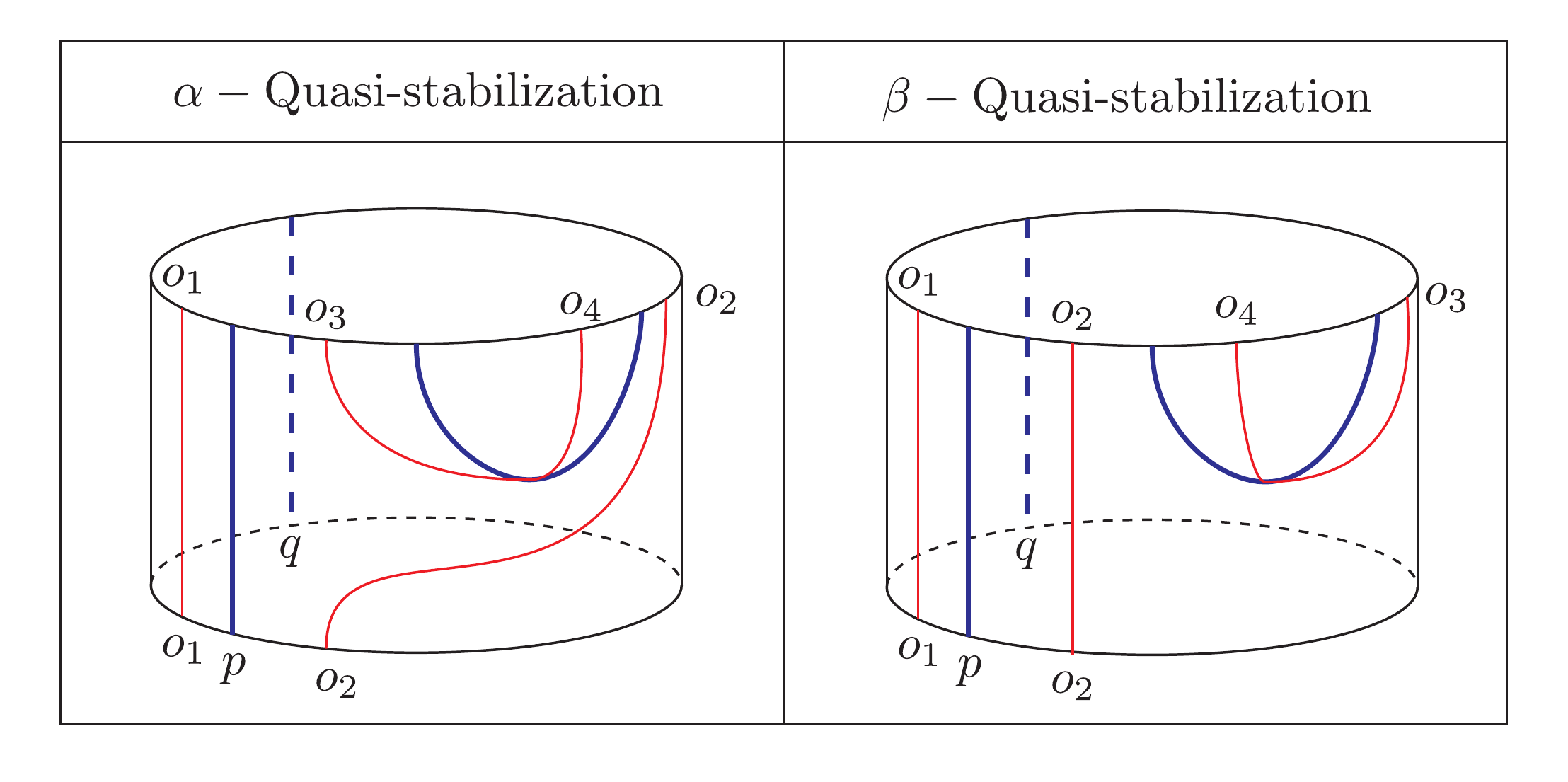}
  \caption{Two types of quasi-stabilization}
  \label{fig:abquasi}
\end{figure}

\begin{rem}
  We can also classify band moves of bipartite disoriented link into
  Type I or Type II.
\end{rem}

\subsection{Coloring and the relations between the three link
  categories}
\label{sec:relat-betw-three}

\begin{defn}
Let $L_{\alpha\beta}$ be a bipartite link. A
\textit{\textbf{coloring}} of basepoints is a map
\[\mathfrak{P}:\mathbf{O}=\{o_1,\cdots,o_{2n}\}\rightarrow \{\pm 1\},\] 
such that the cardinality $|\mathfrak{P}^{-1}(+1)|$ is equal to
$|\mathfrak{P}^{-1}(-1)|$. Furthermore, we say that a coloring is
\textit{\textbf{alternating}}, if and only if, for any pairs of adjacent
basepoints $(o,o')$, the coloring $\mathfrak{P}(o)$ is equal to
$-\mathfrak{P}(o')$. 
\end{defn}

 Let $\mathbf{w}=\{w_1,\cdots,w_n\}$ be the set
$\mathfrak{P}^{-1}(+1)$, $\mathbf{z}=\{z_1,\cdots,z_n\}$ be the set
$\mathfrak{P}^{-1}(-1)$. We denote a bipartite link
together with a
coloring $\mathfrak{P}$ by $(L_{\alpha\beta},\mathfrak{P})$.

\begin{rem}\label{sec:coloring}
  Given a bipartite link $L_{\alpha\beta}$,
  there exists $2^{|L|}$ different alternating colorings,
  where $|L|$ is the number of link components of $L$. Given a bipartite link
  together with a alternating coloring, we can orient the arcs
  $L_{\alpha}$ from $z$ to $w$, and the arcs $L_{\beta}$ from $w$ to
  $z$. This assignment gives rise to a oriented link
  $L_{\mathfrak{P}}$. 

\end{rem}

For convenience, we denote by $\cal{C}_{DL}$ the category of
disoriented links, by $\cal{C}_{BL}$ the category of bipartite links,
and by $\cal{C}_{BDL}$ the category of bipartite disoriented
links. The relations of the three categories are shown below. 
\begin{displaymath}
    \xymatrix{
         & \mathcal{C}_{BDL} \ar[dl]_{F_B} \ar[dr]^{F_D}  &   \\
        \mathcal{C}_{DL} \ar@/_/@{.>}[ur]_{G_B}   &  & \mathcal{C}_{BL} \ar@/^/@{.>}[ul]^{G_D} }
\end{displaymath}

Let $\mathbb{W}$ be a bipartite disoriented link cobordism from
bipartite disoriented link $\mathbb{L}^0$ to $\mathbb{L}^1$. 
\begin{itemize}
\item The forgetful functor $F_B$ removes the basepoints on
  $\mathbb{L}^i$ and the one-manifold $\cal{A}_{\Sigma}$ on $W$. 
\item The forgetful functor $F_D$ removes the $\mathbf{p}$ and
  $\mathbf{q}$ points on $\mathbb{L}^i$ and the oriented one-manifold
  $\cal{A}$ on $W$. By definition \ref{sec:category-3:-pointed}, $F_D$ send $\mathbb{W}$ to a
  bipartite link cobordism
  $(\cal{W},F,F_{\alpha},F_{\beta},\cal{A}_{\Sigma})$ from
  $L_{\alpha\beta}^0=F_D(\mathbb{L}^0)$ to $L_{\alpha\beta}^1=F_D(\mathbb{L}^1)$.
\end{itemize}

The dotted arrow $G_B$ and $G_D$ are not functors, but represent the
processes of lifting objects and morphisms between the repective
category. The processes depend on some choices, as detailed below. 

For a disoriented link, $G_B$ add one basepoint to each of the
oriented arcs $\mathbf{l}={l_1,\cdots,l_{2n}}$.
To lift a disoriented link cobordism $\mathfrak{W}$, we perturb
$\mathfrak{W}$ to regular positoin and decompose
$\mathfrak{W}$ into a composition of elementary cobordisms. Among the
four types of elementary cobordisms, we are particularly interested in
quasi-stabilizatios and band moves. For isotopies and disk-stabilizations/destabilizations, the lifting is unique. 
A band move or quasi-stabilization/destabilization of disoriented
links can be lifted to a bipartite disoriented link cobordism of either $\alpha$-type or $\beta$-type. 
See Figure \ref{fig:ldl} (a) for the lifting of band move, and Figure
\ref{fig:ldl} (b) for
the lifting of quasi-stabilization. 

For a bipartite link, $G_D$ add one $\mathbf{p}$-point to each
component of $L_\alpha$ and one $\mathbf{q}$-point to each component
of $L_{\beta}$.
To lift a bipartite link cobordism, we also perturb
it to a regular positoin decompose it into
elementary cobordisms. Similarly, the lifting for isotopies and
disk-stabilization/destablization is unique. A type-I band move has a unique way to be
lifted, as shown in Figure \ref{fig:lbl} (a); a type-II band move has
two ways to be lifted, as shown in Figure \ref{fig:lbl} (b); a quasi-stabilization/destabilization has two ways to be
lifted, as shown in Figure \ref{fig:lbl} (c).

\begin{figure}
  \centering
  \includegraphics[scale=0.4]{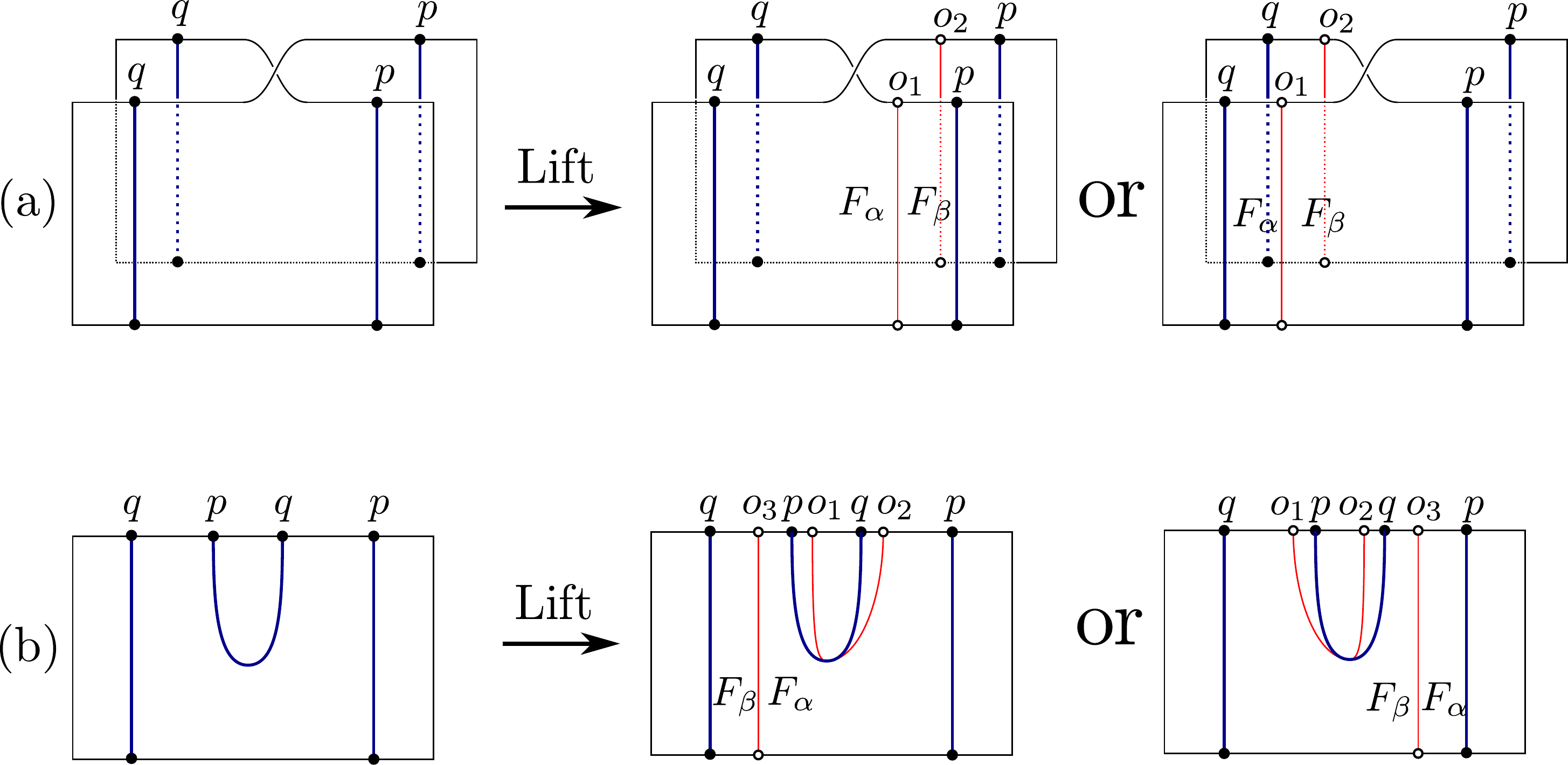}
  \caption{$G_B$-Lift disoriented link cobordism.  }
  \label{fig:ldl}
\end{figure}

\begin{figure}
  \centering
  \includegraphics[scale=0.4]{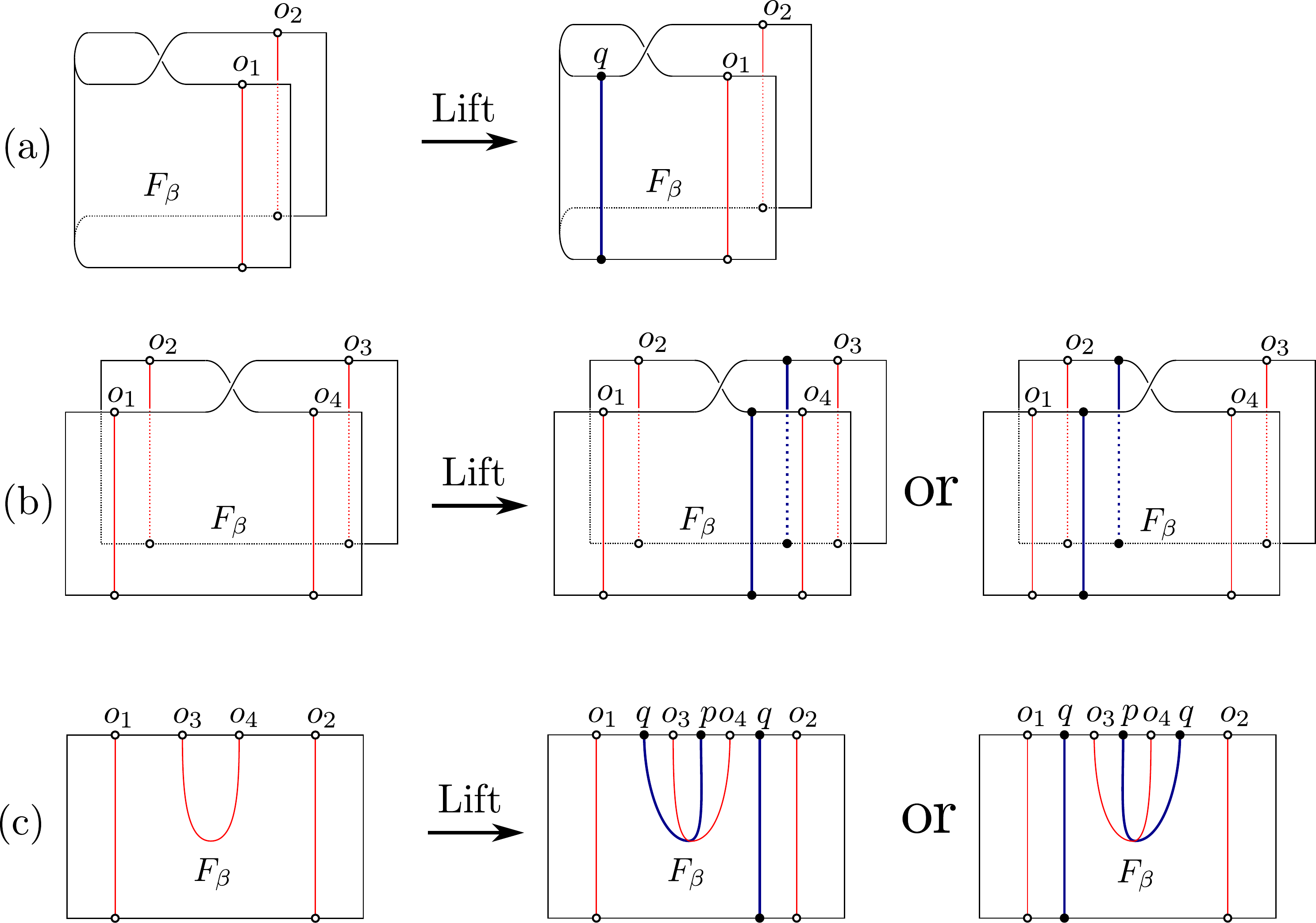}
  \caption{$G_D$-Lift bipartite link cobordism}
  \label{fig:lbl}
\end{figure}

\section{Heegaard diagrams for band moves}
\label{sec:unor-saddle-moves}

In this section, we will relate bipartite link
band moves to Heegaard triples. We focus our discussion on
$\beta$-band moves. Similar results hold for $\alpha$-band moves.

\subsection{Heegaard diagrams for bipartite links}
In this subsection we asssociate a Heegaard diagram to a bipartite link as follows.

Let $L_{\alpha\beta}=(L,L_\alpha,L_\beta,\mathbf{O})$ be a bipartite
link in a closed oriented three-manifold $Y$, where 
$L_\alpha=\{L_{\alpha,1},\cdots,L_{\alpha,n}\}$,
$L_\beta=\{L_{\beta,1},\cdots,L_{\beta,n}\}$,
$\mathbf{O}=\{o_1,\cdots,o_{2n}\}$. We say that a Heegaard Diagram
$H_{\alpha\beta}=(\Sigma,\bs{\alpha},\bs{\beta},\mathbf{O})$ is
\textit{\textbf {compatible}} with $(Y,L_{\alpha\beta})$, if:
\begin{itemize}
\item The surface $\Sigma$ is closed, oriented and embedded in $Y$. The
  genus of $\Sigma$ is $g$. The collection of curves
  $\bs{\alpha}$ is a $(g+n-1)$-tuple $\{\alpha_1,\cdots,\alpha_{g+n-1}\}$. The
  collection of curves $\bs{\beta}$ is a $(g+n-1)$-tuple
  $\{\beta_1,\cdots,\beta_{g+n-1}\}$. 
\item The three-manifold $Y$ is reprensented by
  $(\Sigma,\bs{\alpha},\bs{\beta})$, such that
\begin{align*}
  Y = \Sigma_g &\cup (\bigcup_{i=1}^{g+n-1}D_{\alpha_i})\cup(\bigcup_{k=1}^{n}B_{\alpha,k})\\
             &\cup
               (\bigcup_{j=1}^{g+n-1}D_{\beta_j})\cup(\bigcup_{l=1}^{n}B_{\beta,l}).
\end{align*}
Here $D_{\alpha_i}$ (or $D_{\beta_j}$ resp.) is a closed embedded disk with boundary identified with
$\alpha_i$ (or $\beta_j$ resp.) on $\Sigma$, and $B_{\alpha,k}$ (or
$B_{\beta,l}$ resp.) is a closed three-ball embedded in $Y$.
\item The pair $(L_{\alpha,k},\partial L_{\alpha,k})$ (or $(L_{\beta,l},\partial L_{\beta,l})$ resp.) is unknotted
  embedded in $(B_{\alpha,k},\Sigma\cap \partial B_{\alpha,k})$(or $(B_{\beta,l},\Sigma\cap \partial
  B_{\beta,l})$ resp.) for
  $k=1,\cdots,n$ ( or $l=1,\cdots,n$ resp.). 
\end{itemize}

For convenience, we denote by $U_\alpha$ the handlebody $\Sigma_g\cup
(\bigcup_{i=1}^{g+n-1}D_{\alpha_i})\cup(\bigcup_{k=1}^{n}B_{\alpha,k})$
and by $U_\beta$ the handlebody
$\Sigma_g\cup(\bigcup_{j=1}^{g+n-1}D_{\beta_j})\cup(\bigcup_{l=1}^{n}B_{\beta,l})$. 
We
say that the diagram $(\Sigma,\bs{\alpha},\mathbf{O})$ is \textit{\textbf{compatible}}
with the triple $(U_\alpha,L_\alpha,\mathbf{O})$, and the diagram
$(\Sigma,\bs{\beta},\mathbf{O})$ is \textit{\textbf{compatible}} with the triple $(U_{\beta},L_\beta,\mathbf{O})$.

\subsection{The existence of Heegaard triples subordinate to a band move} 

In \cite[Definition 6.2]{Zemke2016a}, Zemke defines a Heegaard
triple subordinate to an oriented band moves.
Building on his work, in this subsection,
we will construct a standard Heegaard Triple $\cal T$  subordinate to
a band move $B^\beta$ from a bipartite link $(L,L_\alpha,L_\beta,\mathbf{O})$ to
$(L(B^\beta),L_\alpha,L_\beta(B^\beta),\mathbf{O})$. Similar results hold for $\alpha$-band move.

Let $B^\beta:[-\epsilon,+\epsilon]\times [-1,1]\rightarrow Y$ be a
band embedded in $Y$ with two ends $B^\beta([-\epsilon,+\epsilon]\times\{\pm
1\})$ attached on $L_\beta$, such that:
\begin{itemize}
\item The intersection $B^\beta([-\epsilon,+\epsilon]\times[-1,1])\cap
  L$ is the ends of the band
  $B^{\beta}([-\epsilon,+\epsilon]\times\{\pm 1\})$.
\item The union $L_\beta(B^\beta)=(L_\beta\backslash
  B^\beta([-\epsilon,+\epsilon]\times\{\pm 1\}))\cup B^\beta(\{\pm
  \epsilon\}\times[-1,1])$ is still a collection of arcs with boundary
  identified with $\mathbf{O}$.
\end{itemize}
Then the quadruple $(L(B^\beta),L_\alpha,L_\beta(B^\beta),\mathbf{O})$
is still a a bipartite link in $Y$, where $L(B^\beta)$ is the link
$L_\alpha\cup L_\beta(B^\beta).$

\begin{defn}
  We say that a Heegaard triple $\cal
  T=(\Sigma,\boldsymbol{\alpha},\boldsymbol{\beta},\boldsymbol{\gamma},\mathbf{O})$
  is \textit{\textbf{subordinate}} to a $\beta$-band move from a
  bipartite link $L_{\alpha\beta}=(L,L_\alpha,L_\beta,\mathbf{O})$ to $L_{\alpha\gamma}=(L(B^\beta),L_\alpha,L_\beta(B^\beta),\mathbf{O})$, if it satisfies:
\begin{itemize}
\item The Heegaard diagram
  $H_{\alpha\beta}=(\Sigma,\boldsymbol{\alpha},\boldsymbol{\beta},\mathbf{O})$
  is compatible with the  bipartite link $(Y,L_{\alpha\beta})$. 
\item The Heegaard diagram
  $H_{\alpha\gamma}=(\Sigma,\boldsymbol{\alpha},\boldsymbol{\gamma},\mathbf{O})$
  is compatible with the bipartite link $(Y,L_{\alpha\gamma})$.
\end{itemize}
\end{defn}

Suppose the collection of curves $\bs{\alpha}$, (or $\bs{\beta}$,
$\bs{\gamma}$ resp.)  has
$n$-components, we say that the Heegaard Triple subordinate to a
$\beta$-band move is \textit{\textbf{standard}} if it satisfies the
following:
\begin{itemize}

\item The curves $\gamma_1,\cdots,\gamma_{n-1}$ are small Hamiltonian
  isotopies away from any basepoints of $\beta_1,\cdots,\beta_{n-1}$, with geometric intersection
  number $|\beta_i\cap\gamma_j|=2\delta_{ij}$.
\item The curve $\gamma_n$ is a Hamiltonian isotopy of $\beta_n$ with
  intersection number $|\beta_n\cap\gamma_n|=2$.
\item There exists a disk region $D\subset \Sigma$ contains two basepoints $o$ and $o'$
  such that
  $D\cap (\bs{\beta}\cup\bs{\gamma})=D\cap (\beta_n\cup\gamma_n)$ and
  $|D\cap\beta_n\cap\gamma_n|=2$. An example is shown in Figure \ref{fig:st11}.
\end{itemize}

\begin{figure}
  \centering
  \includegraphics[scale=0.6]{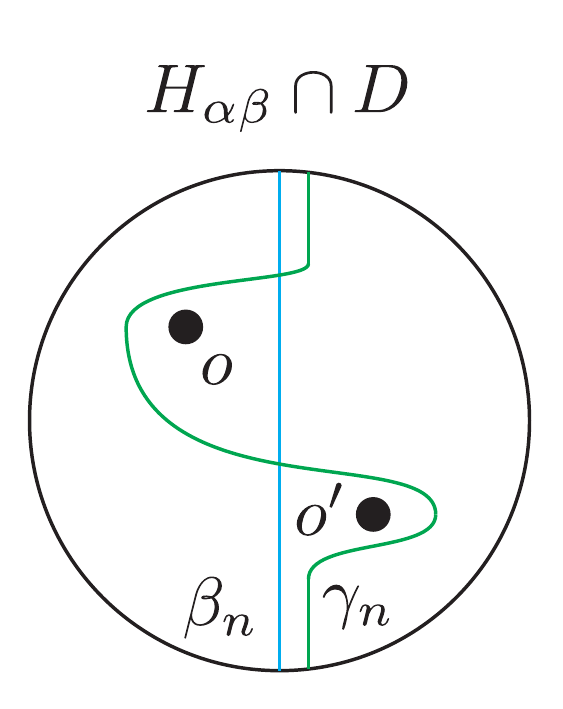}
  \caption{}
  \label{fig:st11}
\end{figure}

\begin{thm}\label{sec:exist-heeg-triple}
Let $(\cal{W},\cal{F},F_\alpha,F_\beta,\cal{A}_\Sigma)$ be a $\beta$-band move
  from a bipartite link $(L,L_\alpha,L_\beta,\mathbf{O})$ to a bipartite link
  $(L(B^\beta),L_\alpha,L_\beta(B^\beta),\mathbf{O})$. Here both
  bipartite link are in same three-manifold $Y$. 
 Starting from a Heegaard diagram $H_{\alpha\beta}$ of
 $(Y,L,L_\alpha,L_\beta,\mathbf{O})$, after a sequence of stabilization/destabilization 
  and handleslides without crossing any basepoints, we
  can construct a Heegaard diagram 
  $H'_{\alpha\beta}=(\Sigma',\bs{\alpha}',\bs{\beta}',\mathbf{O})$
  compatible with $(L,L_\alpha,L_\beta,\mathbf{O})$ together a disk
  region $D$ on 
  $\Sigma'$, such that the local diagram $D\cap H'_{\alpha\beta}$
  is shown in either of the two top figures in  Figure \ref{fig:st-10}.  

Furthermore, after a Hamiltonian perturbation of
  $\beta_n$ across two basepoints $o$ and $o'$, and small
  Hamiltonian perturbation
  of $\beta_1,\cdots,\beta_{n-1}$ without crossing any basepoints, we
  can get a standard Heegaard triple $\cal T$ subordinate to the
  $\beta$-band move, as shown in either of the two bottom figures in  Figure \ref{fig:st-10}. 
\end{thm}

\begin{figure}
  \centering
  \includegraphics[scale=0.6]{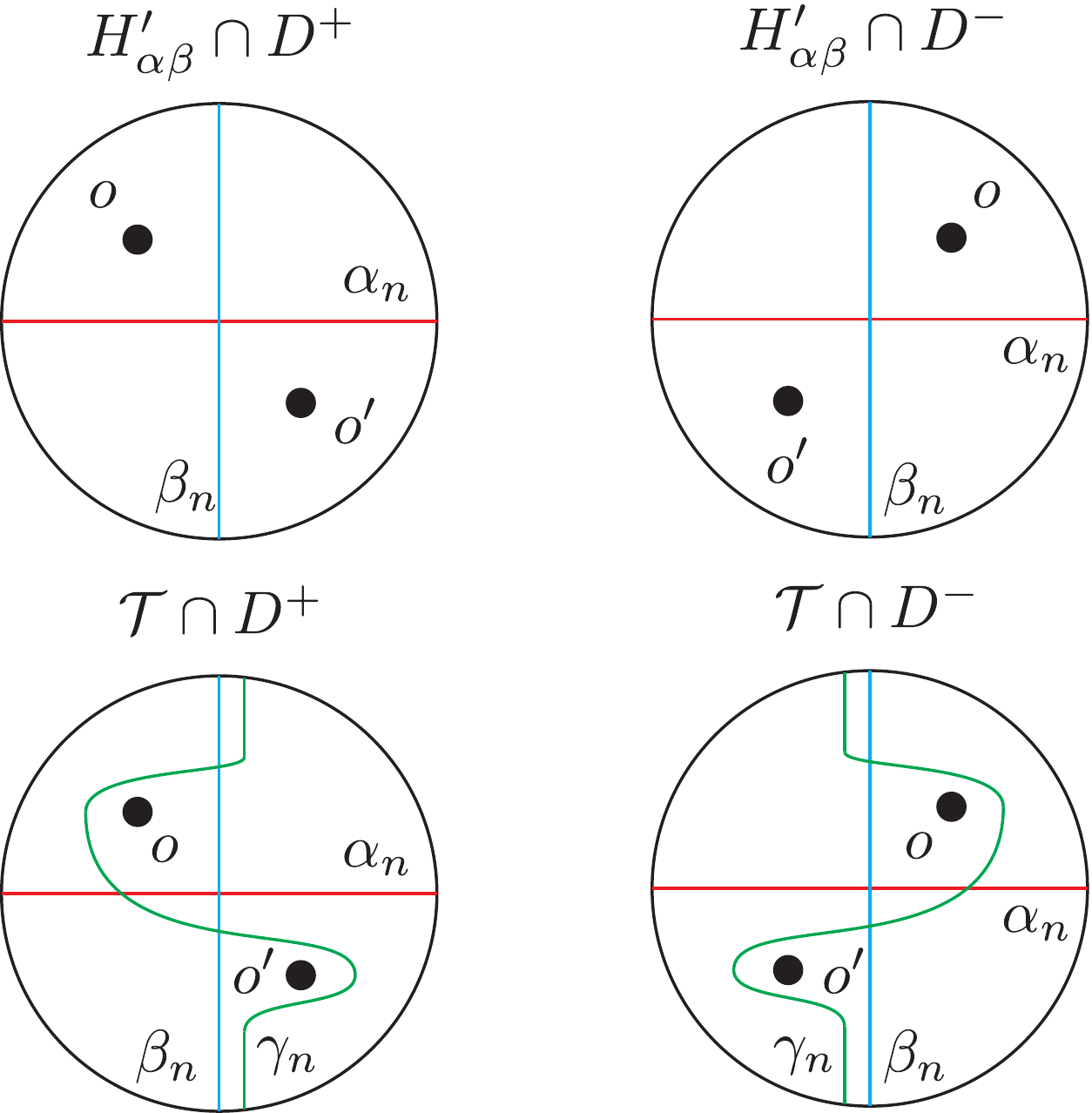}
  \caption{Local picture of the Heegaard triple subordinate to a $\beta$-band move. }
  \label{fig:st-10}
\end{figure}

\begin{proof}

The idea of the proof is to implant the data of the band $B^\beta$
into a Heegaard Diagram compatible with the bipartite link $L_{\alpha\beta}=(L,L_\alpha,L_\beta,\mathbf{O})$. 

 We lift
the bipartite link $(Y,L_{\alpha\beta})$ to a bipartite disoriented link
$(Y,\cal{L},\mathbf{O})$. Let $f$ be a Morse function compatible with  $(Y,\cal{L},\mathbf{O})$, such that its
corresponding Heegaard diagram is $H_{\alpha\beta}$. Then we have a
Heegaard decomposition $U_\alpha\cup_\Sigma U_\beta$ of $Y$ with
respect to $f$.
Without loss of generality, we assume the band 
$B^\beta$ lies in the $\beta$-handlebody $U_\beta$ and the ends of
$B^\beta$ lies in two oriented arc $l\in\mathbf{l}$ and
$l'\in \mathbf{l}$ marked with $o$ and $o'$. Then we lift the
bipartite link cobordism $(\cal{W},\cal{F},F_\alpha,F_\beta,\cal{A}_\Sigma)$ to
a bipartite disoriented link cobordism
$(\cal{W},\cal{F},F_\alpha,F_\beta,\cal{A},\cal{A}_\Sigma)$.

\emph{Step1: Transversality assumption}.
Let $c_{o,o'}$ be the core of the $\beta$-band $B^\beta$.
 After a small perturbation of $f$, we can assume 
  the core  $c_{o,o'}$ of the
  $\beta$-band $B^\beta$ intersects
  the unstable manifolds of $f$ transversely. Particularly,
  $c_{o,o'}$ does not go through the critical points of $f$.

  Now, we project the core $c_{o,o'}$ along the gradient flow of
  $f$ to the Heegaard surface $\Sigma$. The projection image
  $l_{o,o'}$ on $\Sigma$ is a path connecting $o,o'$. By the transversality
  assumption, the path $l_{w,z}$ intersects
  $\mathbf{\alpha,\beta}$-curves transversely on
  $\Sigma$. Furthermore, the path $l_{o,o'}$ only has regular
  self-intersection points, and the interior of $l_{o,o'}$ does not go through any basepoints. See Figure \ref{fig:st1}.

\begin{figure}
  \centering
  \includegraphics[scale=0.4]{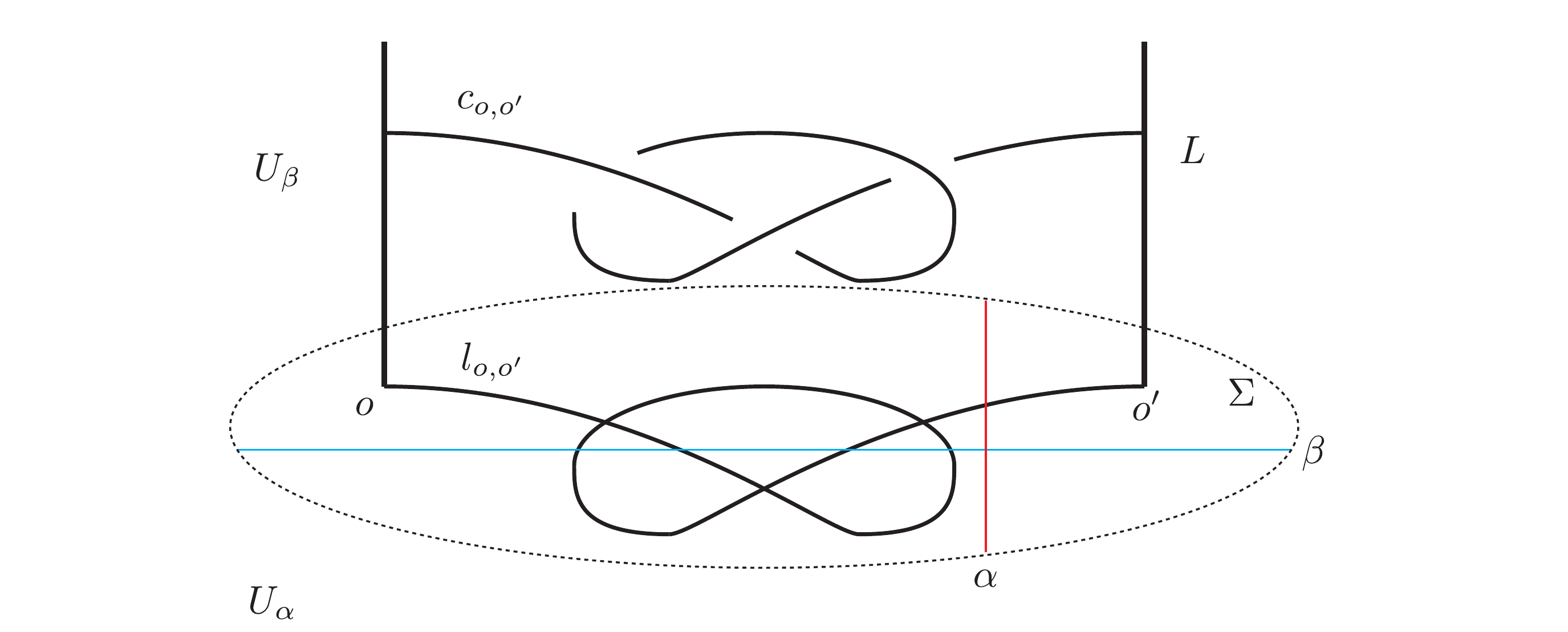}
  \caption{}
  \label{fig:st1}
\end{figure}

\emph{Step 2: Resolving the self-intersection of the path $l_{o,o'}$}.
  Suppose $p$ is a self-intersection of $l_{o,o'}$. We can find a disk
  neighborhood $D_p$ of $p$ such that $D_p\cap( \mathbf{O}\cup
  \bs{\alpha}\cup\bs{\beta})=\emptyset$. Let $\phi_t$ be the
  diffeomorphism induced by the flow $-\nabla f$. There is an
  embedded solid cylinder $C_p= D\times [0,\epsilon]\to U_\beta$
  given by $C_p(d,t)= \phi_t(d)$. By choosing a big enough parameter
  $\epsilon$, we can assume the pair $(C_p, C_p\cap c_{o,o'})$ is shown
  in Figure \ref{fig:st2}.

\begin{figure}
  \centering
  \includegraphics[scale=0.5]{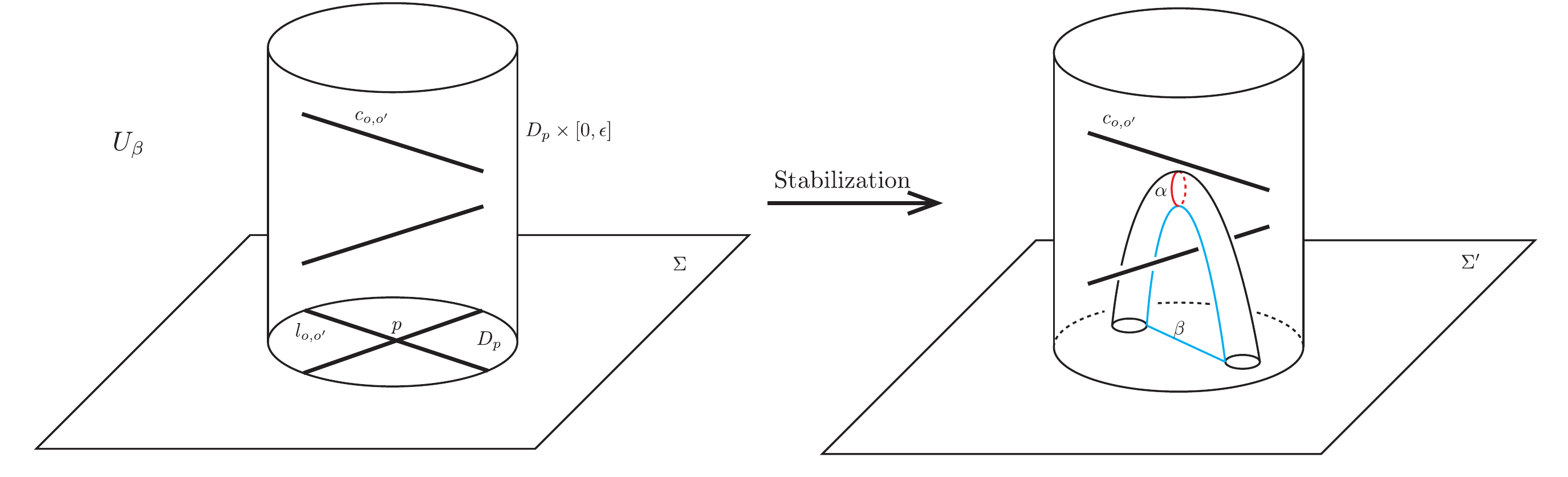}
  \caption{}
  \label{fig:st2}
\end{figure}

  Now we modify the  Morse function $f$ inside the solid cylinder $C_p$ by
  adding a pair of index one and index two critical
  point. In other words, this is doing a stabilization of $H$ inside
  the solid cylinder $C$ shown in Figure \ref{fig:st2}.

  By resolving all self-intersections of $l_{o,o'}$, we will get a new Morse function $f'$ and a
  corresponding balanced Heegaard diagram $H'$. If we
  projects the core $c_{o,o'}$ along the gradient flow of the
  modified function $f'$ on $\Sigma$, the projection image has no
  self-intersections. 

\emph{Step 3: Framing of the band}. By the transversality assumption, the gradient $-\nabla f|_{c_{o,o'}}$
induces a framing of the core $c_{o,o'}$. 
Heegaard Floer homology and integer surgeries on links. On the other hand the
embedding of the $\beta$-band $B^\beta$ also induces a framing of the
core. These two framings differed by $n\pm\frac{1}{2}$. In
other words, consider the solid cylinder neighborhood of $c_{o,o'}$, if
we identify the ends of the two neighborhood by the gradient framing,
then the band will twist along the core of the solid torus by $
(n\pm\frac{1}{2})\times 2\pi$.

Suppose $c'_{o,o'}$ is a small perturbation of $c_{o,o'}$ as shown in
Figure \ref{fig:ft}. The gradient framing of $c'_{o,o'}$ and $c_{o,o'}$ is
differed by $\pm 1$. If we resolve the self-intersection of the
projection image $l'_{o,o'}$ as in Step 2, the gradient framing of
$c'_{o,o'}$ does not change. Therefore, by modifying the Morse function
$f$ as in Step 2, we can modify the gradient framing of $c_{o,o'}$. 
Now we can assume that the difference between the gradient framing and band
framing is 
$\pm\frac{1}{2}$.

\begin{figure}
  \centering
  \includegraphics[scale=0.6]{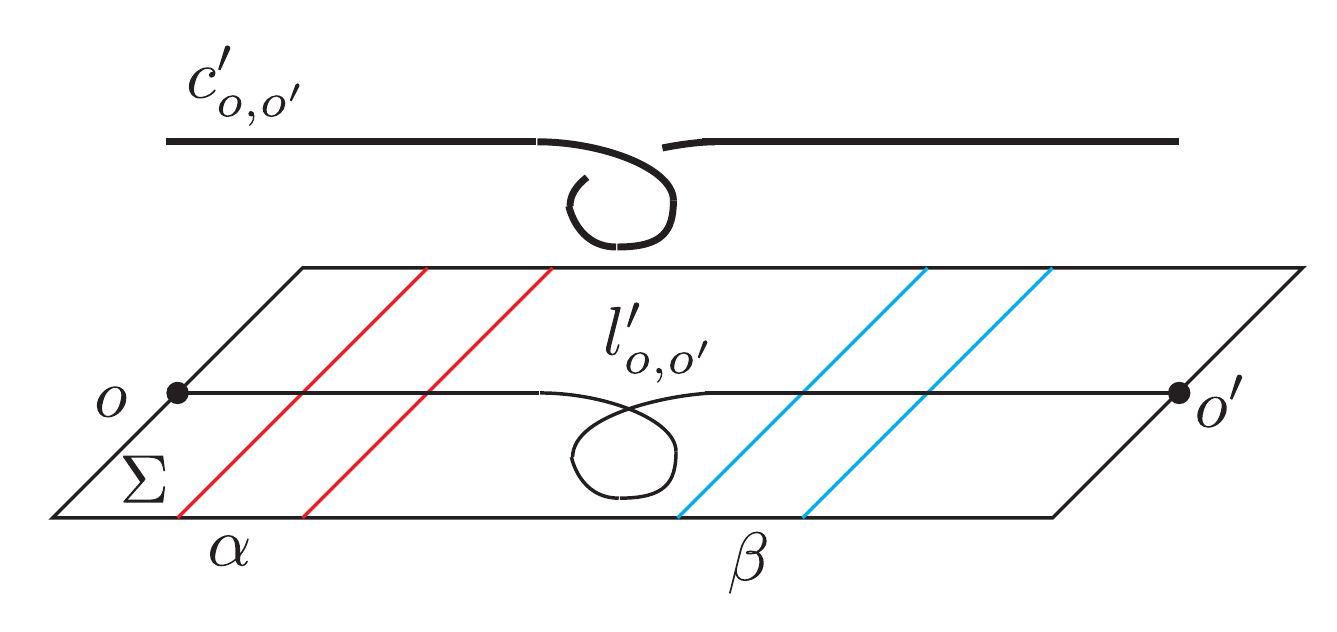}
  \caption{}
  \label{fig:ft}
\end{figure}

\emph{Step 4: Heegaard triple for band moves}.
  After resolving all self-intersections of $l_{o,o'}$, we can find a
  rectangle neighborhood $R:[-\epsilon,\epsilon]\times I \to \Sigma$ of
  $l_{o,o'}$, shown in Figure \ref{fig:st3}  such that:
\begin{itemize}
\item The image $R(0,I)= l_{o,o'}$ and $R(0,0)= o$, $R(0,1)=o'$.
\item The intersection $(\bs{\alpha}\cup\bs{\beta})\cap R =
  [-\epsilon,\epsilon]\times \{t_0,\cdots,t_m\}$, where
  $0<t_0,\cdots,t_m <1$.
\end{itemize}

\begin{figure}
  \centering
  \includegraphics[scale=0.5]{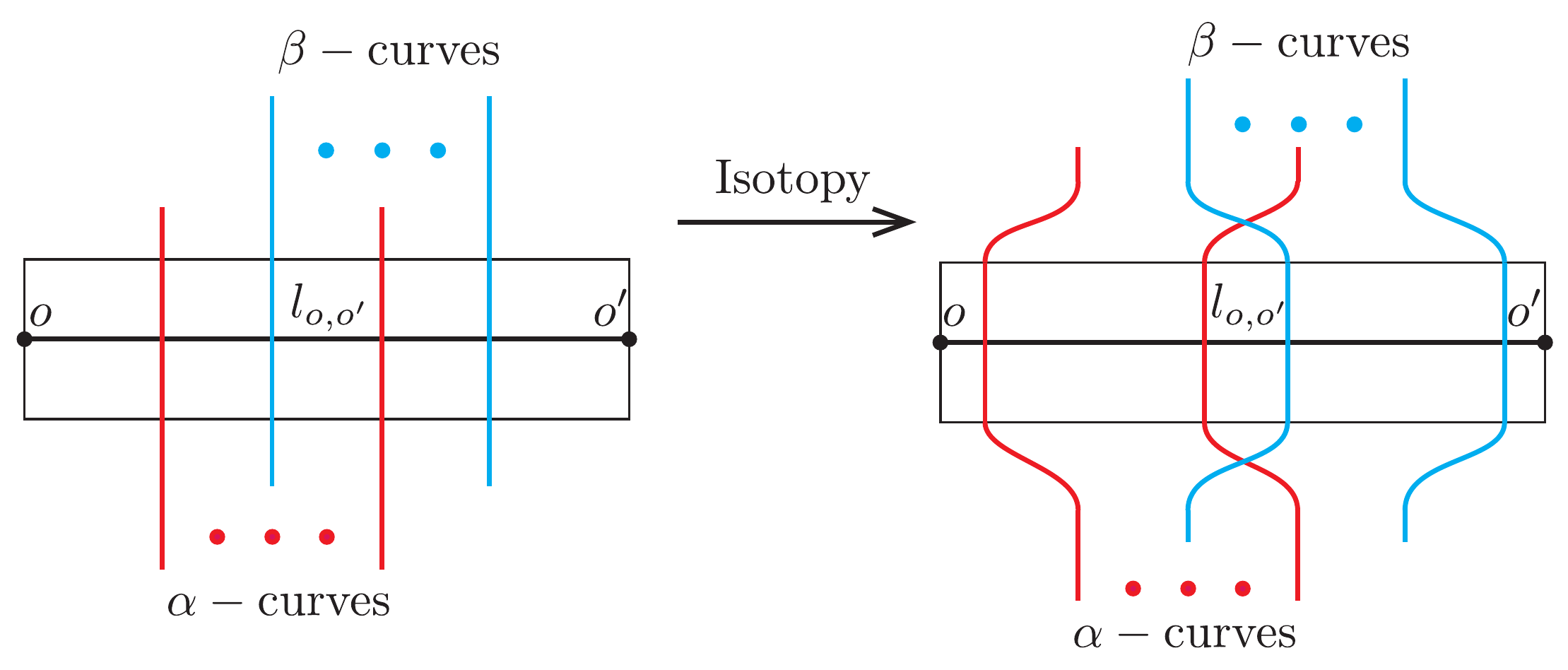}
  \caption{}
  \label{fig:st3}
\end{figure}

  After $\alpha$ or $\beta$- isotopies as shown in
  Figure \ref{fig:st3}, we can assume, the
  intersection $\alpha \cap R = [-\epsilon,\epsilon]\times
  \{t_0,\cdots,t_s\}$ and $\beta \cap R = [-\epsilon,\epsilon]\times
  \{t_{s+1}\cdots, t_m\}$. Then we stabilize the diagram as shown in
  Figure \ref{fig:st-08}. Finally, by an $\alpha$ or $\beta$ isotopy, we get a
  diagram shown in Figure \ref{fig:st-07} . The desired disk region is chosen to be
  the disk contains the two basepoint $o$ and $o'$ and one intersection
  of $\alpha$ and $\beta$. The two intersections determines two disk
  region. 

\begin{figure}
  \centering
  \includegraphics[scale=0.6]{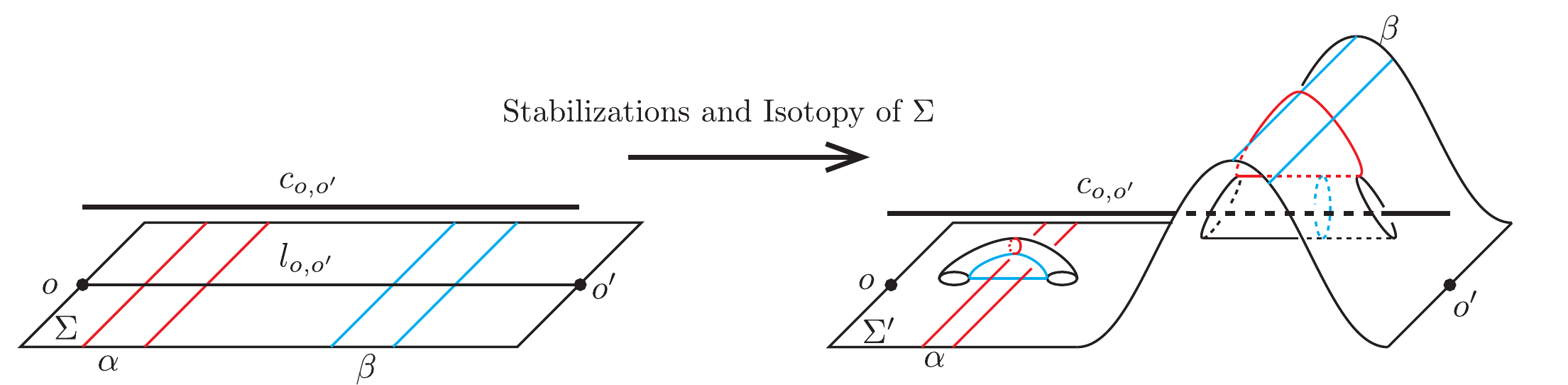}
  \caption{}
  \label{fig:st-08}
\end{figure}

\begin{figure}
  \centering
  \includegraphics[scale=0.7]{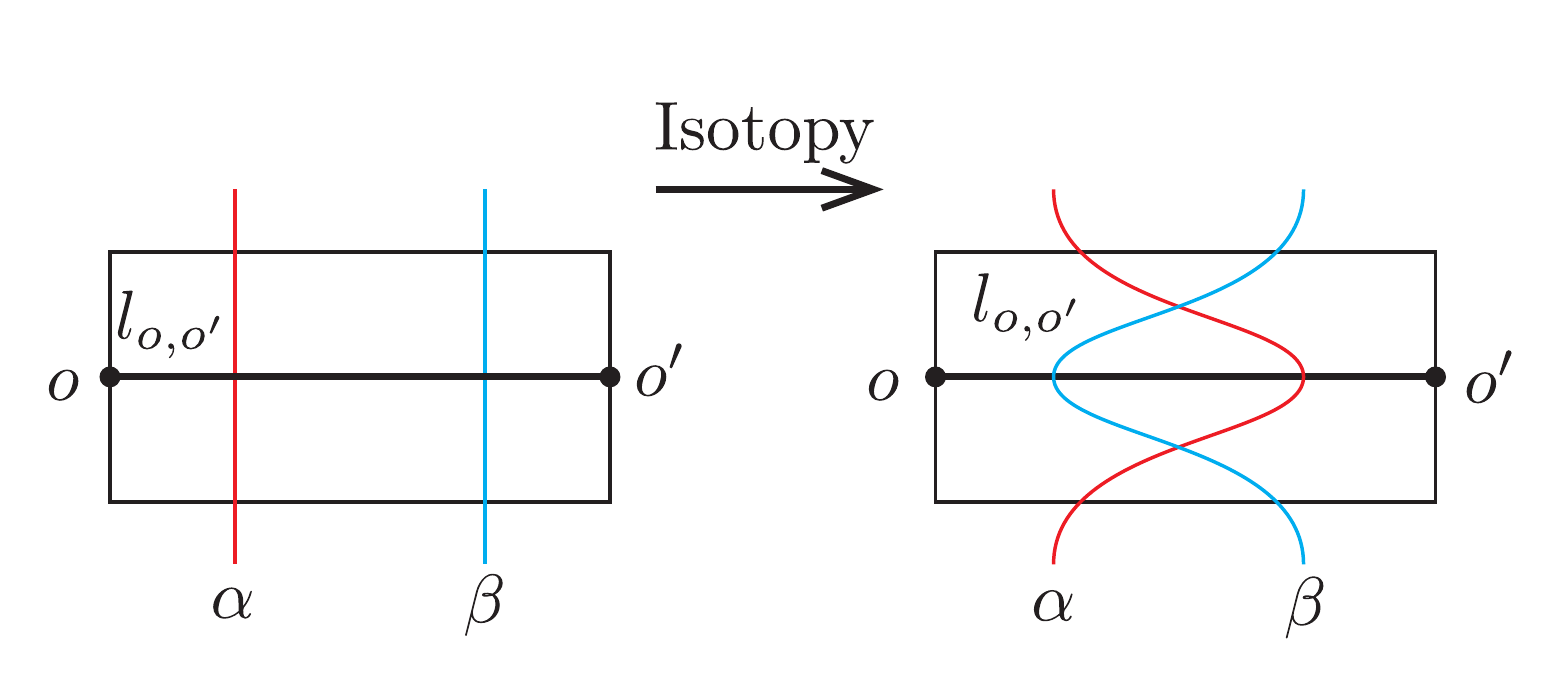}
  \caption{}
  \label{fig:st-07}
\end{figure}

   Now we assume the $\alpha$ and $\beta$ curves in Figure \ref{fig:st-07} are
   $\alpha_n$ and $\beta_n$. Let the curve $\gamma_n$ be a Hamiltonian
   perturbation of $\beta_n$ as shown in Figure \ref{fig:st-10}. We set
   $\gamma_1,\cdots,\gamma_{n-1}$ to be a small Hamiltonian
   perturbation of $\beta_1,\cdots,\beta_{n-1}$ without across any
   basepoints. Therefore, we construct a Heegaard Triple
   $\cal{T}$. Clearly, the choice the disk determines whether the
   result link is $L(B^\beta)$ or $L(B_{\pm 1}^{\beta})$.
   Here $B_{\pm 1}^{\beta}$ refers to the $\beta$-band which has the same
   core as $B^\beta$ but $\pm 1$-framing with respect to the framing of $B^\beta$.

\end{proof}

\subsection{Some topological facts about Heegaard triples subordinate
  to a band move}

\begin{lem}\label{sec:some-topol-facts-1}
  Any two Heegaard triples $\cal T_1$ and $\cal T_2$
  subordinate to a $B^\beta$-band move can be connected by
  the following type of Heegaard moves:

\begin{itemize}
\item Ambient isotopies of $\Sigma$ which fix $L\cup B^\beta$.
\item Isotopies and handle slides amongst
  $\alpha$,$\beta$,$\gamma$-curves without crossing any basepoints.
\item Stabilization or destabilization of a standard Heegaard Triple
  on torus. Here a standard Heegaard triple on torus has only one
  $\alpha$, one $\beta$ and one $\gamma$-curve, with $\alpha$ intersecting $\beta$
  and $\gamma$ once, and $\gamma$ is a small Hamiltonian isotopy of
  $\beta$ with intersection number $|\gamma\cap\beta|=2$. 
\end{itemize}

\end{lem}

\begin{proof}
  We stabilize the two Heegaard diagram sufficiently many times, then
  move the 
  critical points of the three
  Morse functions compatible with the three pairs
  $(U_\alpha,L_\alpha)$,$(U_\beta,L_\beta)$ and $(U_\gamma,L_\gamma)$
  without changing the gradient-like flow in a small neighborhood of $L_\alpha,L_\beta,L_\gamma$. 
\end{proof}

\begin{lem}\label{sec:some-topol-facts}

Suppose $\cal{T}$ is a Heegaard triple subordinate to a
$B^\beta$-band move from bipartite link $L_{\alpha\beta}$ to
$L_{\alpha\gamma}$. Then the induced Heegaard diagram
$H_{\beta\gamma}$ is a diagram for a bipartite link $L_{\beta\gamma}$
in $\#^g(S^1\times S^2)$. In fact, we have 
\[
  (\#^g(S^1\times S^2),L_{\beta\gamma})=
  \begin{cases}
    (\#^g(S^1\times S^2),K)\# (S^3,\mathbb{U}_2^{n-1}), &\text{ if }
    B^\beta \text{ is of Type I.}\\
  (\#^g(S^1\times S^2),\mathbb{U}_4)\# (S^3,\mathbb{U}_2^{n-2}),
   &\text{ if }
    B^\beta \text{ is of Type II}.
  \end{cases}
\]
Here $K$ is a bipartite knot with two basepoint in
  $\#^g(S^1\times S^2)$. The homology $[K]$ of the bipartite knot $K$
  is equal to $(0,\cdots,0,2)$
  in $H_1(\#^g(S^1\times S^2))$. The bipartite link $\mathbb{U}^l_{2k}$ is the bipartite unlink
  with $l$-components and 
  $2k$-basepoints on each of these components.
\end{lem}

\begin{proof}
Without loss of generality, we assume that the band $B^\beta$ is next
to a pair of basepoints $(o_1,o_2)$. By Theorem
\ref{sec:exist-heeg-triple},
we get a Heegaard diagram, as shown in Figure \ref{Fig:knotforsaddle1}
(or Figure \ref{Fig:knotforsaddle2} resp.)
if $B^\beta$ is of Type I (or Type II resp.). The result follows
directly from these two Heegaard diagrams. 

\end{proof}

\begin{figure}
  \centering
  \includegraphics[scale=0.5]{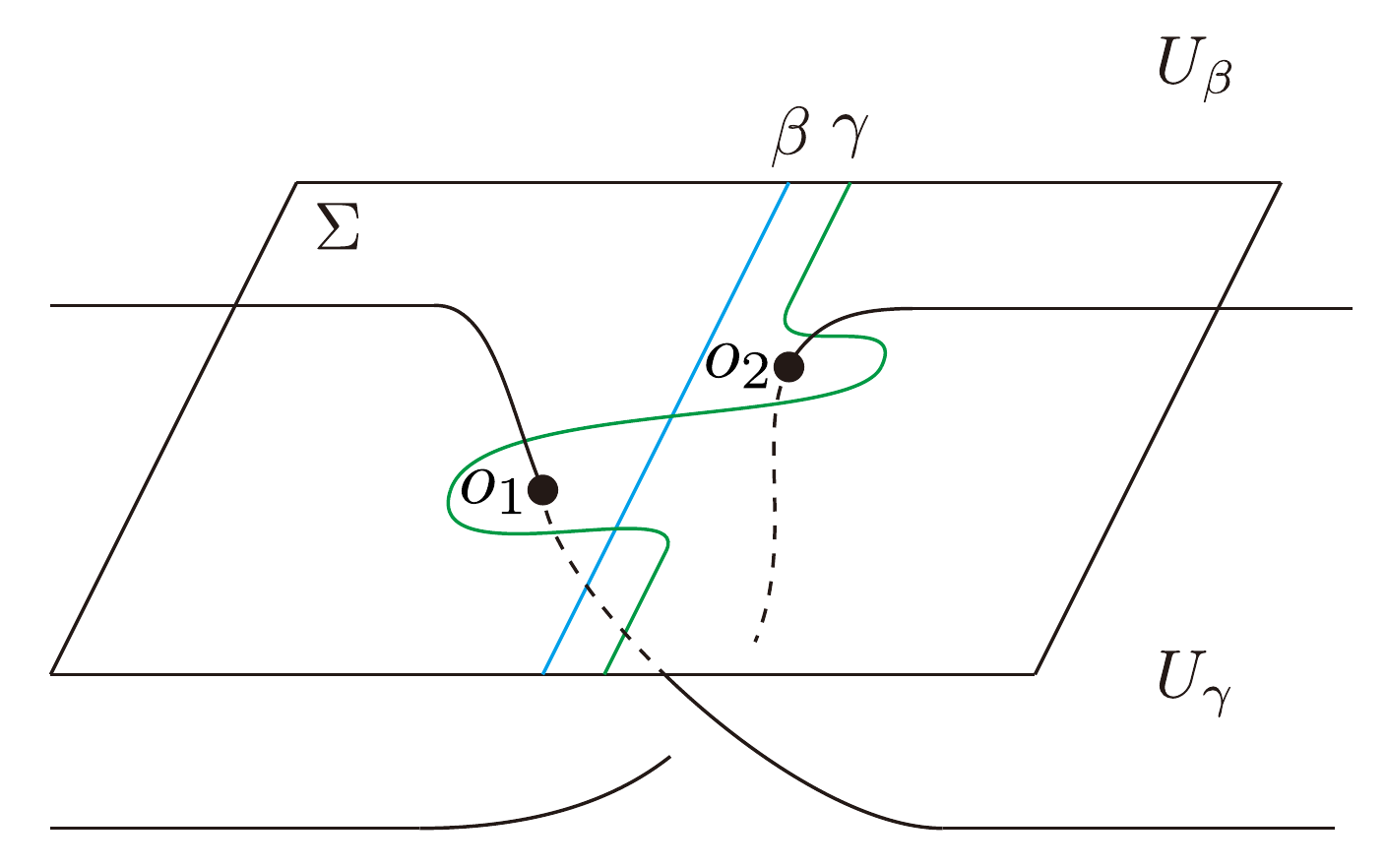}
  \caption{Links in  $\#^n(S^1\times S^2)$ generated by Type I band move }
  \label{Fig:knotforsaddle1}
\end{figure}

\begin{figure}
  \centering
  \includegraphics[scale=0.5]{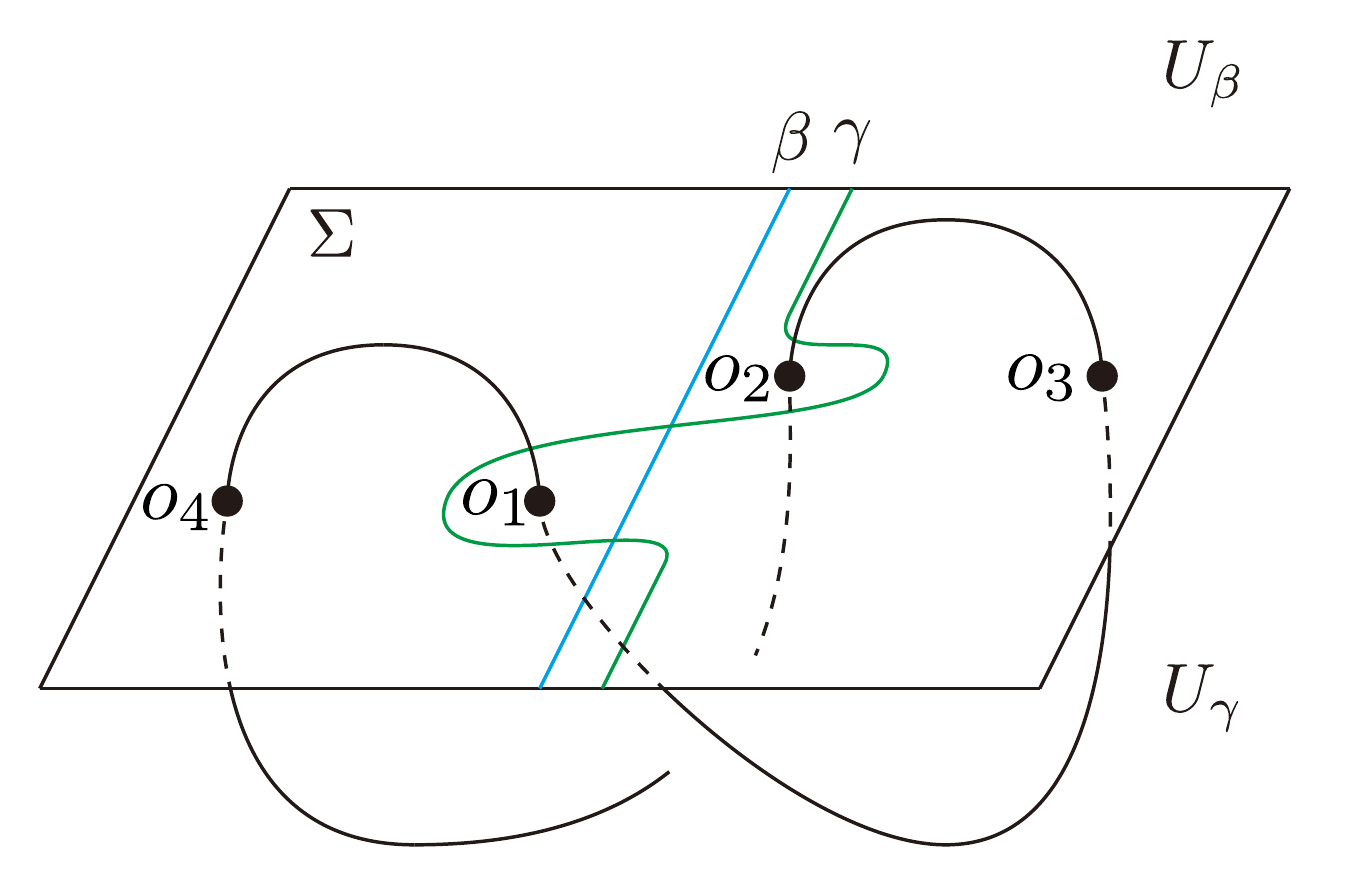}
  \caption{Links in  $\#(S^1\times S^2)$ generated by Type II band move
    }
  \label{Fig:knotforsaddle2}
\end{figure}
\section{Bipartite link Floer homology}
In this section, we construct the bipartite link Floer homology for
homologically even bipartite links. 

\subsection{Bipartite link Floer curved chain complex for
  null-homologous links}

In this subsection, we construct a well-defined $\mathbb{Z}$-graded curved chain complex
for a null homologous biparitite link $L_{\alpha\beta}$ in a closed
oriented three-manifold $Y$.

Suppose $H_{\alpha\beta}$ is a Heegaard diagram for the bipartite link
$L_{\alpha\beta}$. Let $J_t$ be a generic family of almost complex
structures on $\Sigma$. We denote by $\cal{H}_{\alpha\beta}$ the
Heegaard data $(H_{\alpha\beta},J_t)$. Now we can define a
$\mathbb{F}_2[U_1,\cdots,U_{2n}]$-module 
$CFBL^{-}(\cal{H}_{\alpha\beta})$, together with an endomorphism $\partial$ as follows:

\begin{itemize}
\item The module $CFBL^{-}(\cal{H}_{\alpha\beta})$ is freely generated by the
  intersections of the two Lagrangians $\mathbb{T}_\alpha$ and
  $\mathbb{T}_\beta$ in the symplectic manifold $\text{Sym}^{g+n-1}(\Sigma)$.
\item The endomorphism $\partial$ acting on a generator $\x\in\mathbb{T}_\alpha\cap\mathbb{T}_\beta$ is given by:
\[\partial \x
  =\sum_{\y\in\mathbb{T}_\alpha\cap\mathbb{T}_\beta}\sum_{\phi\in\pi_2(x,y),\mu(\phi)=1}\#(\cal{M}(\phi)/\mathbb{R})\prod_{i=1}^{2n}U^{n_{o_i}(\phi)}_i\y.\]
\end{itemize}
Here $g$ is the genus of the Heegaard surface $\Sigma$, $2n$ is the number of basepoints on the bipartite link $L_{\alpha\beta}$.

Given an alternating coloring $\mathfrak{P}$ of
$L_{\alpha\beta}$, we can define a map from the set of generators of
$CFBL^{-}(\cal{H}_{\alpha\beta})$ to the set of Spin$^c$-structure of
$Y$ as in \cite[Section 2.6]{Ozsvath2004c}. In detail, we set our
Spin$^c$-structure map $\mathfrak{s}_{\mathfrak{P}}(\x)$ to be
$\mathfrak{s}_{\mathbf{w}}(\x)$, where $\mathbf{w}$ is the subset
$\mathfrak{P}^{-1}(+1)$ of basepoints.
If $\mathfrak{P}'$ is another alternating coloring of
$L_{\alpha\beta}$, the difference between the two maps
$s_{\mathfrak{P}}$ and $s_{\mathfrak{P}'}$ is:
\begin{equation}\label{eq:2}
  \mathfrak{s}_{\mathfrak{P}}(\x)-\mathfrak{s}_{\mathfrak{P}'}(\x)=\frac{1}{2}(\text{PD}[L_{\mathfrak{P}}]-\text{PD}[L_{\mathfrak{P}'}]). 
\end{equation}

Here $L_{\mathfrak{P}}$  and $L_{\mathfrak{P}'}$ are the oriented links
 determined by the
bipartite link $L_{\alpha\beta}$ 
and the coloring $\mathfrak{P}$ and $\mathfrak{P}'$ respectively in Remark \ref{sec:coloring}. As we assume $L_{\alpha\beta}$ in $Y$
is null-homologous, the difference
$s_{\mathfrak{P}}(\x)-s_{\mathfrak{P}'}(\x)$ is zero. Therefore, we
have a map
$\mathfrak{s}_{\delta}=\mathfrak{s}_{\mathfrak{P}}:\mathbb{T}_{\alpha}\cap\mathbb{T}_\beta\rightarrow
\text{Spin}^c(Y)$, which is independent of the choice of alternating
coloring $\mathfrak{P}$ of $L_{\alpha\beta}$. For the details of
Spin$^c$-structure map, see \cite[Section 2.6]{Ozsvath2004c}
and \cite[Section 3.3]{Ozsvath2008b}.
 The module $CFBL^{-}(\cal{H}_{\alpha\beta})$
now splits into a direct sum:
\[CFBL^-(\cal{H}_{\alpha\beta})=\bigoplus_{\mathfrak{s}\in
    \text{Spin}^c(Y)}CFBL^-(\cal{H}_{\alpha\beta},\mathfrak{s}),\]
where $CFBL^-(\cal{H}_{\alpha\beta},\mathfrak{s})$ consists of generators
$\x$ whose image under $\mathfrak{s}_{\delta}$ is equal to $\mathfrak{s}$.

From now on, we assume $\mathfrak{s}$ is a torsion Spin$^c$-structure of
$Y$. The diagram
$H_{\alpha\beta}$ is weakly $\mathfrak{s}$-admissible, then we can get
finite counts of moduli spaces. We will discuss admissibility in
Section \ref{sec:admissibility}.

Based on the work in \cite{Ozsvath2015}, we have a relative
$\delta$-grading for
$CFBL^-(\cal{H}_{\alpha\beta},\mathfrak{s})$ as follows. Let $\x$ and $\y$ be two
generators in $CFBL^-(\cal{H}_{\alpha\beta},\mathfrak{s})$ with $\pi_2(\x,\y)$ being non-empty. The relative $\delta$-grading
between $\x$ and $\y$ is given by
\begin{equation}\label{eq:3}
  \delta(\x,\y)=\mu(\phi)-n_{\mathbf{O}}(\phi),
\end{equation}
where $n_{\mathbf{O}}$ denotes the sum $\sum^{2n}_{i=1}n_{o_i}(\phi)$,
and $\phi$ is an element in $\pi_2(\x,\y)$. Notice that there is a pair of
basepoints $(o,o')$ on each component of $\Sigma\backslash
\bs{\alpha}$ or $\Sigma\backslash\bs{\beta}$. Using Lipshitz's formula
\cite[Corollary 4.3]{Lipshitz2006},
we know that the Maslov index $\mu(\cal{P})$ of a periodic domain $\cal{P}$ is
 equal to $n_{\mathbf{O}}(\cal{P})$. This 
implies that the relative $\delta$-grading is well-defined.
Now we can set the $\delta$-grading of
the variable $U_i$ to be $-1$ and extend the relative $\delta$-grading to the
submodule $CFBL^-(\cal{H}_{\alpha\beta},\mathfrak{s})$. 
Moreover, as the
Spin$^c$-structure $\mathfrak{s}$ is torsion, the $\delta$-grading is actually
a $\mathbb{Z}$-grading on $CFBL^-(\cal{H}_{\alpha\beta},\mathfrak{s})$.

As a corollary of \cite[Lemma 2.1]{Zemke2016}, we have
\[\partial^2 = \sum_i
  (U_{i,1}U_{i,2}+U_{i,2}U_{i,3}+\cdots+U_{i,k_i}U_{i,1}).\]
Here $i$ refers the $i$-th component $L_i$ of $L$, the variable $U_{i,j}$ is the
variable assigned to the 
basepoint $o_{i,j}$. The basepoints
$o_{i,j}$ appear in order on the link component $L_i$.

By the above discussion, we know that the
$\mathbb{F}_2[U_1,\cdots,U_{2n}]$-module
$CFBL^-(\cal{H}_{\alpha\beta},\mathfrak{s})$ together with the 
endomorphism $\partial$ and the $\delta$-grading is a well-defined
$\mathbb{Z}$-graded curved chain complex.

\begin{rem}
Given an alternating coloring $\mathfrak{P}$ of
$L_{\alpha\beta}$, we can get the curved chain complex
$CFL^-_{UV}(\cal{H_{\alpha\beta}})$ defined by Zemke in \cite{Zemke2016} 
from $CFBL^-(\cal{H_{\alpha\beta}})$ by setting the variable of
$\mathbf{z}$-basepoints to be $V_i$'s. Furthermore, the average
of the two gradings $gr_{\mathbf{w}}$ and $gr_{\mathbf{z}}$
gives the $\delta$-grading on the curved chain complex $CFL^-_{UV}$.  The curved chain
complex $CFBL^-(\cal{H}_{\alpha\beta})$ together with a coloring
$\mathfrak{P}$ has the same amount of data as $CFL^-_{UV}(\cal{H}_{\alpha\beta})$.
\end{rem}

\subsection{Spin$^c$-structures}

In this section, we will discuss the relation between the Spin$^c$-structures of $Y$ and a Spin$^c$-structure map
induced by homologically even bipartite links. 

Consider the Heegaard diagram $H_{\beta\gamma}$ of a bipartite link
$L_{\beta\gamma}$ induced from a standard
Heegaard triple $\cal{T}$ subordinate to a $\beta$-band move of type
I as in Lemma \ref{sec:some-topol-facts} and Figure
\ref{Fig:knotforsaddle1}. Notice that $L_{\beta\gamma}$ is not null-homologous.
Let $\mathfrak{P}$ be an alternating coloring of
$L_{\beta\gamma}$. 
Then all the generators in $\mathbb{T}_{\beta}\cap\mathbb{T}_{\gamma}$
belong to the same equivalence class $\mathfrak{s}$.
We have the following equality for the
Spin$^c$ structure map $\mathfrak{s}_{\mathfrak{P}}$:
\begin{equation}
  \mathfrak{s}_{\mathfrak{P}}(x)-\mathfrak{s}_0 = \mathfrak{s}_0 -
  \mathfrak{s}_{\mathfrak{-P}}(x) = \pm\text{PD}[\beta_n],
\end{equation}
where $\mathfrak{s}_0$ is the only
torsion Spin$^c$-structure on $\#^n(S^1\times S^2)$.
The grading $gr_{\mathfrak{P}}$ (or $gr_{-\mathfrak{P}}$) on the equivalence class $\mathfrak{s}$
is a $\mathbb{Z}_2$-grading, which can not be lifted to a
$\mathbb{Z}$-grading. Therefore, it is necessary to find a sufficient
condition for which the $\delta$-grading on the equivalence class
$\mathfrak{s}$ is a $\mathbb{Z}$-grading.

\begin{defn}
  We say that a bipartite link $L_{\alpha\beta}$ is \textit{\textbf{homologically
  even}}, if the homology class $[L_{\mathfrak{P}}]\in H_1(Y)$ is
  divisible by two, where $\mathfrak{P}$ is
  an alternating coloring of $L_{\beta\gamma}$. 
\end{defn}

\begin{lem}\label{sec:curved-chain-complex}
  Let $L_{\beta\gamma}$ be a homologically even bipartite link in a closed oriented
  three-manifold $Y$.  Then, the map
  $\mathfrak{s}_{\delta}:\mathbb{T}_{\alpha}\cap \mathbb{T}_{\beta}\to
  \textnormal{Spin}^c(Y)$ defined by 
\begin{equation}\label{eq:6}
  \mathfrak{s}_\delta(\x) \triangleq \mathfrak{s}_{\mathfrak{P}}(\x)-\frac{1}{2}\text{PD}[L_{\mathfrak{P}}],
\end{equation}
is independent the choice of $\mathfrak{P}$. Furthermore, if
$\mathbf{x}$ and $\mathbf{y}$ are two generators in
$\mathbb{T}_\alpha\cap\mathbb{T}_\beta$ with $\pi_2(\x,\y)$ being
non-empty, then $\mathfrak{s}_\delta(\x)=\mathfrak{s}_\delta(\y)$. 
Therefore, we have the following decomposition 
\[CFBL^-(\cal{H}_{\alpha\beta})=\bigoplus_{\mathfrak{s}\in
    \textnormal{Spin}^c(Y)}CFBL^-(\cal{H}_{\alpha\beta},\mathfrak{s}),\]
where $CFBL^-(\cal{H}_{\alpha\beta},\mathfrak{s})$ consists of
generators $\x$ with $\mathfrak{s}_\delta(\x)=\mathfrak{s}$.
\end{lem}

\begin{proof}
  Let $\mathfrak{P}'$ be another alternating coloring of
  $L_{\beta\gamma}$. By Equation \ref{eq:2}, the difference between
  the two Spin$^c$-structure
  $\mathfrak{s}_{\mathfrak{P}}(\x)-\mathfrak{s}_{\mathfrak{P}'}(\x)=\text{PD}[L_{\mathfrak{PP}'}]$. Here the link $L_{\mathfrak{PP}'}$ is a sublink of $L_{\mathfrak{P}}$ consists of components $L_i$ of $L$ with $\mathfrak{P}(L_i)=-\mathfrak{P}'(L_i)$. On
  the other hand,  by the construction in Remark \ref{sec:coloring}, we have
  $2\text{PD}[L_{\mathfrak{PP}'}]=\text{PD}[L_{\mathfrak{P}}]-\text{PD}[L_{\mathfrak{P}'}]$. Combining
  this two equalities, we get that the Spin$^c$-structure map
  $\mathfrak{s}_{\mathfrak{P}}(\x)-\frac{1}{2}\text{PD}[L_{\mathfrak{P}}]$ is equal to $\mathfrak{s}_{\mathfrak{P}'}(\x)-\frac{1}{2}\text{PD}[L_{\mathfrak{P}'}]$.

If $\pi_2(\x,\y)$ is non-empty, we get
$\mathfrak{s}_{\mathfrak{P}}(\x)=\mathfrak{s}_{\mathfrak{P}}(\y)$. Therefore
the map
$\mathfrak{s}_\delta$ sends the equivalence classes of generators in
$\mathbb{T}_\beta\cap\mathbb{T}_\gamma$ to the
Spin$^c$-structures of $Y$.
\end{proof}

From now on, we assume all the bipartite link we discuss satisfies the same
condition in Lemma \ref{sec:curved-chain-complex}, i.e. its homology
class is divisible by two.

\subsection{Admissibility}\label{sec:admissibility}

Although one can assign an alternating coloring to a bipartite link and
define the admissibility with respect to the
$\mathbf{w}$-basepoints.  This admissibility with respect to
$\mathbf{w}$-basepoints can not help us define a $\mathbb{Z}$-graded
complex for some null-homologous links in $\#^g(S^1\times S^2)$. In
this subsection, we will introduce the admissibility with respect to
$\textnormal{Spin}^c$-structure map $\mathfrak{s}_{\delta}$.

\begin{defn}
  Suppose $\mathfrak{s}$ is a $\textnormal{Spin}^c$-structure over
  $Y$, we say that a Heegaard diagram
  $H_{\alpha\beta}=(\Sigma,\bs{\alpha},\bs{\beta},\mathbf{O})$ is
  $\mathfrak{s}$-\textit{\textbf{realized}}, if there is a point $\x\in
  \mathbb{T}_\alpha\cap\mathbb{T}_\beta$ and a $n$-tuple
  of points $\mathbf{q}$ on $\Sigma\backslash
  (\bs{\alpha}\cup\bs{\beta})$, such that:
\[\mathfrak{s}_{\delta}(\x)=\mathfrak{s}_{\mathbf{q}}(\x)=\mathfrak{s}.\]
Furthermore, we say that an $\mathfrak{s}$-realized Heegaard diagram $H_{\alpha\beta}$ is \textit{\textbf{weakly \textnormal{$\mathfrak{s}$-}admissible}} (or
\textit{\textbf{strongly \textnormal{$\mathfrak{s}$-}admissible}} resp.), if the diagram
$(\Sigma,\bs{\alpha},\bs{\beta},\mathbf{q})$ is weakly \textnormal{$\mathfrak{s}$-}admissible (or
strongly \textnormal{$\mathfrak{s}$-}admissible resp.).
\end{defn}

\begin{lem}
  Given a bipartite link $L_{\beta\gamma}$ in a closed three-manifold
  $Y$, together with a fixed $\textnormal{Spin}^c$-structure
  $\mathfrak{s}$, we can construct a weakly
  \textnormal{$\mathfrak{s}$-}admissible (or strongly
  \textnormal{$\mathfrak{s}$-}admissible resp.) Heegaard
  diagram $H_{\beta\gamma}$ compatible with the pair $(Y,L_{\beta\gamma})$. 
\end{lem}

\begin{proof}
  This follows directly from the proof of \cite[Lemma
  5.2]{Ozsvath2004c} and the proof of \cite[Lemma 5.4]{Ozsvath2004c}. Here, we require
  the finger moves of  $\bs{\beta}$ and $\bs{\gamma}$ curves do not
  across the basepoints $\mathbf{O}$.
\end{proof}

\begin{lem}
  Any two weakly
  \textnormal{$\mathfrak{s}$-}admissible (or strongly
  \textnormal{$\mathfrak{s}$-}admissible resp.) Heegaard
  diagram $H_{\beta\gamma}$ compatible with the pair
  $(Y,L_{\beta\gamma})$ can be connected by a sequence of Heegaard
  moves without crossing basepoints $\mathbf{O}$. Furthermore, each
  intermediate Heegaard diagram is weakly
  \textnormal{$\mathfrak{s}$-}admissible (or strongly
  \textnormal{$\mathfrak{s}$-}admissible resp.).
\end{lem}

\begin{proof}
This follows from the proof of \cite[Lemma 5.6]{Ozsvath2004c}. See also
\cite[Section 3.4]{Ozsvath2008b}, and \cite{Juhasz2012}. Again, we
require the Heegaard moves do not across basepoints $\mathbf{O}$.
\end{proof}

\subsection{The curved chain complex $CFBL^-$ for bipartite links in
  $\#^g(S^1\times S^2)$ }

\begin{lem}\label{sec:curved-chain-complex-1}
Let the Heegaard data $\cal{H}_{\beta\gamma}$ for the bipartite link
$L_{\beta\gamma}$ be weakly $\mathfrak{s}$-admissible. Here $\mathfrak{s}$
is a torsion Spin$^c$-structure of $Y$.  Then the curved chain
complex $CFBL^-(\cal{H}_{\beta\gamma},\mathfrak{s})$ is a well-defined
$\mathbb{Z}$-graded curved chain complex, with grading defined as in
Equation \ref{eq:3}.
\end{lem}

\begin{proof}
  Let $\x,\y$ be two generators in
  $CFBL^-(\cal{H}_{\beta\gamma},\mathfrak{s})$. We want to show the grading:
  \[\delta(\x,\y)=\mu(\phi)-n_{\mathbf{O}}(\phi)\]
is well-defined. It suffice to show the following equality:
\[\mu(\psi)=n_{\mathbf{O}}(\psi)\] 
holds for every class $\psi\in \pi_2(\x,\x)$ of
$H_{\alpha\beta}$.
Because $\mathfrak{s}$ is torsion, we have a relative
$\mathbb{Z}$-grading:
 \[gr_{\mathbf{q}}(\x,\y)=\mu(\phi)-2n_{\mathbf{q}}(\phi),\]
where $\phi$ is a class in $\pi_2(\x,\y)$.
Its Maslov index of $\psi\in\pi_2(\x,\x)$ is
\begin{equation}\label{eq:4}
\mu(\psi)= 2n_{\mathbf{q}}(\psi). 
\end{equation}

Given an alternating coloring  $\mathfrak{P}$ of
$L_{\beta\gamma}$, we denote by $\mathbf{w}=\{w_1,\cdots,w_n\}$ the
$n$-tuple of basepoints satisfying $\mathfrak{P}(w_i)=1$ and by
$\mathbf{z}=\{z_1,\cdots,z_n\}$ the $n$-tuple of basepoints
satisfying $\mathfrak{P}(z_i)=-1$. By \cite[Lemma 2.18]{Ozsvath2004c},
we have the following equalities:
\begin{equation}\label{eq:5}
  n_{\mathbf{w}}(\psi)-n_{\mathbf{q}}(\psi)=\langle
                                             H(\psi),a^*\rangle =
  n_{\mathbf{q}}(\psi)-n_{\mathbf{z}}(\psi).
\end{equation}
Here, $H(\psi)\in H_2(Y;\mathbb{Z})$ is the homology class belonging to
the periodic class, $a^*$ is a cohomology class in $H^1(Y)$ determined by the
relative position between $\mathbf{w}$ and $\mathbf{q}$. 
Combining the Equation \ref{eq:4} and Equation \ref{eq:5}, we have
Maslov grading
\[
  \mu(\psi)=2n_{\mathbf{q}}(\psi)=
  n_{\mathbf{w}}(\psi)+n_{\mathbf{z}}(\psi)=n_{\mathbf{O}}(\psi).
\]
This implies the relative $\delta$-grading which is defined by
\[\delta(\x,\y)=\mu(\phi)-n_{\mathbf{O}}(\phi)\]
is a $\mathbb{Z}$-grading on $CFBL^-(\cal{H_{\beta\gamma},\mathfrak{s}})$.

Moreover, as $n_{\mathbf{O}}(\psi)=2n_{\mathbf{q}}(\psi)$, the
finiteness of counting follows from the admissibility with respect to $\mathbf{q}$.

\end{proof}

Applying the above results to the bipartite links in $\#^g(S^1\times
S^2)$, we have the following.

\begin{cor}\label{sec:textsp-struct-delta}
  Suppose $\cal{T}$ is a Heegaard triple subordinate to a
  $\beta$-band move $B^\beta$, and $H_{\beta\gamma}$ is the induced
  Heegaard diagram for bipartite link $L_{\beta\gamma}$ in
  $\#^g(S^1\times S^2)$. Let $\mathfrak{s}_0$ be the unique torsion
  Spin$^c$-structure of $\#^g(S^1\times S^2)$.
 Then, the $\delta$-graded curved chain complex
 $CFBL^-(\cal{H_{\alpha\beta}},\mathfrak{s}_0)$ is a $\mathbb{Z}$-graded
 curved chain complex.
Furthermore, the span of the top grading generators in
$HFL'(\cal{H_{\alpha\beta}},\mathfrak{s}_0)$ is a two-dimensional
vector space. 
\end{cor}

\begin{proof}  

 By
Lemma \ref{sec:some-topol-facts}, the Heegaard diagram
$H_{\beta\gamma}$ is a diagram for the bipartite link $(\#^g(S^1\times
S^2),K)\#(S^3,\mathbb{U}^{n-1}_2)$ if $B^\beta$ is of type I, and for 
$(\#^g(S^1\times S^2),\mathbb{U}_4)\# (S^3,\mathbb{U}_2^{n-2}))$ if $B^\beta$
is of type II . Here $n$ is the number of
$\alpha$-circles on $\cal{T}$, $K$ is a knot with homology equal to
twice of the dual of $\beta_n$. In either of these two cases, the
homology class $[L_{\mathfrak{P}}]$ for an alternating coloring
$\mathfrak{P}$ of the
bipartite link $L_{\beta\gamma}$ is divisible by two. Applying Lemma
\ref{sec:curved-chain-complex-1}, we get that the $\delta$-grading for
curved chain complex 
 $CFBL^-(\cal{H_{\alpha\beta}},\mathfrak{s}_0)$ is a $\mathbb{Z}$-grading.

Without loss of generality, suppose
$\cal{T}$ is a standard Heegaard triple subordinate to the 
band move $B^\beta$ from $L_{\alpha\beta}$ to $L_{\alpha\gamma}$.
 By Theorem
\ref{sec:exist-heeg-triple}, as $\cal{T}$ is standard, 
the curves $\gamma_{1},\cdots,\gamma_{n-1}$ are a small Hamiltonian
isotopies of $\beta_{1},\cdots,\beta_{n-1}$ without crossing any
basepoints. The curve $\gamma_n$, which is a Hamiltonian isotopy of
$\beta_n$ crossing basepoints,
intersect $\beta$-circles at two points (See Figure
\ref{Fig:knotforsaddle1} and Figure \ref{Fig:knotforsaddle2}).  Therefore, we have $2^n$
generators for the module
$CFBL^-(H_{\beta\gamma},\mathfrak{s}_0)$. Clearly, there are two top
grading generators in the kernel of $\partial$ for
$CFL'(H_{\beta\gamma},\mathfrak{s}_0)$.  
\end{proof}

\subsection{Holomorphic triangles} \label{sec:holotri}

Recall from \cite[Section 8]{Ozsvath2004c} that, given a Heegaard
triple $\cal{T}$, we can classify the homotopy classes of Whitney
triangles as follows. Suppose
$\psi\in\pi_2(\x,\y,\mathbf{v})$ and $\psi'\in
\pi_2(\x',\y',\mathbf{v}')$ are two Whitney triangles. We say that $\psi$ and
$\psi'$ are equivalent, if there exists
classes $\phi_1\in\pi_2(\x,\x')$ and $\phi_2\in\pi_2(\y,\y')$ and
$\phi_3\in \pi_2(\mathbf{v},\mathbf{v}')$, such that:
\[\psi'=\psi+\phi_1+\phi_2+\phi_3.\] 
We denoted by
$S_{\alpha\beta\gamma}(\cal{T})$ the set of equivalence classes of the Whitney
triangles of $\cal{T}$. 

Suppose $\mathbf{q}$ is a $n$-tuple of basepoints on $\cal{T}$, such
that the induced $n$-pointed diagrams
$H_{\alpha\beta},H_{\beta\gamma}$ and $H_{\alpha\gamma}$ are
well-defined pointed Heegaard diagrams for pointed closed oriented three-manifolds
$Y_{\alpha\beta}$, $Y_{\beta\gamma}$ and $Y_{\alpha\gamma}$
. By \cite[Proposition 8.5]{Ozsvath2004c}, we have a one-to-one
map:
\[\mathfrak{s}_{\mathbf{q}}:S_{\alpha\beta\gamma}(\cal{T})\rightarrow \text{Spin}^c(X_{\alpha\beta\gamma}).\]
Here $X_{\alpha\beta\gamma}$ is an oriented four-manifold constructed
from the Heegaard triple $\cal{T}$, whose boundary $\partial X_{\alpha\beta\gamma}$ is the union of
three-manifolds $-Y_{\alpha\beta}\sqcup-Y_{\beta\gamma}\sqcup Y_{\alpha\gamma}$.

Suppose the Heegaard triple $\cal{T}$ is subordinate to a band move
from a bipartite link $(Y,L_{\alpha\beta})$ to
$(Y,L_{\alpha\gamma})$. The three-manifold $Y_{\alpha\beta}$ and $Y_{\alpha\gamma}$ are
diffeomorphic to $Y$, the three-manifold
$Y_{\beta\gamma}$ is diffeomorphic to $\#^n(S^1\times S^2)$ and the
four-manifold $X_{\alpha\beta\gamma}$ is actually $(Y\times I)
\backslash N(U_{\beta}\times\{\frac{1}{2}\})$. 

Notice that, we can always construct an almost complex structure
over $Y\times I$. Choices of almost complex structure estabilish a
one-to-one correspondence between 
the $\text{Spin}^c$-structure of $Y\times I$ and $H^2(Y\times
I;\mathbb{Z})\cong H^2(Y;\mathbb{Z})$. For a $\text{Spin}^c$-structure $\mathfrak{s}$ on $Y\times
I$, its restriction $\mathfrak{s}_{\alpha\beta}$ on $Y_{\alpha\beta}$
is equal to $\mathfrak{s}_{\alpha\gamma}$  on  $Y_{\alpha\gamma}$, and
its restriction $\mathfrak{s}_{\beta\gamma}$ should be
$\mathfrak{s}_0$ on $\#^g(S^1\times S^2)$. Conversely, we claim that  a
$\text{Spin}^c$-structure $\mathfrak{s}$ on $Y_{\alpha\beta}$ and
$Y_{\alpha\gamma}$, together with the unique
$\text{Spin}^c$-structure $\mathfrak{s}_0$ on $\#^g(S^1\times S^2)$,
can be uniquely extended to a $\text{Spin}^c$-structure
$\mathfrak{s}_{\alpha\beta\gamma}$ on
$X_{\alpha\beta\gamma}$. 
Otherwise, we suppose
$\mathfrak{s}^1_{\alpha\beta\gamma}$ and
$\mathfrak{s}^2_{\alpha\beta\gamma}$ are two possible extensions. As
$\mathfrak{s}^i_{\alpha\beta\gamma}$ restricted on $\#^g(S^1\times S^2)$
is $\mathfrak{s}_0$, we can further extend $\mathfrak{s}^i_{\alpha\beta\gamma}$
to a $\text{Spin}^c$-structure
$\tilde{\mathfrak{s}}^i_{\alpha\beta\gamma}$ over $Y\times I$. By
previous disscussion, we know
$\tilde{\mathfrak{s}}^1_{\alpha\beta\gamma}$ and
$\tilde{\mathfrak{s}}^2_{\alpha\beta\gamma}$ have the same restriction
on $Y_{\alpha\beta}$ and $Y_{\alpha\gamma}$. This implies
$\tilde{\mathfrak{s}}^1_{\alpha\beta\gamma}-\tilde{\mathfrak{s}}^2_{\alpha\beta\gamma}$
 is $0\in
H^2(Y\times I)$. Hence the difference
$\mathfrak{s}^1_{\alpha\beta\gamma}-\mathfrak{s}^2_{\alpha\beta\gamma}$
is $0\in H^2(X_{\alpha\beta\gamma})$.

\begin{defn}\label{def:sadmissible}
  We suppose that the Heegaard triple $\mathcal{T}$ is a $2n$-pointed
  Heegaard triple subordinate to a band move from
  $(Y,L_{\alpha\beta})$ to $(Y,L_{\alpha\gamma})$. We also fix a
  $\text{Spin}^c$-structure $\mathfrak{s}_{\alpha\beta\gamma}$ over
  $X_{\alpha\beta\gamma}$ which comes from the restriction of a
  $\text{Spin}^c$-structure over
  $Y\times I$, we say that the Heegaard triple $\mathcal{T}$ is
  $\mathfrak{s}$-\textit{\textbf{realized}}, if 
 there exists
 points $\x\in \mathbb{T}_\alpha\cap\mathbb{T}_\beta$,  $\bs{\theta}\in
 \mathbb{T}_\beta\cap\mathbb{T}_\gamma$ and $\y\in
 \mathbb{T}_\alpha\cap\mathbb{T}_\gamma$ and a $n$-tuple of points
 $\mathbf{q}\in \Sigma \backslash(\bs{\alpha}\cup\bs{\beta})$, such
 that:
\begin{itemize}
\item there is a unique point $q_i$ in each component of
  $\Sigma\backslash\bs{\alpha}$ and $\Sigma\backslash\bs{\beta}$,
\item the map
  $\mathfrak{s}_q(\x)=\mathfrak{s}_{\delta}(\x)=\mathfrak{s}_{\alpha\beta}$,
  where $\mathfrak{s}_{\alpha\beta}$ is the restriction of
  $\mathfrak{s}_{\alpha\beta\gamma}$ on the boundary three-manifold
  $Y_{\alpha\beta}$. 
\item the map
  $\mathfrak{s}_q(\bs{\theta})=\mathfrak{s}_{\delta}(\bs{\theta})=\mathfrak{s}_{\alpha\beta}$, 
  where $\mathfrak{s}_{\beta\gamma}$ is the restriction of
  $\mathfrak{s}_{\alpha\beta\gamma}$ on the boundary three-manifold
  $Y_{\beta\gamma}$. 
\item  the map $\mathfrak{s}_q(\y)=\mathfrak{s}_{\delta}(\y)=\mathfrak{s}_{\beta\gamma}$,
  where $\mathfrak{s}_{\alpha\gamma}$ is the restriction of
  $\mathfrak{s}_{\alpha\beta\gamma}$ on the boundary three-manifold
  $Y_{\alpha\gamma}$. 
\end{itemize}
Furthermore, we say that an $\mathfrak{s}$-realized Heegaard triple $\cal{T}=(\Sigma,\bs{\alpha},\bs{\beta},\bs{\gamma},\mathbf{O})$ is \textit{\textbf{weakly \textnormal{$\mathfrak{s}$-}admissible}} (or
\textit{\textbf{strongly \textnormal{$\mathfrak{s}$-}admissible}}
resp.), if the pointed Heegaard triple
$(\Sigma,\bs{\alpha},\bs{\beta},\bs{\gamma},\mathbf{q})$ is weakly \textnormal{$\mathfrak{s}$-}admissible (or
strongly \textnormal{$\mathfrak{s}$-}admissible resp.).
\end{defn}

\begin{rem}\label{rem:reason}
  If $\mathcal{T}$ is $\mathfrak{s}$-realized, we claim that there is a unique
  class $[\Delta]\in S_{\alpha\beta\gamma}$ of trianlges with
  $\mathfrak{s}_{\bs{q}}([\Delta])=\mathfrak{s}_{\alpha\beta\gamma}$
  for some $n$-tuple of points $\bs{q}$ satisfying the condition in
  Definition \ref{def:sadmissible}. Otherwise, we suppose there is
  another class $[\Delta]'$ in $S_{\alpha\beta\gamma}$ with
  $\mathfrak{s}_{\bs{q}'}([\Delta]')=\mathfrak{s}_{\alpha\beta\gamma}$.
We know that $\mathfrak{s}_{\bs{q}'}([\Delta]')$ and
  $\mathfrak{s}_{\bs{q}}([\Delta]')$ has the same restriction on
  boundary and can be extend to a $\text{Spin}^c$-structure over
  $Y\times I$. Hence we have 
  $\mathfrak{s}_{\bs{q}}([\Delta]')=\mathfrak{s}_{\alpha\beta\gamma}$. As
  the map $\mathfrak{s}_{\bs{q}}$ is one-to-one, we get
  $[\Delta]=[\Delta]'$.

  If the Heegaard triple $\mathcal{T}$ is subordinate to a
  four-dimensional two-handle attachment, one may not find a
  well-defined map. Recall that if $(F,\partial F)$ is orientable, we can
  choose $\mathfrak{s}_{w}$ or $\mathfrak{s}_{z}$ as the one-to-one
  map from $S_{\alpha\beta\gamma}$ to
  $\text{Spin}^c(X_{\alpha\beta\gamma})$. If  $(F,\partial F)$ is
  non-orientable, we have no cannonical choice of $n$-tuple of
  basepoints $\mathbf{q}$ on $\Sigma$. The choice of $\mathbf{q}$
  may affects the map $\mathfrak{s}_{q}$, i.e. in some cases, there exist
  $\mathbf{q}$ and $\mathbf{q}'$, both of which satisfy the condition
  in Definition \ref{def:sadmissible}, but
  $\mathfrak{s}_{\mathbf{q}}\neq\mathfrak{s}_{\mathbf{q}'}$. If this
  happens, we will not know which class of triangle in
  $S_{\alpha\beta\gamma}$ should be associated to certain
  $\text{Spin}^c$-structure $\mathfrak{s}$.  
\end{rem}

 In light of the proof in \cite[Lemma 5.2]{Ozsvath2004c}, by finger moves of
  $\bs{\alpha},\bs{\beta}$ and $\bs{\gamma}$ curves along their dual
  curves on $\Sigma$ without crossing the basepoints $\mathbf{O}$, we
  can isotope a Heegaard triple $\mathcal{T}$ subordinate to a
  band move to a   
 weakly \textnormal{$\mathfrak{s}$-}admissible (or
strongly \textnormal{$\mathfrak{s}$-}admissible resp.) Heegaard
triple, where $\mathfrak{s}$ is a $\text{Spin}^c$-structure for
$Y\times I$. 

\begin{lem} \label{lem:admtri} 

  Let $\mathcal{T}$ be a strongly $\mathfrak{s}$-admissible Heegaard
  triple subordinate to a band move from $(Y,L_{\alpha\beta})$ to
  $(Y,L_{\alpha\gamma})$. If $\mathfrak{s}$ is torsion, we have a
   triangle chain map: 
\[f_{\alpha\beta\gamma}:CFL'(\cal{H}_{\alpha\beta},\mathfrak{s}_{\alpha\beta})\otimes
  CFL'(\cal{H}_{\beta\gamma},\mathfrak{s}_0)\rightarrow CFL'(\cal{H}_{\alpha\gamma},\mathfrak{s}_{\alpha\gamma})\]
whose restriction on generators are defined by:
\begin{equation}\label{eq:7}
  f_{\alpha\beta\gamma}(\x\otimes\bs{\theta};\mathfrak{s})\triangleq
  \sum_{\y\in\mathbb{T}_{\alpha}\cap\mathbb{T}_{\gamma}} \sum_{\{\psi\in
  \pi_{2}(\x,\bs{\theta},\y)|\mu(\psi)=0,\mathfrak{s}_{\mathbf{q}}(\psi)=\mathfrak{s}\}}
(\#\cal{M}(\psi))U^{n_{\mathbf{O}}(\psi)} \y. 
\end{equation} Here $\mathbf{q}$ is $n$-tuple of points being used to
define $\mathfrak{s}$-admissibility in Definition \ref{def:sadmissible}.
\end{lem}

\begin{proof}
Let $\psi,\psi'\in \pi_2(\x,\bs{\theta},\y)$, then the difference 
\[\psi-\psi'=\phi_{\alpha\beta}+\phi_{\beta\gamma}+\phi_{\alpha\gamma},\]
where $\phi_{\alpha\beta}$ is a class in $ \pi_2(\x,\x)$,
$\phi_{\beta\gamma}$ is a class in $ \pi_2(\bs{\theta},\bs{\theta})$,
$\phi_{\alpha\gamma}$ is a class in $ \pi_2(\y,\y)$.

As in the proof of Lemma \ref{sec:curved-chain-complex-1}, we have three equalities,
$n_{\mathbf{O}}(\phi_{\alpha\beta})=2n_{\mathbf{q}}(\phi_{\alpha\beta})
$, $n_{\mathbf{O}}(\phi_{\beta\gamma})=2n_{\mathbf{q}}(\phi_{\beta\gamma})$,
and
$n_{\mathbf{O}}(\phi_{\alpha\gamma})=2n_{\mathbf{q}}(\phi_{\alpha\gamma})$,
the finiteness of counting follows from the strongly $\mathfrak{s}$-admissibility for diagram
$\cal{T}_{\mathbf{q}}=(\Sigma,\bs{\alpha},\bs{\beta},\bs{\gamma},\mathbf{q})$.

\end{proof}

Furthermore, by tracking the proof of \cite[Theorem 8.16]{Ozsvath2004c}, we have
$CFL'$ flavor of associativity as described below.

\begin{lem}\label{sec:glob-textsp-assoc-1}

Given a strongly $\mathfrak{S}$-admissible Heegaard quadruple
$(\Sigma,\bs{\alpha},\bs{\beta},\bs{\gamma},\bs{\delta},\mathbf{O})$
with induced Heegaard triples $\cal{T}_{\alpha\beta\gamma}$ and
$\cal{T}_{\alpha\gamma\delta}$ satisfying the same condition in Lemma
\ref{lem:admtri}, then we have the following equality
\begin{align*}
  \sum_{\mathfrak{s}\in\mathfrak{S}}F_{\alpha\gamma\delta}(F_{\alpha\beta\gamma}(\theta_{\alpha\beta}\otimes
  \theta_{\beta\gamma};\mathfrak{s}_{\alpha\beta\gamma})\otimes\theta_{\gamma\delta};\mathfrak{s}_{\alpha\gamma\delta})\\
=\sum_{\mathfrak{s}\in\mathfrak{S}}F_{\alpha\beta\delta}(\theta_{\alpha\beta}\otimes
  F_{\beta\gamma\delta}(\theta_{\beta\gamma}\otimes\theta_{\gamma\delta};\mathfrak{s}_{\beta\gamma\delta});\mathfrak{s}_{\alpha\beta\delta}).
\end{align*}
Here $\theta_{\alpha\beta},\theta_{\beta\gamma}$ and
$\theta_{\gamma\delta}$ lie in
$HFL'(Y_{\alpha\beta}), HFL'(Y_{\beta\gamma})$ and
$HFL'(Y_{\gamma\delta})$ respectively, and $\mathfrak{S}$ is a $\delta
H^1(Y_{\beta\gamma})+\delta H^1(Y_{\alpha\gamma})$ orbit of a fixed
$\textnormal{Spin}^c$-structure over $X_{\alpha\beta\gamma\delta}$.
\end{lem}

\section{Band moves and triangle maps}\label{sec:band-moves-triangle}

\subsection{Assumptions}

In this section,
we assume that the bipartite link cobordisms
or bipartite disoriented link cobordisms satisfy the following conditions:
\begin{itemize}
\item The four-manifold $W$ is a product $Y\times I$.
\item The
inclusion $\cal{F}:(F,\partial F)\rightarrow (W,\partial W)$ induces a
trivial map $\cal{F}_*:H_2(F,\partial F)\rightarrow H_1(W,\partial W)$. Consequently, the
bipartite links in the boundary three-manifolds are null-homologous.
\end{itemize}

\begin{lem} Given a $\beta$-band move  $B^\beta$ from bipartite link
  $L_{\alpha\beta}$ in $Y$ to a bipartite link $L_{\alpha\gamma}$ in
  $Y$, there exists a strongly $\mathfrak{s}$-admissible Heegaard
  triple (or standard Heegaard triple) subordinate this band move $B^\beta$. Here
  $\mathfrak{s}$ is a $\textnormal{Spin}^c$-structure over
  $X_{\alpha\beta\gamma}$ such that the restriction
  $\mathfrak{s}_{\alpha\beta}$ on
  $Y_{\alpha\beta}\cong Y$ is equal to its restriction on
  $Y_{\alpha\gamma}\cong Y$.
\end{lem}

\begin{proof}
Recall that in the final step of the construction of standard Heegaard
triple in Theorem \ref{sec:exist-heeg-triple}, we can do finger moves for the
diagram $H'_{\alpha\beta}$ without crossing the basepoints $\mathbf{O}$ and away
from the disk neighborhood $D$, such that the diagram $H_{\alpha\beta}$ is strongly
$\mathfrak{s}$-admissible. Then the result follows.     
\end{proof}

As a corollary of Lemma \ref{sec:some-topol-facts-1} and
\cite[Proposition 7.2]{Ozsvath2004c}, we have the following lemma.

\begin{lem}
  Suppose   $\cal{T}_1$ and $\cal{T}_2$ are two strongly
  $\mathfrak{s}$-admissible Heegaard triples 
 subordinate to the same band move $B^\beta$ from a bipartite link
 $L_{\alpha\beta}$ to a bipartite link $L_{\alpha\gamma}$. Then these two triples
 can be connected by a sequence moves in Lemma
 \ref{sec:some-topol-facts-1} , such that in each intermediate step, Heegaard
 triple is strongly $\mathfrak{s}$-admissible.
\end{lem}

Suppose the bipartite disoriented link cobordism
$(\cal{W},\cal{F},F_\alpha,F_\beta,\cal{A},\cal{A}_\Sigma)$ is a
$B^\beta$-band move next to basepoints $(o_i,o_j)$ (cf. Figure \ref{fig:twobandmoves}). 
From the discussion in Section \ref{sec:relat-betw-three}, we know that
it determines a unique bipartite link cobordism
$(\cal{W},\cal{F},F_\alpha,F_\beta,\cal{A}_\Sigma)$. By Theorem
\ref{sec:exist-heeg-triple}, there exists a Heegaard triple $\cal{T}$ subordinate to
this bipartite link cobordism. As the bipartite link $L_{\alpha\beta}$
and $L_{\alpha\gamma}$ in $Y_{\alpha\beta}\cong Y$ and
$Y_{\alpha\gamma}\cong Y$ are null-homologous, we can associate them with two
$\mathbb{Z}$-graded curved chain complex
$CFBL^-(\cal{H}_{\alpha\beta},\mathfrak{s})$ and
$CFBL^-(\cal{H}_{\alpha\gamma},\mathfrak{s})$ respectively, where $\mathfrak{s}$ is
a torsion $\text{Spin}^c$-structure of $ Y$. 
For the bipartite link
$L_{\beta\gamma}$ in $Y_{\beta\gamma}\cong \#^n(S^1\times S^2)$, by
Corollary \ref{sec:textsp-struct-delta}, we have also has a
$\mathbb{Z}$-graded curved chain complex
$CFBL^-(\cal{H}_{\beta\gamma},\mathfrak{s}_0)$.

\subsection{Distinguishing the top grading generators}

Let $B^\beta$ be a band move from bipartite link $L_{\alpha\beta}$ to $L_{\alpha\gamma}$.
 If the surface  $(F,\partial F)$ is orientable, for the diagram
  $(\Sigma,\bs{\beta},\bs{\gamma},\mathbf{w},\mathbf{z})$, the top
  $\delta$-grading generators of $CFL'(\mathcal{H}_{\beta\gamma},\mathfrak{s}_0)$ can be distinguished by using
  $gr_{\mathbf{w}}$ and $gr_{\mathbf{z}}$. Clearly, the choice of
  generators depends on the coloring data. See
  \cite[Section 6]{Zemke2016a} for details. 

\begin{figure}
  \centering
  \includegraphics[scale=1.2]{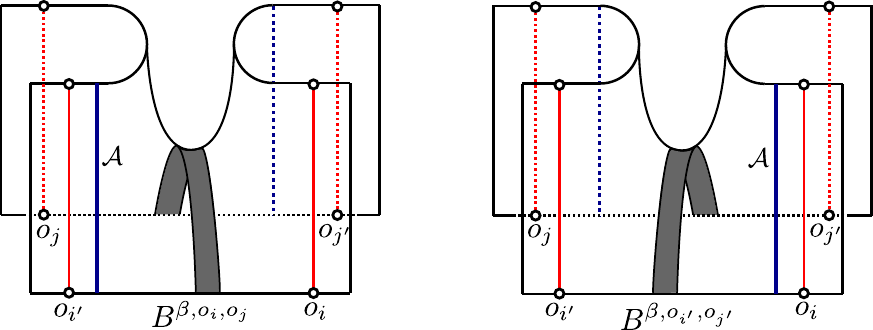}
  \caption{Two band moves near different basepoints}
  \label{fig:twobandmoves}
\end{figure}

Generically, as the surface $(F,\partial F)$ can be non-orientable,
we do not have $gr_w$ or $gr_z$-grading. Therefore we
need extra data to distinguish the top
grading generators in the homology
$HFL'(\cal{H_{\beta\gamma}},\mathfrak{s}_0)$. Actually, the extra  data we need are included in
$\cal{A}$.
In fact, the band $B^\beta$ can either be next to a pair of
basepoints $(o_i,o_j)$ or the other pair $(o_{i'},o_{j'})$ (See Figure \ref{fig:twobandmoves}).
Here, we denoted a $\beta$-band by $B^{\beta,o_i,o_j}$ if it is
next to the pair of basepoints $(o_i,o_j)$.

  Recall that, for the construction of Heegaard triple in Theorem
  \ref{sec:exist-heeg-triple}, we set the band $B^\beta$ being next to a pair of
  basepoints $(o_i,o_j)$ of $L_{\alpha\beta}$. If $B^\beta$ is of type
  II, we can choose
  the other pair $(o_{i'},o_{j'})$ and construct a triple $\cal{T}'$.
  Then both of the triples $\cal{T}$ are subordinate to the band move
  $B^\beta$, and can be connected by Heegaard moves in Lemma
 \ref{sec:some-topol-facts-1}.

\begin{lem}\label{sec:dist-top-grad}
  Given a band move $B^\beta_{o_i,o_j}$ of a bipartite disoriented
  link and a strongly $\mathfrak{s}$-admissible Heegaard triple $\cal{T}$ subordinate to it, there is a chain complex
  $CFL'_{o_i,o_j}(\cal{H}_{\beta\gamma},\mathfrak{s}_0)$, such that
  the top $\delta$-grading $\mathbb{F}_2$-submodule of its homology is one-dimensional..
\end{lem}

\begin{proof}\label{sec:dist-top-grad-1}
  We define the module 
\[ CFL'_{o_i,o_j}(\cal{H}_{\beta\gamma},\mathfrak{s}_0)=
CFBL(\cal{H}_{\beta\gamma},\mathfrak{s}_0)\otimes_{\mathbb{F}_2[U_1,\cdots,U_{2n}]}
\mathbb{F}_2[U,U_1,\cdots,U_{2n}]/I,\]
where $I$ is an ideal generated by $(U-U_k),k\neq i,j$. By \cite[Lemma
2.1]{Zemke2016a}, the endomorphism $\partial$ is a
differential. Without loss of generality, if the triple is standard,
we have 
\begin{align*}
  &\partial \bs{\theta} = (U_i+U_j)\bs{\theta}'\\
&\partial \bs{\theta'} = 0,
\end{align*}
where $\bs{\theta}$ and $\bs{\theta'}$ are two top grading generators of $CFL_{o_i,o_j}'(\cal{H}_{\beta\gamma},\mathfrak{s}_0)$
or $CFBL_{o_i,o_j}'(\cal{H}_{\beta\gamma},\mathfrak{s}_0)$.
\end{proof}

 Let $\bs{\theta}^{o_i,o_j}$ be the generator $\bs{\theta}$ in the proof
 of Lemma \ref{sec:dist-top-grad-1}. We have the following theorem.

\begin{thm} \label{sec:bipart-disor-link}

Given a $B^\beta$-band move next to a pair of basepoints
  $(o_i,o_j)$ of a bipartite disoriented link $L_{\alpha\beta}$ in
  $Y$,  and  a
torsion $\textnormal{Spin}^c$-structure  $\mathfrak{s}$ of $Y$,
we can construct a $\mathbb{Z}$-filtered chain map
 \[\sigma^{o_i,o_j}:CFL'(Y,L_{\alpha\beta},\mathfrak{s})\rightarrow
 CFL'(Y,L_{\alpha\gamma},\mathfrak{s}),\] 
which is well-defined up to $\mathbb{Z}$-filtered chain homotopy. This
construction is independent of the
choices of Heegaard triples. Therefore, it gives rise to a map on homology:
 \[\sigma_*^{o_i,o_j}:HFL'(Y,L_{\alpha\beta},\mathfrak{s})\rightarrow
 HFL'(Y,L_{\alpha\gamma},\mathfrak{s}),\] which is an invariant of
this bipartite disoriented link cobordism. 
   
\end{thm}

\begin{proof}
We define the chain map $\sigma^{o_i,o_j}:
CFL'(\cal{H}_{\alpha\beta},\mathfrak{s})\rightarrow
CFL'(\cal{H}_{\alpha\gamma},\mathfrak{s})$ by 
\[\sigma^{o_i,o_j}=f_{\alpha\beta\gamma}(\x\otimes\bs{\theta}^{o_i,o_j};\mathfrak{s}_{\alpha\beta\gamma}).\]
Here $\bs{\theta}^{o_i,o_j}$ is the generator defined in the proof of Lemma
\ref{sec:dist-top-grad}, and the $\text{Spin}^c$-structure
$\mathfrak{s}_{\alpha\beta\gamma}$ restricts on $Y$ and on
$\#^g(S^1\times S^2)$ are $\mathfrak{s}$ and $\mathfrak{s}_0$ respectively.
 
Similar to the proof of \cite[Theorem 6.9]{Juhasz2016}, using the
associativity which we proved in Lemma
\ref{sec:glob-textsp-assoc-1}, the map $\sigma^{o_i,o_j}$ is
well-defined up to $\delta$-filtered chain homotopy, i.e. it is
independent of the choices of Heegaard triple $\cal{T}$. Thus, it gives
rise to a  homomorphism on homology level.

\end{proof}

\subsection{The relation of generators}

Suppose $B^{\beta,o_i,o_j}$ is a band move from bipartite disoriented
link $(\cal{L}_0,\mathbf{O})$ to bipartite disoriented link
$(\cal{L}_1,\mathbf{O})$. Then there exists an inverse band move $(B^{\beta,o_i,o_j})^{-1}$ from
$(\cal{L}_1,\mathbf{O})$ to $(\cal{L}_0,\mathbf{O})$. See Figure
\ref{fig:rogandcob} for an example of a composition of two type II
band move with surface $(F,\partial F)$ orientable.

\begin{figure}
  \centering
  \includegraphics[scale=0.5]{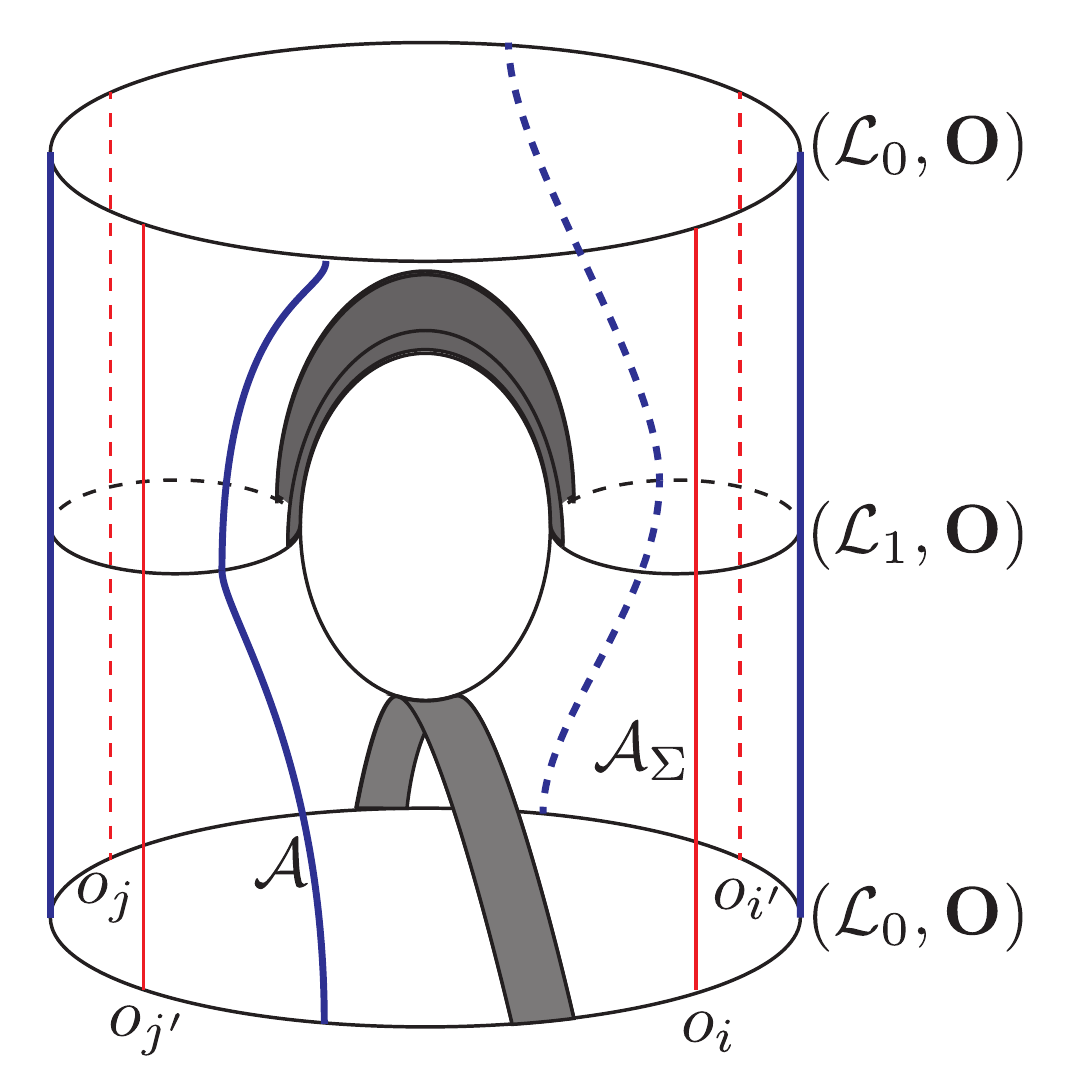}
  \caption{Composition of $B^{\beta,o_i,o_j}$ and its inverse.}
  \label{fig:rogandcob}
\end{figure}

Let $\cal{T}_{\alpha\beta\gamma}=(\Sigma,\bs{\alpha},\bs{\beta},\bs{\gamma},\mathbf{O})$ be a strongly $\mathfrak{s}$-admissible Heegaard triple
subordinate to the band move $B^{\beta,o_i,o_j}$. We set the curves $\bs{\delta}$ to
be  small Hamiltonian isotopies of the curves $\bs{\alpha}$. Then the triple
$\cal{T}_{\alpha\gamma\delta}=(\Sigma,\bs{\alpha},\bs{\gamma},\bs{\delta},\mathbf{O})$ is
subordinate to the inverse band move
$(B^{\beta,o_i,o_j})^{-1}$. We have the following lemma for the induced Heegaard triple
$\cal{T}_{\beta\gamma\delta}=(\Sigma,\bs{\beta},\bs{\gamma},\bs{\delta},\mathbf{O})$.

\begin{lem}\label{sec:relation-generators}
  Suppose
  $(\Sigma,\bs{\alpha},\bs{\beta},\bs{\gamma},\bs{\delta},\mathbf{O})$
  is a strongly $\mathfrak{s}$-admissible Heegaard quadruple, with the
  induced Heegaard triples $\cal{T}_{\alpha\beta\gamma}$ subordinate
  to $B^{\beta,o_i,o_j}$ and $\cal{T}_{\alpha\gamma\delta}$
  subordinate to $(B^{\beta,o_i,o_j})^{-1}$. For the top grading
  generators $\Theta_{\beta\gamma}\in HFL'(H_{\beta\gamma})$,
  $\Theta_{\beta\gamma}\in HFL'(H_{\beta\delta})$ and
  $\Theta_{\gamma\delta}\in HFL'(H_{\gamma\delta})$, we have the
  following relation:
\begin{equation}\label{eq:8}
f_{\beta\gamma\delta}(\Theta_{\beta\gamma}\otimes\Theta_{\gamma\delta})
  = U\cdot \Theta_{\beta\delta}
\end{equation}

\end{lem}

\begin{figure}
  \centering
  \includegraphics[scale=0.7]{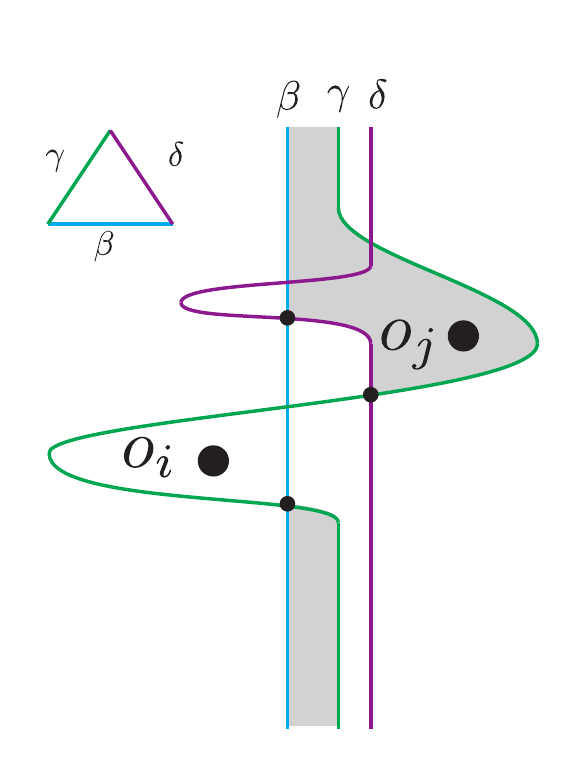}
  \caption{Local diagram for Heegaard triple $\cal{T}_{\beta\gamma\delta}$}
  \label{fig:bandmap10}
\end{figure}

\begin{proof}
  Without loss of generality, we assume the Heegaard triples
  $\cal{T}_{\beta\gamma\delta}$ is a standard Heegaard triple as shown in Figure
  \ref{fig:bandmap10}. The three small black dots are in the
  generators $\Theta_{\beta\gamma}$, $\Theta_{\beta\delta}$ and
  $\Theta_{\gamma\delta}$ respectively. The shaded area in Figure
  \ref{fig:bandmap10} represents a triangle $\Delta$ in
  $\pi_2(\Theta_{\beta\gamma},\Theta_{\gamma\delta},\Theta_{\beta\delta})$
  with Maslov index equal to zero. Similar to the arguments in
  \cite[Section 9]{Ozsvath2004c}, by checking the grading of those
  generators and periodic
  domains, there is no other triangle with two vertices
  $\Theta_{\beta\gamma},\Theta_{\gamma\delta}$ and Maslov index
  zero. As $n_{\mathbf{O}}(\Delta)=1$, we get a $U$ before
  $\Theta_{\beta\delta}$ in Equation \ref{eq:8}. 
\end{proof}

\begin{lem}\label{sec:relation-generators-1} Let $\sigma^{o_i,o_j}$ be the
  chain map induced by $B^{\beta,o_i,o_j}$ defined in Theorem \ref{sec:bipart-disor-link}. Then,
  there exists a  chain map induced by the
  inverse of $B^\beta$,
\[\tau^{o_i,o_j}:CFL'(Y,L_{\alpha\gamma},\mathfrak{s})\rightarrow
 CFL'(Y,L_{\alpha\beta},\mathfrak{s}),\]
such that

\begin{itemize}
\item the composition 
$\tau^{o_i,o_j}\circ\sigma^{o_i,o_j}$ is
chain homotopic to:
\[U:CFL'(Y,L_{\alpha\beta},\mathfrak{s})\rightarrow
  CFL'(Y,L_{\alpha\beta},\mathfrak{s})\] 
\item the composition 
$\sigma^{o_i,o_j}\circ\tau^{o_i,o_j}$ is
chain homotopic to:
\[U:CFL'(Y,L_{\alpha\gamma},\mathfrak{s})\rightarrow
  CFL'(Y,L_{\alpha\gamma},\mathfrak{s})\] 
\end{itemize}

\end{lem}

\begin{proof}

Consider a strongly $\mathfrak{s}$-admissible Heegaard quadruple
constructed from the composition $(B^{\beta,o_i,o_j})^{-1}\circ
B^{\beta,o_i,o_j}$ of cobordisms as in Lemma
\ref{sec:relation-generators}. The composition is cobordism from
$L_{\alpha\gamma}$ to itself.
The map $\sigma^{o_i,o_j}$ is defined to be
$f_{\alpha\beta\gamma}(\cdot\otimes \Theta_{\beta\gamma})$, where
$\Theta_{\beta\gamma}$ is the generator defined in Lemma
\ref{sec:dist-top-grad}. Similarly, we define $\tau^{o_i,o_j}$ to be
$f_{\alpha\gamma\delta}(\cdot\otimes \Theta_{\gamma\delta})$, where
$\Theta_{\beta\delta}$ is also the generator defined in Lemma
\ref{sec:dist-top-grad}.

Then the composition, 

\begin{align*}
  \tau^{o_i,o_j}\circ\sigma^{o_i,o_j}&=f_{\alpha\gamma\delta}(f_{\alpha\beta\gamma}(\cdot\otimes
  \Theta_{\beta\gamma})\otimes \Theta_{\gamma\delta})\\
&=f_{\alpha\beta\delta}(\cdot \otimes
  (f_{\beta\gamma\delta}(\Theta_{\beta\gamma}\otimes\Theta_{\gamma\delta})))\text{ 
  (by Lemma \ref{sec:glob-textsp-assoc-1}})\\ 
&=f_{\alpha\beta\delta}(\cdot \otimes U\cdot \Theta_{\beta\delta}
  )\text{ (by Lemma \ref{sec:relation-generators})}\\
&= Uf_{\alpha\beta\delta}(\cdot\otimes\Theta_{\beta\delta}) 
\end{align*}

From the construction of the quadruple, we know that the curves $\bs{\delta}$ are just small Hamiltonian isotopy of
$\mathbf{\beta}$ without crossing any basepoints.
As the small triangle map $f_{\alpha\beta\delta}(\cdot\otimes\Theta_{\beta\delta})$ is chain homotopic to the nearest point map, the composition $\tau^{o_i,o_j}\circ\sigma^{o_i,o_j}$ is
actually $\mathbb{Z}$-filtered chain homotopic to the map $U$.
\end{proof}

\begin{rem}
  
If the band
$B^{\beta,o_i,o_j}$ is of type II, we can construct the other composition of cobordism $(B^{\beta,o_{i'},o_{j}'})^{-1}\circ
B^{\beta,o_i,o_j}$, where the basepoints $o_i$ and $o_{i'}$ are connected by a
component of $L_\beta$, the basepoints $o_i$ and $o_{i'}$  are connected by another
component of $L_\beta$. See Figure \ref{fig:rogandcob-03} for an
example of the above composition of cobordisms with surface
$(F,\partial F)$ orientable. Then this cobordism induces a map: 
\[
  \tau^{o_{i'},o_{j'}}\circ\sigma^{o_i,o_j}=f_{\alpha\beta\delta}(\cdot\otimes\Theta'_{\beta\delta}).\]
Here the generator $\Theta'_{\beta\delta}$, which has the same
$\delta$-grading as 
$U\Theta_{\beta\delta}$, is shown in Figure \ref{fig:bandmap-12}.

\end{rem}

\begin{figure}
  \centering
  \includegraphics[scale=0.5]{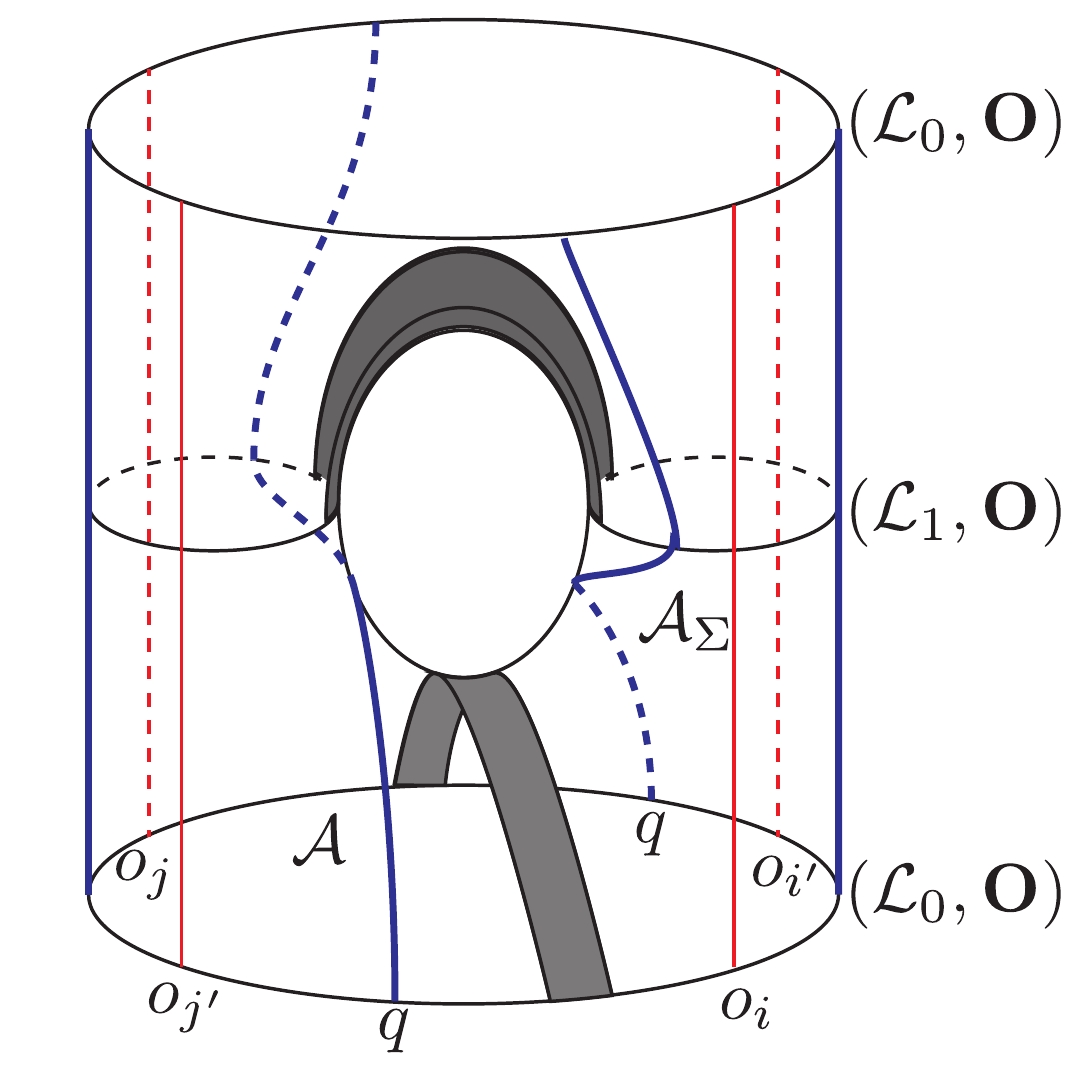}
  \caption{The composition of cobordisms $(B^{\beta,o_{i'},o_{j}'})^{-1}\circ
B^{\beta,o_i,o_j}$  }
  \label{fig:rogandcob-03}

\end{figure}
\begin{figure}
  \centering
  \includegraphics[scale=0.7]{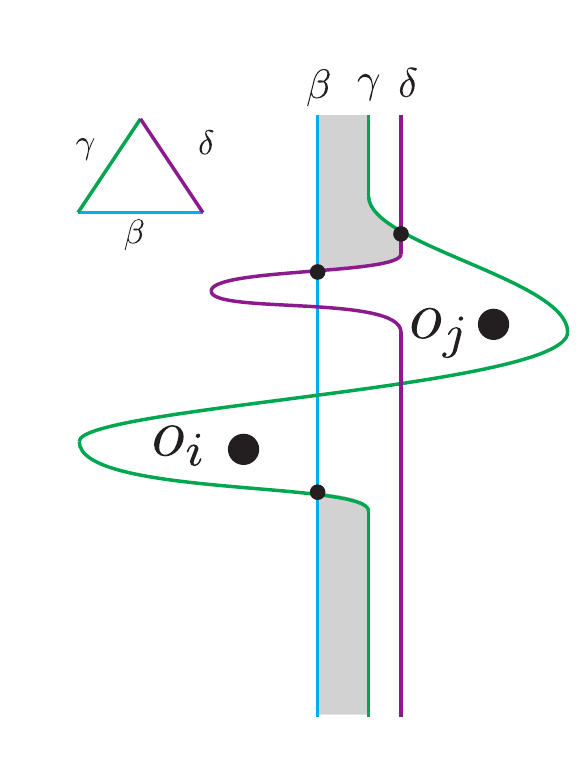}
  \caption{The composition of cobordisms $(B^{\beta,o_{i'},o_{j}'})^{-1}\circ
B^{\beta,o_i,o_j}$  }
  \label{fig:bandmap-12}
\end{figure}

\subsection{A comparison with  Ozsv\'{a}th, Stipsicz
and Szab\'{o}'s definition of band move maps}\label{sec:comp-with-ozsv}

\begin{figure}
  \centering
  \includegraphics[scale=0.7]{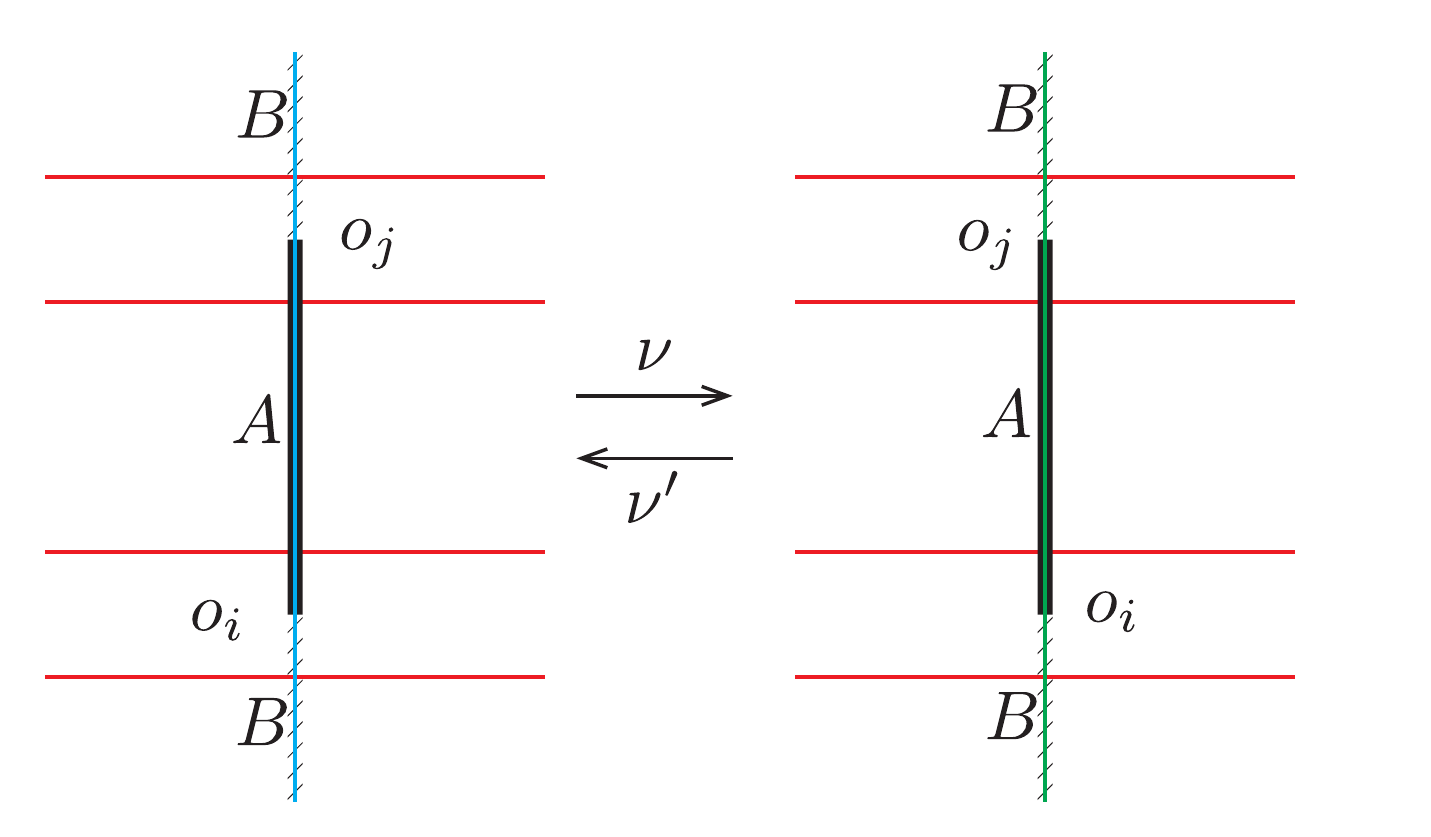}
  \caption{The standard grid move for the band move
    $F_{B^{\beta,o_i,o_j}}$}\label{Fig:gridtriangle1} 
\end{figure}

\begin{figure}
  \centering
  \includegraphics[scale=0.7]{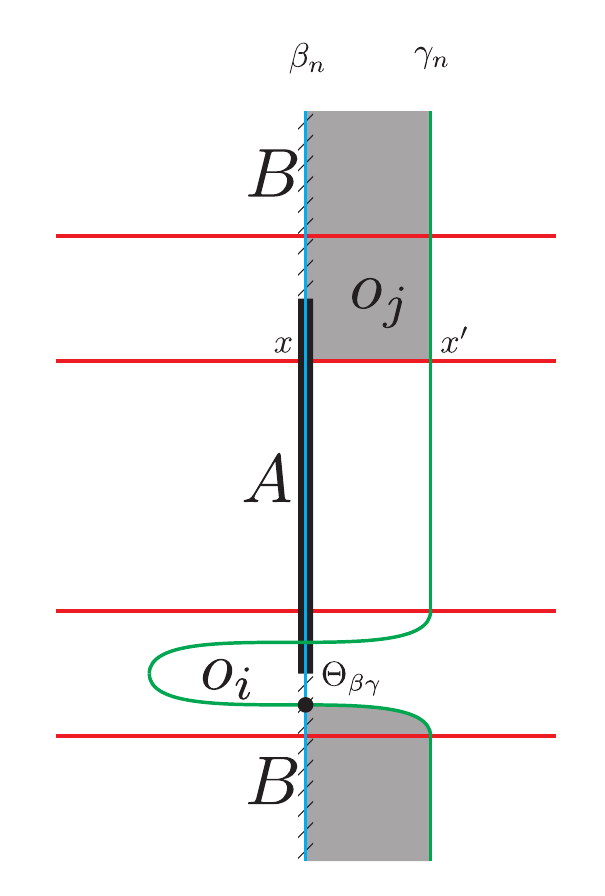}
  \caption{The standard Heegaard triple for the band move
    $F_{B^{\beta,o_i,o_j}}$}\label{Fig:gridtriangle2} 
\end{figure}

In this subsection, we compare our band move maps
$F_{B^{\beta,o_i,o_j}}$ with the band move maps defined by  Ozsv\'{a}th, Stipsicz
and Szab\'{o} in \cite{Ozsvath2015}.

\begin{defn}
  We say that a band $B$ for links (or bipartite links, bipartite
  disoriented links resp.) is an \textit{\textbf{oriented band}}, if the number of link
  components changes after the band move $B$. Otherwise, we say that
  $B$ is an \textit{\textbf{unoriented band}}.
\end{defn}

Recall that in  \cite{Ozsvath2015}, for a band move from a link in
$S^3$ to a link in $S^3$, they construct a standard grid move as
shown in Figure \ref{Fig:gridtriangle1}. The band move is represented by switching the
markings in a grid diagram. Furthermore, we can require that switching
corresponds to the unoriented resolution of a positive crossing.

To establish the relation between the standard grid move above and the
Heegaard triple subordinate to a band move, we need the following
lemma.

\begin{lem}
  Let $B^{\beta,o_i,o_j}$ be a band move from a bipartite disoriented
  link $(\cal{L},\mathbf{O})$ to $(\cal{L}',\mathbf{O})$. After a
  sequence of quasi-stabilizations and Heegaard moves without changing
  or crossing the basepoints, we can find a standard Heegaard triple
  $\cal{T}_{\alpha\beta\gamma}$ such that:
\begin{itemize}
\item the induced Heegaard diagram
  $\cal{H}_{\alpha\beta},\cal{H}_{\alpha\gamma}$ are grid diagrams.
\item the grid moves switching the markings $o_i$ and $o_j$
  corresponds to a resolution of a positive crossing, as shown in
  Figure \ref{Fig:gridtriangle1}.
\end{itemize}
\end{lem}

\begin{proof}
Recall that in Theorem \ref{sec:exist-heeg-triple}, we can require the
moves $\cal{H}_{\alpha\beta}$ to $\cal{H}_{\alpha\gamma}$ be as shown in
the bottom right of Figure \ref{fig:st1}. The remaining part of the proof
follows from Lemma \ref{sec:some-topol-facts-1}.
\end{proof}

We call the grid move in Figure \ref{Fig:gridtriangle1} the \textit{\textbf{standard grid move}}
representing $B^{\beta,o_i,o_j}$.
Recall that  in  \cite{Ozsvath2015}, the band move maps are defined as
follows. Let $A$ and $B$ be the subset of the curve $\beta_n$ as shown
in Figure \ref{Fig:gridtriangle1}. For a grid state $\x\in
\mathbb{T}_{\alpha}\cap\mathbb{T}_{\beta}$, the map $\nu(x)$ is 
\[\nu(x)=\begin{cases}
U\cdot \x &\text{if }\x\cap A\neq \emptyset\\
\x &\text{if }\x\cap A = \emptyset.
\end{cases}\]

For our definition, the band map $F_{B^{\beta,o_i,o_j}}(\cdot)$ is
defined to be the triangle map 
$F_{\alpha\beta\gamma}(\cdot \otimes \Theta_{\beta\gamma})$, where
$\Theta_{\beta\gamma}$ is the top grading generator determined by
$B^{\beta,o_i,o_j}$ as shown in Figure \ref{Fig:gridtriangle2}. For a intersection
$x\in\beta_n\cap\bs{\alpha}$, we let $x'$ be the closet intersection
to $x$ in $\gamma_n\cap\bs{\alpha}$. We have the following observation:
\begin{itemize}
\item if $x\in A$, the small triangle with endpoints
  $x,x',\Theta_{\beta\gamma}$ intersects the basepoint set $\mathbf{O}$
  once;
\item  if $x\notin A$, the small triangle with endpoints
  $x,x',\Theta_{\beta\gamma}$ does not intersect any basepoints.
\end{itemize}
 
Based on these observations, we have the following proposition.

\begin{prop}\label{thm:osseq}
  For a standard grid move representing an unoriented band $B=B^{\beta,o_i,o_j}$, we have
  the following:
\begin{itemize}
\item The map $\nu$ defined in \cite{Ozsvath2015} agrees with the
  triangle map $F_{B}$.
\item  The map $\nu'$ defined in \cite{Ozsvath2015} agrees with the
  triangle map $F_{B^{-1}}$. 
\end{itemize}
\end{prop}

\begin{proof}
  As $\gamma_n$ is a small isotopy of $\beta_n$ crossing basepoints
  $o_i,o_j$, we can identify
the map $\nu$ with the continuation map induced by the Hamiltonian isotopy.
This follows from the proof of  \cite[Theorem 6.6]{Ozsvath2010}. The
only difference is that, in our case, the continuation map may cross
the basepoints set $\mathbf{O}$, and we keep track of that with
variable $U$.
On the other hand, the equivalence between the continuation map and
triangle maps in \cite[Proposition 11.4]{Lipshitz2006} also works in
our case. 

Therefore, based on these two results and the observations above, we
conclude that the map $\nu$ is filtered chain homotopic to the
triangle map $F_{B}$. The equivalence between $\nu'$ and $F_{B^{-1}}$ is similar.  
\end{proof}

\begin{rem}
  In a similar vein, we can also identify the oriented band move maps
  $\sigma$ and $\mu$ in \cite{Ozsvath2015} with $F_{B}$ and
  $F_{B^{-1}}$. Furthermore, the formulas
  $\nu'\circ\nu=U,\nu\circ\nu'=U$ in \cite[Proposition
  5.7]{Ozsvath2015} 
(and similarly, the
formulas $\sigma\circ\mu =U$ and $\sigma\circ\mu =U$ in
\cite[Proposition 5.1]{Ozsvath2015}) agree with the formulas in Lemma 
\ref{sec:relation-generators-1}.
\end{rem}

\subsection{A comparison with Manolescu's definition of unoriented band move maps}
Recall that, some band maps were constructed in the proofs of the unoriented skein
exact triangle in \cite{Manolescu2007a} using special Heegaard
diagrams.
In the
construction, the band move is actually a type-II $\beta$-band
move, and the Heegaard triple $\cal{T}_{\alpha\beta\gamma}$ is
standard. The Heegaard triple shown in \cite[Figure 6]{Manolescu2007a}
matches the Figure \ref{fig:simpletype2}.

Let $\Theta,\Theta'$ be the two top $\delta$-grading generator of
$\mathbb{T}_{\alpha}\cap\mathbb{T}_{\beta}$. In \cite{Manolescu2007a}, the map
$f_0:\widehat{CFL}(L_0)\rightarrow \widehat{CFL}(L_1)$ 
is defined by the triangle map:
\begin{equation}
f_0(\cdot)=f_{\alpha\beta\gamma}(\cdot
\otimes(\Theta+\Theta')).
\end{equation}

 The map $f_0$ induces a map $f_0'$ on the unoriented link Floer chain
complex. In fact, $f_0':CFL'(L_0)\rightarrow CFL'(L_1)$ is given by:
\begin{equation}\label{eq:9}
f_0'(\cdot)=f_{\alpha\beta\gamma}(\cdot
\otimes(\Theta+\Theta')).
\end{equation}

\begin{prop}
Suppose a type II $\beta$-band move $B^\beta$ is next to either
the pair of basepoints $(o_i,o_j)$ or the pair of basepoints
$(o_{i'},o_{j'})$. For the band map defined in \ref{eq:9},  we have:
\begin{equation*}
f_0'= F_{B^{\beta,o_i,o_j}}+F_{B^{\beta,o_{i'},o_{j'}}}.
\end{equation*}Here
$B^{\beta,o_i,o_j}$ and $B^{\beta,o_{i'},o_{j'}}$ are the two band
moves for the bipartite disoriented links lifted from $B^\beta$. 
\end{prop}
\begin{proof}
Immediate from the definition.
\end{proof}

\begin{rem}
Notice that, the definition of $f_0'$ does not depend on the
one-manifold $\cal{A}$ on the bipartite disoriented link
cobordism. Therefore, the map $f_0$ is a well-defined map for band
moves $B^\beta$ between bipartite links. 
\end{rem}

\subsection{An example: a bipartite link cobordism from trefoil to unknot}

Let $\cal T$ shown in Figure \ref{fig:tretriple} be a Heegaard triple subordinate a type I $\beta$-band move.
The Heegaard diagram $H_{\alpha\beta}$ is compatible with the trefoil $K_0\subset S^3$;
the diagram $H_{\beta\gamma}$ is compatible with a homologically even bipartite link $K_1\subset S^1\times S^2$; the
diagram $H_{\alpha\gamma}$ is compatible with the unknot $\mathbb{U}$ in $S^3$.

Let $a,a',b,b',c,c'$ be the intersections shown in
Figure \ref{fig:trefoilgenerator-09}, which is a local diagram of Figure
\ref{fig:tretriple}. 

\begin{figure}
  \centering
  \includegraphics[scale=0.7]{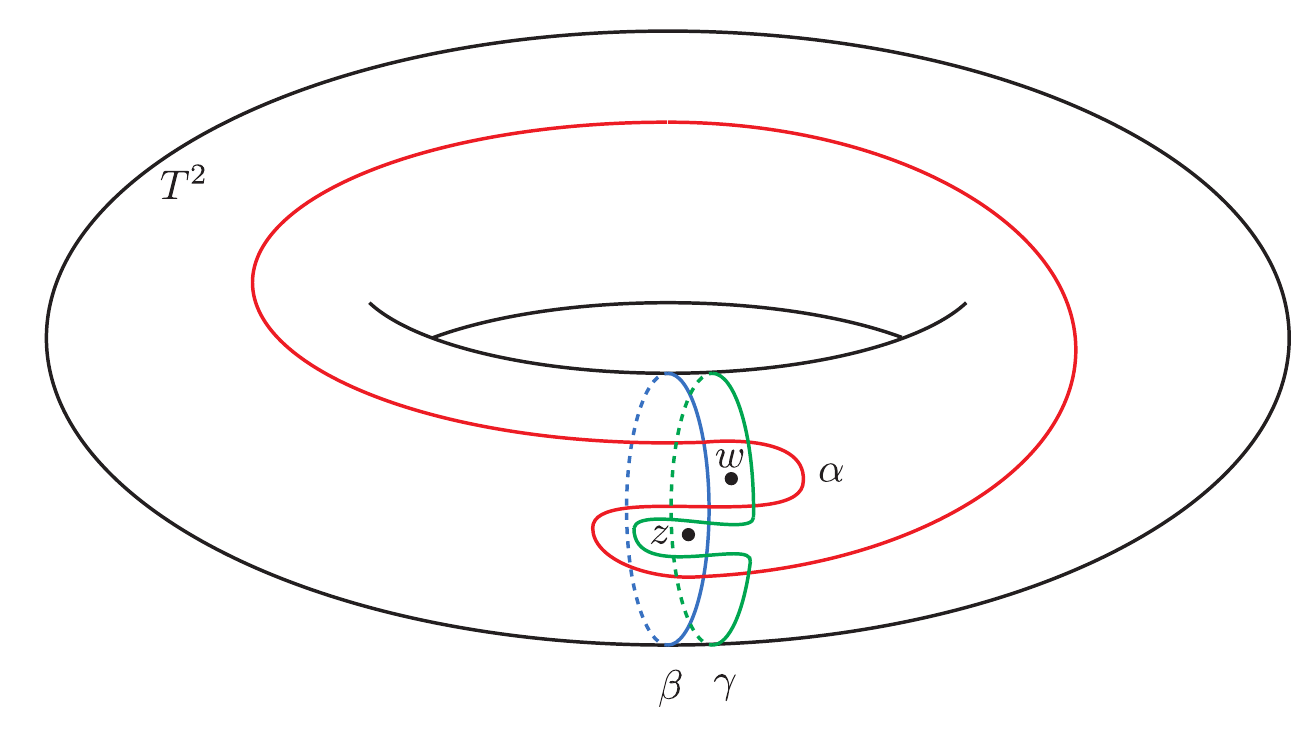}
  \caption{A Heegaard triple subordinate to a saddle from trefoil to unknot }
  \label{fig:tretriple}
\end{figure}

\begin{figure}
  \centering
  \includegraphics[scale=0.7]{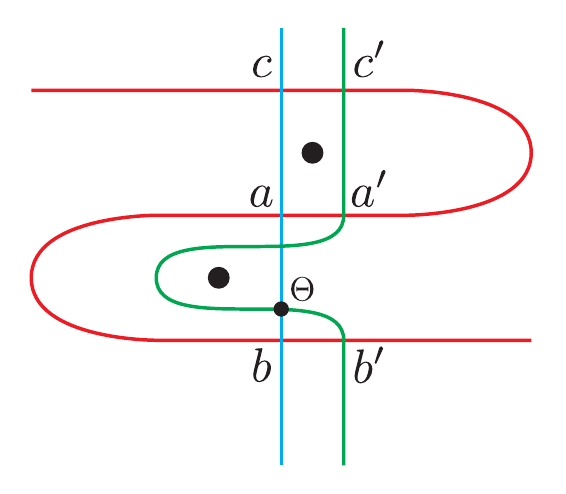}
  \caption{Local diagram for a Heegaard triple}
  \label{fig:trefoilgenerator-09}
\end{figure}

For the chain complex $CFL'(H_{\alpha\beta})$, we have:
\[\partial a = U(b+c)\text{ and } \partial b =\partial c=0.\]
For the chain complex $CFL'(H_{\alpha\gamma})$, we have:
\[\partial a' = (b'+c')\text{ and } \partial b =\partial c=0.\]
It is easy to check that, the cobordism map $\sigma$ acting on $a,b,c$
are given by the following:
\[\sigma(a)=U a',\sigma(b)=b'\text{ and }\sigma(c)=c'.\] Furthermore,
by computation, the map $\tau:CFL'(\mathbb{U}) \to
CFL'(K_0)$ defined in  Lemma \ref{sec:relation-generators-1} sends $a'$ to $a$, $b'$ to
$Ub$ and $c'$ to $Uc$. Therefore, the composition $\sigma\circ\tau$
(or $\tau\circ\sigma$ resp.) is exactly the map $U$.


\section{Quasi-stabilizations and Disk-stabilizations}\label{sec:quasi-stabilization}

In this subsection, we recall some facts about quasi-stabilization from
\cite{Manolescu2010}, \cite{Zemke2016} and \cite{Zemke2016a}.
For convenience, we focus on $\beta$-quasi-stabilizations. Similar
results hold for
$\alpha$-quasi-stabilizations.

\subsection{Topological facts about quasi-stabilizations}\label{sec:topol-facts-about}

Recall that in Section \ref{sec:categ-3:-bipart}, a quasi-stabilization $\mathfrak{W}_{qs}$ from bipartite disoriented link
$(\cal{L}_0,\mathbf{O}_0)$ to a bipartite disoriented link
$(\cal{L}_1,\mathbf{O}_1)$ is an elementary bipartite link cobordism
such that:
\begin{itemize}
\item The bipartite disoriented links $(\cal{L}_0,\mathbf{O}_0)$ and
  $(\cal{L}_1,\mathbf{O}_1)$ determine the same link $L$ in $Y$.
\item The dividing set $(\mathbf{p}_1,\mathbf{q}_1)$ of $\cal{L}_1$ is
  the union $(\mathbf{p}_0\cup p_s,\mathbf{q}\cup q_s)$, where $
  (\mathbf{p}_0,\mathbf{q}_0)$ is the dividing set of $\cal{L}_0$. 
\item The basepoints set $\mathbf{O}_1$ is the union $\mathbf{O}_0\cup \{o,o'\}$.
  
\end{itemize}

We denote a quasi-stabilization by $S_{+}^{\beta,o_i}$ (or in short, by  $S^{\beta,o_i}$), if
the four new points $(q_s,o,p_s,o')$ appear in order on $L$
and are between
the point $(o_i,q_i)$ in $U_\beta$.

Starting from a Morse function $f_0$ which is compatible with
$(\cal{L}_0,\mathbf{O}_0)$, we can modify $f_0$ inside a three ball
$D\subset Y$ to get a Morse function $f_1$ compatible with
$(\cal{L}_1,\mathbf{O}_1)$. 
In fact, we can pick up a three-ball $D$ 
inside $Y\backslash \text{Crit}(f)$, such that:
\begin{itemize}
\item The three ball $D$ contains the arc connecting $(q_s,o')$.
\item The intersection $D\cap \Sigma$ is a disk contained in $\Sigma\backslash (\bs{\alpha}\cup\bs{\beta})$.
\end{itemize}

We add a pair of index zero and three critical points at $p_s$ and
$q_s$ respectively. Then we replace the disk $D\cap \Sigma$ with a
disk $\overline{D}$ which intersects $L$ at the two basepoints
$(o,o')$. We also need to introduce an index two critical point inside
$U_\beta\cap D$ and an index one critical point inside $U_\alpha\cap
D$, such that the gradient flow of the modified Morse function $f_1$
agrees with $f$ on the boundary $\partial D$ of the three ball
$D$. The unstable manifold of the new index two critical points
intersects $\partial \overline{D}$ at two point. The $\alpha$-circle
corresponded to
the new index one critical point is inside the disk
$\overline{D}$. This modification from $f_0$ to $f_1$ is shown in
Figure \ref{fig:qsh-01}. 

\begin{figure}
  \centering
  \includegraphics[scale=0.7]{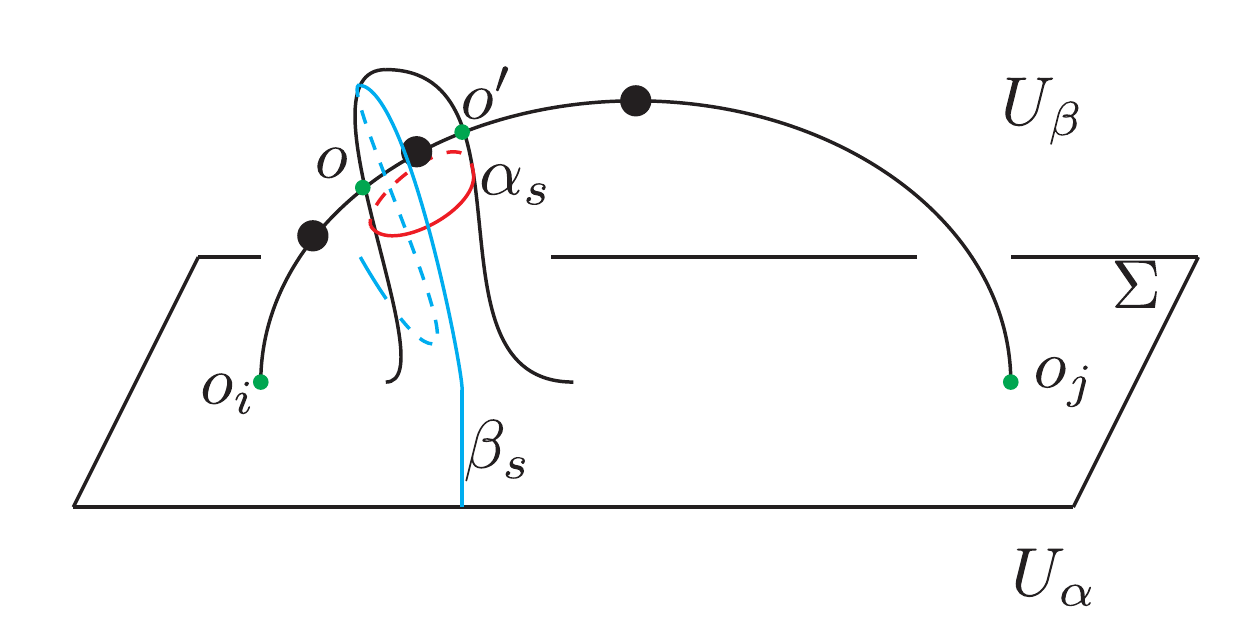}
  \caption{An example of quasi-stabilization $S_{+}^{\beta,o_i}$. The dividing point $q_i$ is in the center of the arc connecting $o_i$ and $o_j$. The four new points $(q_s,o,p_s,o')$ appear in order on the arc between $o_i$ and $q_i$.} 
  \label{fig:qsh-01}
\end{figure}

Now we consider this construction at the level of Heegaard
diagram. Given a Heegaard diagram 
$H=(\Sigma,\bs{\alpha},\bs{\beta},\mathbf{O}_0)$ compatible with $(\cal{L}_0,\mathbf{O}_0)$, we choose a point $p$ on
$\Sigma\backslash (\bs{\alpha}\cup\bs{\beta})$ together with a small
disk neighborhood $D_p$ and a circle $\beta_s$ through $p$ without
intersecting $\bs{\beta}$. We replace $D_p$ by another disk
$\overline{D}$ and introduce two basepoint $(o,o')$ on
$\overline{D}$. Next, we extend the $\beta_s$ into the disk
$\overline{D}$ and get a closed curve separating $(o,o')$ on
$\overline{\Sigma}=(\Sigma\backslash D_p)\cup \overline{D}$. From the Morse function viewpoint,
as we requre the curve $\beta_s$ intersect the projection of $L$ at
one point, this extension of $\beta_s$ over $\overline{D}$ is unique
upto isotopy.  Then we set the new alpha circle
$\alpha_s$ is parallel to $\partial \overline{D}$ and get a 
Heegaard diagram
$\overline{H}=(\overline{\Sigma},\bs{\alpha}\cup\alpha_s,\bs{\beta}\cup\overline{\beta}_s,\mathbf{O}\cup\{o,o'\})$,
where $\overline{\beta}_s$ is the extension of $\beta_s$ over $\overline{\Sigma}$. See Figure \ref{fig:qsh-01} for this construction.

\begin{rem}
  The construction of Heegaard diagram $\overline{H}$ does not depend on the curve
  $\cal{A}$ 
  of $\mathfrak{W}_{qs}$.
\end{rem}

\subsection{Choice of generators}\label{sec:choice-generators}

In this subsection, we will define a $\mathbb{Z}$-filtered curved chain homomorphism:
\[F_{\mathfrak{W}_{qs}}: CFBL(\cal{L}_0,O_0)\to CFBL(\cal{L}_1,O_1).\]

As in Section \ref{sec:topol-facts-about}, suppose that the quasi-stabilization $\mathfrak{W}_{qs}$ is $S_{+}^{\beta,o_i}$. In other words, the four new points $(q_s,o,p_s,o')$ appear in order on $L$
and are between
the point $(o_i,q_i)$ in $U_\beta$. For the Heegaard diagram
$\overline{H}$, the $\alpha_s$ intersects $\beta_{s}$ at two
points. If we close the diagram
$(\overline{D},\alpha_s,\overline{\beta_s}\cap\overline{D},o,o')$, we
will get a diagram
$(S^2,\alpha_s,\overline{\beta_s}\cap\overline{D},o,o')$. We denote by
$x^+$ the intersection with the highest $gr_o$-grading, and by $x^-$
the other intersection.
The generators of $CFBL(\overline{H})$ are of the form $\x\times
x^{\pm}$, where $\x$ is a generator of $CFBL(\overline{H})$

We define $F_{S^{\beta,o_i}}:CFBL(\cal{L}_0,O_0)\to CFBL(\cal{L}_1,O_1)$ by
\begin{equation}
  F_{S^{\beta,o_i}}(\x)= \x\times x^+.
\end{equation}
If we reverse a quasi-stabilization $S_{+}^{\beta,o_i}$,
we will get a quasi-destabilization $S_{-}^{\beta,o_i}$ from
$(\cal{L}_1,\mathbf{O}_1)$ to $(\cal{L}_0,\mathbf{O}_0)$.
Then we define  $F_{S_{-}^{\beta,o_i}}:CFBL(\cal{L}_1,O_1)\to
CFBL(\cal{L}_0,O_0)$ by 
\begin{align}
  F_{S_{-}^{\beta,o_i}}(\x\times x^+)&= 0,\\
  F_{S_{-}^{\beta,o_i}}(\x\times x^-)&= \x.
\end{align}

\begin{rem}
The construction of $F_{S^{\beta,o_i}}$ is similar to Zemke's definition $S^+_{w,z}$
and $T^+_{w,z}$ in \cite{Zemke2016}. Instead of colorings, the curves
$\cal{A}$ of $\mathfrak{W}_{qs}$ will provide us extra data to distinguish
the generators.
\end{rem}

We have the following lemma about the invariance of $F_{S_{+}^{\beta,o_i}}$
and $F_{S_{-}^{\beta,o_i}}$.

\begin{lem}
  The map $F_{S_{+}^{\beta,o_i}}$
and $F_{S_{-}^{\beta,o_i}}$  is well-defined upto $\mathbb{Z}$-filtered
curved chain homotopy and is independent of the choices of Heegaard
diagram $\overline{H}$.
\end{lem}

\begin{proof}
  One could prove the
following 
invariance by tracking 
the proof of \cite[Theroem A]{Zemke2016}.
\end{proof}

\begin{figure}
  \centering
  \includegraphics[scale=0.8]{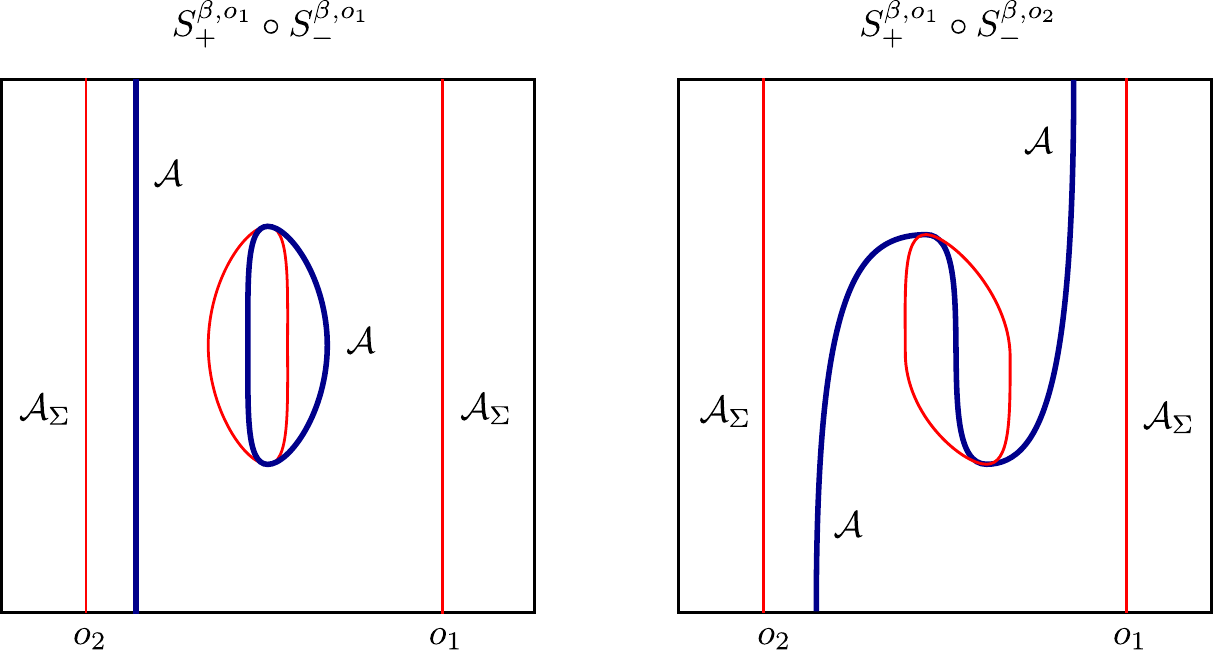}
  \caption{The composition of two quasi-stabilizations} 
  \label{fig:compquasi}
\end{figure}

\begin{prop}
  Suppose that $\mathbb{W}^1=S^{\beta,o_1}\circ S^{\beta,o_1}$ and
  $\mathbb{W}^2=S^{\beta,o_1}\circ S^{\beta,o_2}$ are two bipartite
  disoriented link cobordism from $(\cal{L},\mathbb{O})$ to
  $(\cal{L},\mathbb{O})$, where $o_1$ and $o_2$ are two adjacent
  basepoints. See Figure \ref{fig:compquasi} for the local picture of
  $\mathbb{W}^1$ and $\mathbb{W}^2$. For the two cobordism maps
  $F_{\mathbb{W}^i}: CFL'(\cal{L},\mathbb{O})\to
  CFL'(\cal{L},\mathbb{O})$, we have

\begin{align*}
&F_{\mathbb{W}^1}\cong 0\\
\text{ and }&F_{\mathbb{W}^2}\cong \operatorname{Id}.
\end{align*}

\end{prop}

\subsection{Disk-stabilizations}\label{sec:disk-stab}

Before we provide the proof for Theorem \ref{sec:introduction}, we
recall the definition of the map induced by a
disk-stabilization/destabilization from \cite[Section 7.3]{Zemke2016a}.

Let $(\mathcal{L},\mathbf{O})$ be a bipartite disoriented link in $Y$
and $(\mathcal{U},o_1\cup o_2)$ be a bipartite disoriented unknot with
two basepoints. Suppose $\mathcal{U}$ bounds a disk $D$ such that
$(D\cap L)=\emptyset$. Then we can form a new biparite
disoriented link $(\mathcal{L},\mathbf{O})\cup(\mathcal{U},o_1\cup
o_2)$. We call this process a disk-stabilization and its inverse a
disk-destabilization (its cobordism picture is shown in Figure
\ref{fig:elecob}). For convenience, we denote by $D^+$ a
disk-stabilization and by $D^-$ a disk-destabilization.

\begin{figure}
  \centering
  \includegraphics[scale=0.8]{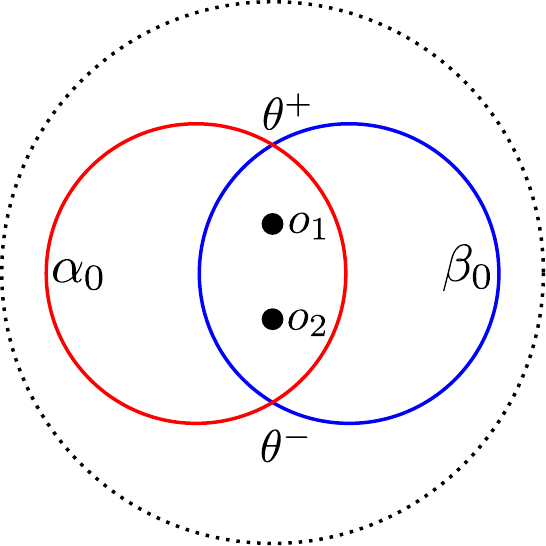}
  \caption{Local picture of Heegaard diagram for disk-stabilization/destabilization}
  \label{fig:disk} 
\end{figure}

We can pick a Heegaard diagram $(\Sigma,\bs{\alpha},\bs{\beta},\mathbf{O})$ for
$(\mathcal{L},\mathbf{O})$ such that
$(D\cap\Sigma)\cap(\bs{\alpha}\cup\bs{\beta})=\emptyset$. By replacing a
small disk neighborhood of $(D\cap\Sigma)$ on $\Sigma$ with the picture
shown in Figure \ref{fig:disk}, we get a Heegaard diagram
$\hat{H}=(\Sigma,\bs{\alpha}\cup\alpha_0,\bs{\beta}\cup\beta_0,\mathbf{O}\cup
o_1\cup o_0)$ for $(\mathcal{L},\mathbf{O})\cup(\mathcal{U},o_1\cup
o_2)$. The generators of $CFL'(\hat{H})$ are of the form $\x\times
\theta^+$ or $\x\times\theta^-$, where $\x\in CFL'(H)$, $\theta^+$ and $\theta^-$ are the
two new intersection of $\alpha_0$ and $\beta_0$. We define the
map $f_{D^+}: CFL'(H)\rightarrow CFL'(\hat{H})$ by setting 
\[ f_{D^+}(\x)=\x\times\theta^+,\]
and the map $f_{D^-}: CFL'(\hat{H})\rightarrow CFL'(H)$ by setting
\[ f_{D^-}(\x\times\theta^+)=0 \textnormal{ and }
f_{D^-}(\x\times\theta^-)=\x.\] 
From the discusstion \cite[Section 7.1]{Zemke2016a}, we know that
$f_{D^+}$ and $f_{D^-}$ induce well-defined maps $F_{D^+}$ and $F_{D^-}$ on unoriented link
Floer homology $HFL'$.

\section{Commutations}\label{sec:commutations}

We divide this section into two parts. In the first part (Subsection
\ref{sec:com:ab},\ref{sec:com:bb} and \ref{sec:comm-betw-band}), we
show that changing the ordering of two critical points of $\pi|_F$ or
$\pi|_A$ does not affect the cobordism mpas at the level of
$CFL'$. The results from this part will be used to show the
well-definedness of the cobordism maps on $CFL'$ constructed from a
given bipartite disoriented link cobordism. 

In the second part (Subsection \ref{sec:com:abb} and
\ref{sec:com:abq}), we provide a relation between the maps induced by $\alpha$-band moves
and $\beta$-band moves, and a relation between the maps induced by
$\alpha$-quasi-stabilizations and $\beta$-quasi-stabilizations. These
two relations will be used to show that the cobordism maps on $CFL'$ constructed from
bipartite dioriented link cobordisms are independent of the motion of
basepoints (the data (D3)). Hence we get well-defined maps on $CFL'$ for
disoriented link cobordism.

\subsection{Commutation between $\alpha$-band moves and
  $\beta$-band moves}\label{sec:com:ab}

Based on Zemke's work in \cite[Section 7.3]{Zemke2016a}, by
constructing a Heegaard quadruple, we will show the commutation
between $\alpha$-band moves and $\beta$-band moves at the level of $CFL'$.

Suppose $B^{\alpha,o_{i'},o_{j'}}$ and $B^{\beta,o_i,o_j}$ are two
disjoint band on bipartite disoriented link
$(\cal{L}_1,\mathbf{O})$. Then we have four elementary bipartite disoriented
link cobordisms:

\begin{itemize}
\item the band move
  $B^{\alpha,o_{i'},o_{j'}}$ from $(\cal{L}_0,\mathbf{O})$ to
  $(\cal{L}_1,\mathbf{O})$.
 \item  the band move
  $B^{\alpha,o_{i'},o_{j'}}$ from $(\cal{L}_2,\mathbf{O})$ to
  $(\cal{L}_3,\mathbf{O})$.
\item  the band move
  $B^{\beta,o_{i},o_{j}}$ from $(\cal{L}_0,\mathbf{O})$ to
  $(\cal{L}_2,\mathbf{O})$.
\item the band move
  $B^{\beta,o_{i},o_{j}}$ from $(\cal{L}_1,\mathbf{O})$ to
  $(\cal{L}_3,\mathbf{O})$.
\end{itemize}

The two composition of bipartite disoriented link cobordisms
$B^{\beta,o_{i},o_{j}}\circ B^{\alpha,o_{i'},o_{j'}}$ and
$B^{\alpha,o_{i'},o_{j'}}\circ B^{\beta,o_{i},o_{j}}$ are
isomorphic. Both of the composition are from the bipartite disoriented
link $(\cal{L}_0,\mathbf{O})$ to
  $(\cal{L}_3,\mathbf{O})$.

\begin{lem}\label{sec:comm-betw-alpha}

  There exists a Heegaard quadruple $(\Sigma,\bs{\alpha'},\bs{\alpha},\bs{\beta},\bs{\beta'},\mathbf{O})$ such that,
\begin{itemize}
\item the induced Heegaard triple $\cal{T}_{\alpha'\alpha\beta}$ is subordinate to the band move
  $B^{\alpha,o_{i'},o_{j'}}$ from $(\cal{L}_0,\mathbf{O})$ to
  $(\cal{L}_1,\mathbf{O})$.
 \item the induced Heegaard triple $\cal{T}_{\alpha'\alpha\beta'}$ is subordinate to the band move
  $B^{\alpha,o_{i'},o_{j'}}$ from $(\cal{L}_2,\mathbf{O})$ to
  $(\cal{L}_3,\mathbf{O})$.
\item the induced Heegaard triple $\cal{T}_{\alpha\beta\beta'}$ is subordinate to the band move
  $B^{\beta,o_{i},o_{j}}$ from $(\cal{L}_0,\mathbf{O})$ to
  $(\cal{L}_2,\mathbf{O})$.
\item the induced Heegaard triple $\cal{T}_{\alpha'\alpha\beta}$ is subordinate to the band move
  $B^{\beta,o_{i},o_{j}}$ from $(\cal{L}_1,\mathbf{O})$ to
  $(\cal{L}_3,\mathbf{O})$.
\end{itemize}

\end{lem}

\begin{proof}
This construction can be done in two steps.

\textbf{Step1:} There exists a Heegaard decomposition such that the
$\alpha$-band lies in $U_\alpha$ and the $\beta$-band lies in
$U_\beta$. 

Suppose we have a Morse function $f$ compatible with
the bipartite disoriented link $(\cal{L}_0,\mathbf{O}$). This induces
a Heegaard decomposition $U_\alpha\cup_{\Sigma}U_\beta$ of the three
manifold $Y$. The core of the
band $B^{\alpha,o_{i'},o_{j'}}$ intersects $U_\beta $ with some
arcs. For each of these arcs, one can replace the
two disk neighborhoods of the two endpoints of the arc with a
surface with two holes which does not intersects $\cal{L}_0$ and any
bands and get a new surface $\overline{\Sigma}$. Furthermore, one can
require $\overline{\Sigma}$ cut the manifold into two handlebodies.  One example of this process is shown in Figure \ref{fig:cob-09}. When we
perform this for both bands, the $\alpha$ band will lie in
$U_\alpha$ and $\beta$ band lies in $U_\beta$.

\begin{figure}
  \centering
  \includegraphics[scale=0.5]{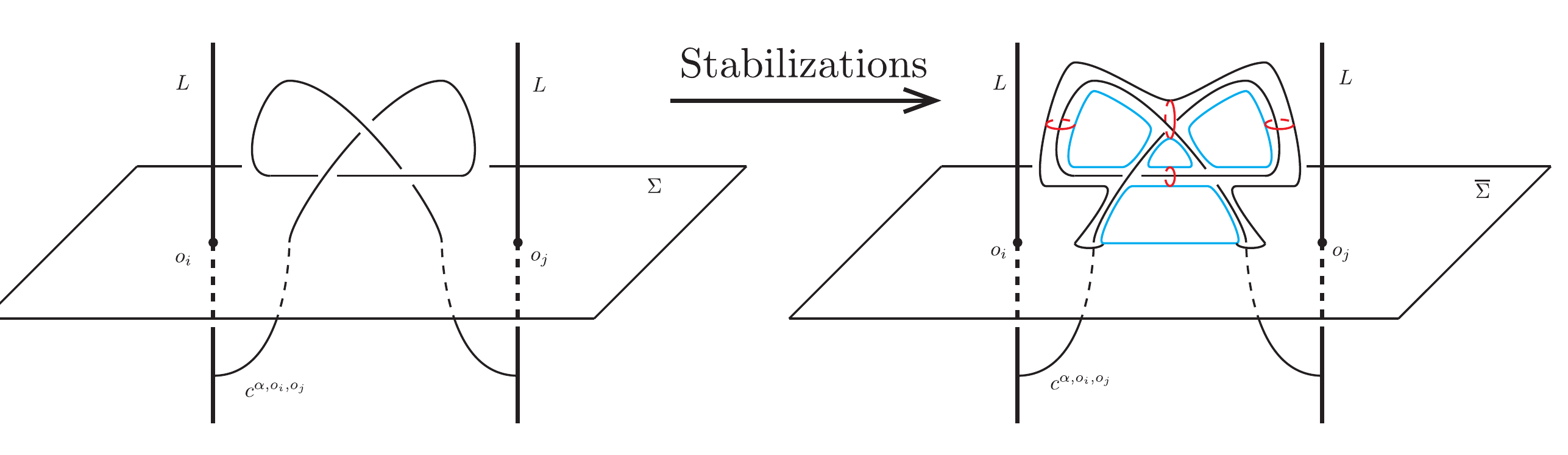}
  \caption{ Commutation between band moves.}
  \label{fig:cob-09}
\end{figure}

\textbf{Step 2:} Using the techniques described in the proof of Theorem \ref{sec:exist-heeg-triple}, we further stabilize the diagram $H_{\alpha\beta}$ of
  $(\cal{L}_0,\mathbf{O})$ to implant the data of the two bands
  $B^{\beta,o_{i},o_{j}}$ and $B^{\alpha,o_{i'},o_{j'}}$ into the
  Heegaard diagram. Notice that, these two implantation
  processes 
  (stabilizations and handleslides) happened in different handlebodies. Therefore, the two process
  are independent to each other. Finally, after Hamiltonian isotopies,
  we will get the desired Heegaard quadruple.

\end{proof}

Similar to the proof of \cite[Proposition 7.7]{Zemke2016a}, we have
the following lemma.

\begin{lem}\label{sec:comm-betw-alpha-1}
Let  
 $B^{\alpha,o_{i'},o_{j'}}$ and $B^{\beta,o_i,o_j}$ be the two
disjoint band on bipartite disoriented link
$(\cal{L}_0,\mathbf{O})$. For the induced $\mathbb{Z}$-filtered chain maps
$F_{B^{\beta,o_i,o_j}}$
and $F_{B^{\alpha,o_{i'},o_{j'}}}$ at the level of $CFL'$, we have the following commutation:
 
\[F_{B^{\beta,o_i,o_j},\mathfrak{s}}\circ F_{B^{\alpha,o_{i'},o_{j'}},\mathfrak{s}}\simeq
F_{B^{\alpha,o_{i'},o_{j'}},\mathfrak{s}}\circ
F_{B^{\beta,o_i,o_j},\mathfrak{s}}\]
\end{lem}

\begin{proof}
 Let's consider the 
 Heegaard quadruple constructed in Lemma \ref{sec:comm-betw-alpha},.
 By the associativity we proved in Lemma \ref{sec:glob-textsp-assoc-1}, we have the
 following equality: 
\begin{align*}
F_{B^{\beta,o_i,o_j}}\circ F_{B^{\alpha,o_{i'},o_{j'}}}&=
  f_{\alpha'\beta\beta'}(f_{\alpha'\alpha\beta}(\Theta_{\alpha\alpha'}\otimes -)\otimes
  \Theta_{\beta\beta'})\\
&\simeq f_{\alpha'\alpha\beta'}(\Theta_{\alpha\alpha'}\otimes
  f_{\alpha\beta\beta'}(-\otimes \Theta_{\beta\beta'}))\\
&=F_{B^{\alpha,o_{i'},o_{j'}}}\circ
F_{B^{\beta,o_i,o_j}}.
\end{align*}
\end{proof}

\subsection{Commutation between $\beta$-band moves}
\label{sec:com:bb}

In this subsection we will generalize the results in \cite[Section
7.5]{Zemke2016a} and show the commutation of the triangle maps induced
by two distinct $\beta$-band move. As we introduce type I band moves,
the proof will be slightly different from the proofs in \cite[Section
7.5]{Zemke2016a}. See Figure \ref{fig:com01} for an example of the
commutation between a type I $\beta$-band move and type II
$\beta$-band move. 

\begin{figure}
  \centering
  \includegraphics[scale=0.5]{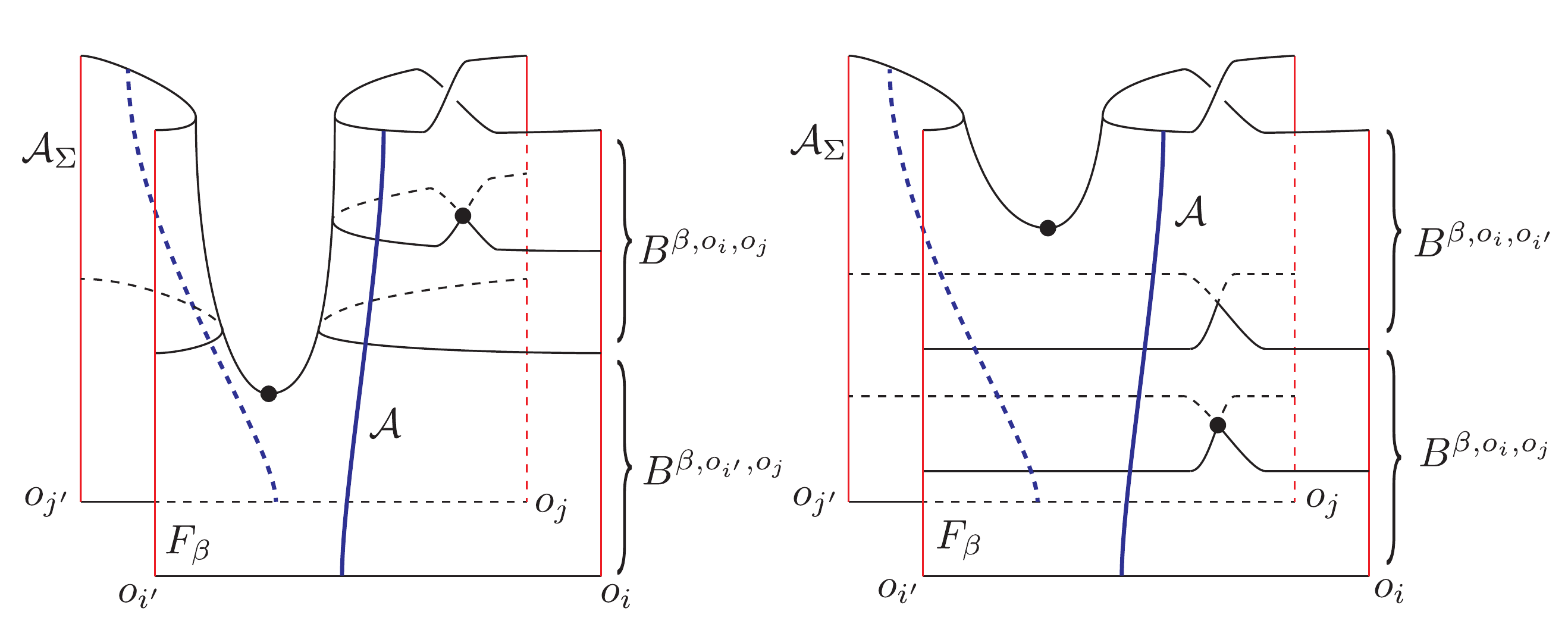}
  \caption{ Commutation between two $\beta$-band moves.}
  \label{fig:com01}
\end{figure}

Suppose $B_1^\beta$ is a type II $\beta$-band on bipartite link
$L_{\alpha\beta}$. Recall from Section \ref{sec:categ-2:-bipart} that,
the set $L_\beta={L_{\beta,1},\cdots,L_{\beta,n}}$ is a $n$-tuple of
arcs lie in $\beta$-handlebody. The endpoints of these arcs are exactly
all the basepoints $\mathbf{O}$ of $L_{\alpha\beta}$. As the band $B_1^\beta$ is of
type II, the two ends of the bands should lie in two components
$L_{\beta,i},L_{\beta,j}$ of
$L_\beta$. We call the ends of $L_{\beta,i},L_{\beta,j}$ the four
nearest basepoints adjacent to
$B_1^{\beta}$. 
Following the proof of Theorem \ref{sec:exist-heeg-triple}, we have a
lemma below.

\begin{lem}\label{sec:comm-betw-beta}
Suppose $B^\beta$ is a type II $\beta$-band on bipartite link
$L_{\alpha\beta}$. Let $\cal{T}$ be a standard Heegaard triple
subordinate to $B^\beta$. We can do $\beta$ and $\gamma$ handleslides
of the triple
without crossing any basepoints $\mathbf{O}$ to get a simplified
Heegaard triple $\cal{T}'$, with a disk neighborhood
$D$ on the surface $\Sigma$ such that $D\cap H_{\beta\gamma}'$ is shown in
Figure \ref{fig:simpletype2}. Here, the four basepoints $o_1,o_2,o_3,o_4\in \mathbf{O}$
are the four nearest basepoints adjacent to the band $B^\beta$ the
diagram $H_{\beta\gamma}'$ is the Heegaard diagram induced from
$\cal{T}'$ .  
\end{lem}

\begin{proof}
Straightforward.
\end{proof}

\begin{figure}
  \centering
  \includegraphics[scale=0.5]{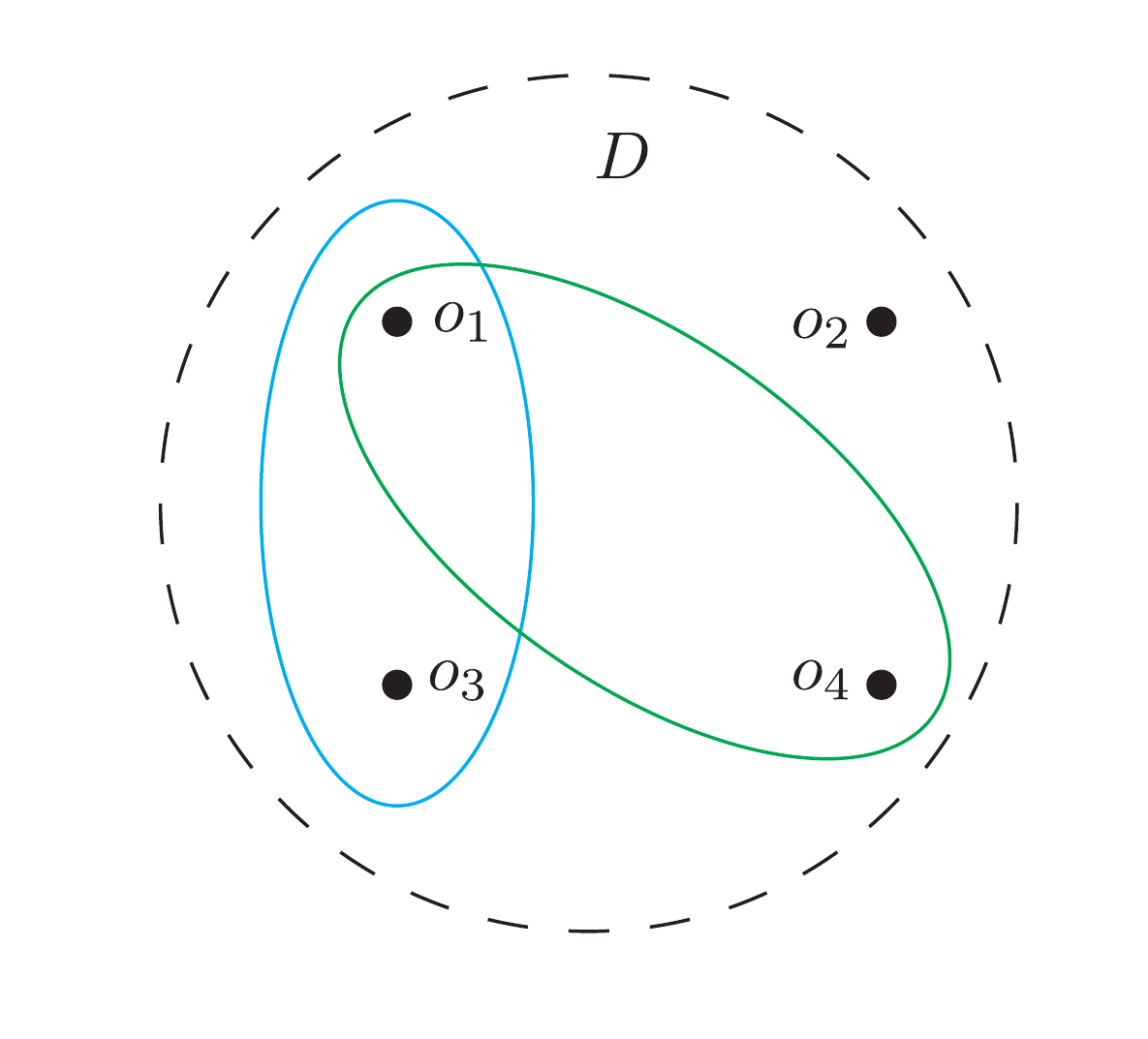}
  \caption{ A simplified Heegaard diagram $H_{\beta\gamma}$ for type II $\beta$-band move.}
  \label{fig:simpletype2}
\end{figure}

\begin{rem}
Notice that the top right basepoint in Figure \ref{fig:simpletype2} can be any of the four nearest
basepoints adjacent to band $B^\beta$.
\end{rem}

We say that two type II $\beta$-band move $B_1^{\beta}$ and $B_2^{\beta}$ are
\textit{\textbf{away}} from each other if the two ends of $B_1^{\beta}$ and the two ends
of $B_2^{\beta}$ lie in four different components of $L_{\beta}$. 

\begin{lem}\label{sec:comm-betw-beta-1}
  Suppose $B_1=B_1^{\beta,o_i,o_j}$ and $B_2=B_2^{\beta,o_{i'},o_{j'}}$ are two distinct bands on
  $L_{\alpha\beta}$, such that the composition of cobordisms $B_2\circ
  B_1$ is isomorphic to  $B_1\circ
  B_2$.
  There exists a Heegaard quintuple $(\Sigma,\alpha,\beta,\gamma,\gamma',\delta)$ such that,
\begin{itemize}
\item The induced Heegaard triple $\cal{T}_{\alpha\beta\gamma}$ is subordinate to the band move
  $B_1$ from $(\cal{L}_0,\mathbf{O})$ to
  $(\cal{L}_1,\mathbf{O})$.
 \item The induced Heegaard triple $\cal{T}_{\alpha\beta\gamma'}$ is subordinate to the band move
  $B_2$ from $(\cal{L}_2,\mathbf{O})$ to
  $(\cal{L}_3,\mathbf{O})$.
\item The induced Heegaard triple $\cal{T}_{\alpha\gamma\delta}$ is subordinate to the band move
  $B_2$ from $(\cal{L}_0,\mathbf{O})$ to
  $(\cal{L}_2,\mathbf{O})$.
\item  The induced Heegaard triple $\cal{T}_{\alpha\gamma'\delta}$ is subordinate to the band move
  $B_1$ from $(\cal{L}_1,\mathbf{O})$ to
  $(\cal{L}_3,\mathbf{O})$.
\end{itemize}
Furthermore, if $B_1$ and  $B_2$ are of type II and away
from each other, then there exists two distinct disk neighborhoods $D_1$ and
$D_2$ on $\Sigma$, such that the four local diagram $D_1\cap
\cal{T}_{\beta\gamma\delta}$, $D_2\cap
\cal{T}_{\beta\gamma\delta}$, $D_1\cap
\cal{T}_{\beta\gamma'\delta}$, $D_2\cap
\cal{T}_{\beta\gamma'\delta}$ are the four diagrams shown in Figure
\ref{fig:cog06} and Figure
\ref{fig:cog07}.
\end{lem}

\begin{figure}
  \centering
  \includegraphics[scale=0.5]{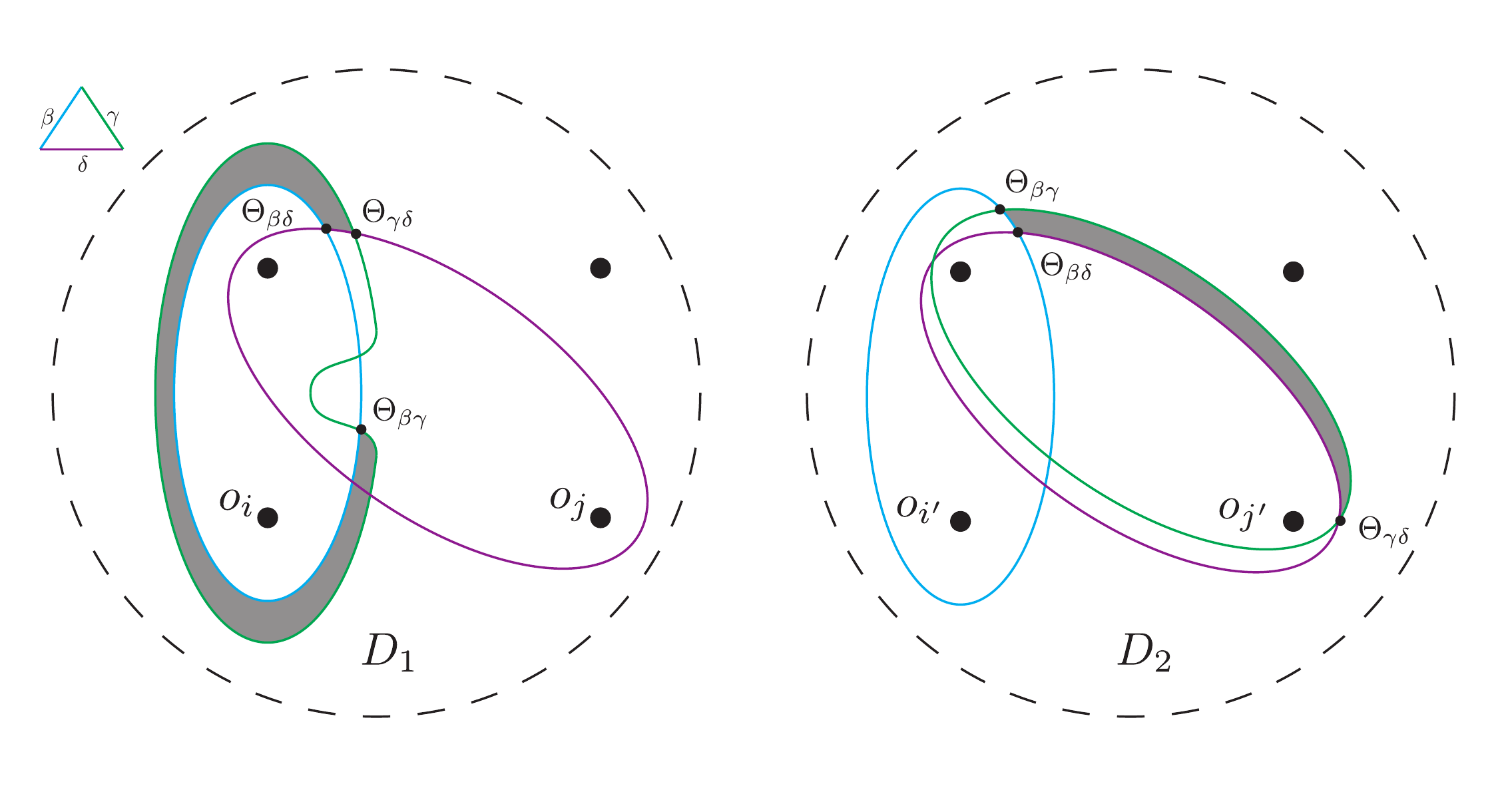}
  \caption{ Local picture of $\cal{T}_{\beta\gamma\delta}$.}
  \label{fig:cog06}
\end{figure}

\begin{figure}
  \centering
  \includegraphics[scale=0.5]{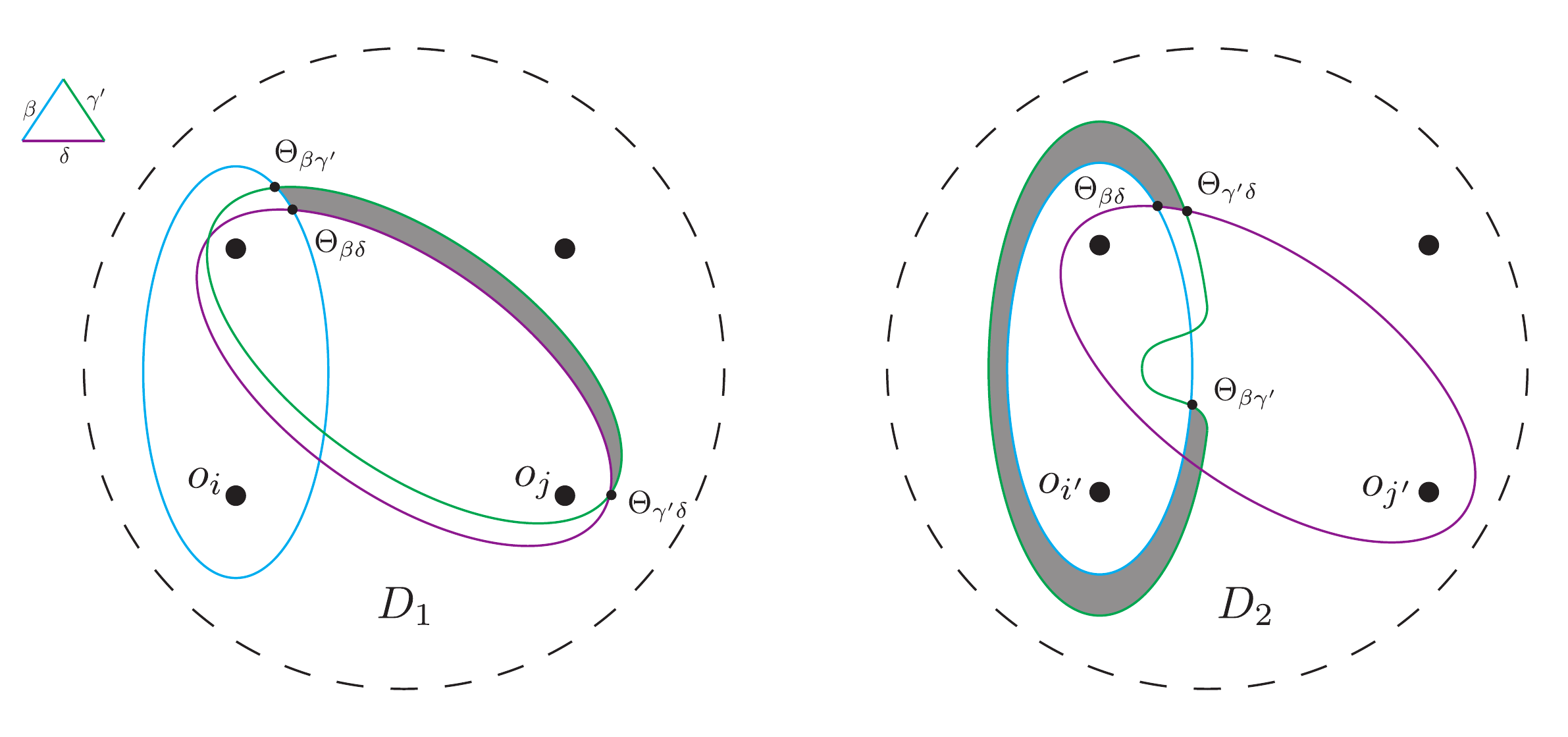}
  \caption{ Local picture of $\cal{T}_{\beta\gamma'\delta}$.}
  \label{fig:cog07}
\end{figure}

\begin{proof}

Recall in the proof of Theorem \ref{sec:exist-heeg-triple}, we start from the Heegaard diagram
of the bipartite link $L_{\alpha\beta}$ determined by
$(\cal{L}_0,\mathbf{O})$ and implant the data of the core of the bands by
stabilize the Heegaard diagram. As the two $\beta$-band are distinct
bands, the stabilizations and Dehn twists for implanting the data
$B_1$ is independent to the stabilizations and Dehn twists
for implanting the data of $B_2$. Then we
can get a desired Heegaard quintuple. One example of this process is shown in Figure
\ref{fig:bandcom}. 

\begin{figure}
  \centering
  \includegraphics[scale=0.6]{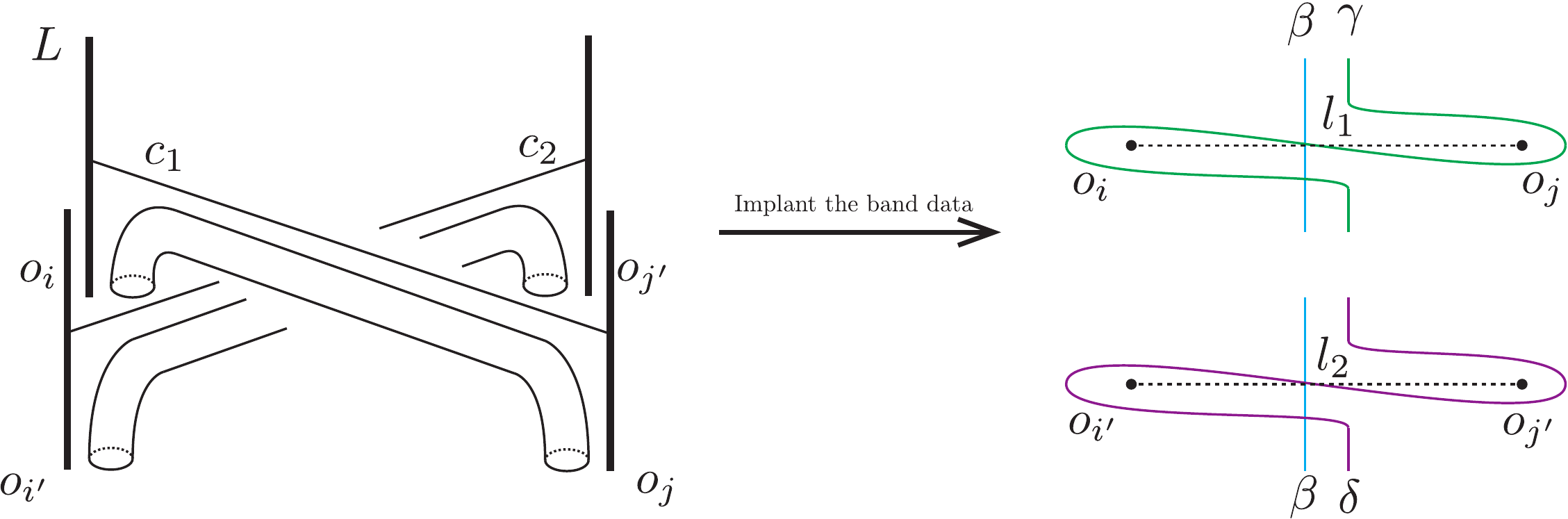}
  \caption{Construction of Heegaard quintuple.   Here, $c_1$ and $c_2$ are the cores of the band $B_1$ and $B_2$. The
  dashed line $l_1$ and $l_2$ are the projection
  image on $\Sigma$ of $c_1$ and $c_2$. For convenience, we didn't draw
  the curves $\bs{\gamma}'$ and some other curves.}
  \label{fig:bandcom}
\end{figure}

Furthermore, if the two bands are of type II and away from each other,
then we can start from the Heegaard triple subordinate to $B_1$
shown in Lemma \ref{sec:comm-betw-beta}, we have a disk region $D_1$
shown in Figure \ref{fig:simpletype2}. 

Now we can require that the stabilization
and Dehn twist for $B_2$ happens away from the region $D_1$. After handleslide, we will get the desired
Heegaard quintuple with the two disk region $D_1$ and $D_2$ shown in
Figure \ref{fig:cog06} and Figure \ref{fig:cog07}. 

\end{proof}

The following lemma is a result from the triangle computations of the bipartite
disoriented links in $\#^n(S^1\times S^2)$ induced by the above
Heegaard quintuple.

\begin{lem}\label{sec:comm-betw-beta-3}
  Suppose $B_1=B_1^{\beta,o_i,o_j}$ and $B_2=B_2^{\beta,o_{i'},o_{j'}}$ are
  two type II $\beta$-band away from each other. Then the induced map
  $F_{B_1}$ and $F_{B_2}$ (defined in Theorem
  \ref{sec:bipart-disor-link}) commutes with each other. 
\end{lem}

\begin{proof}
We construct a Heegaard quintuple by Lemma
\ref{sec:comm-betw-beta-1}. We denote by $\Theta_{\beta\gamma}$ (and
$\Theta_{\gamma'\delta}$ resp.)  the
generator determined by $B_1$ and by $\Theta_{\beta\gamma'}$ (and
$\Theta_{\gamma\delta}$ resp.) the
generator determined by $B_2$. 

By the triangle calculation shown in Figure \ref{fig:cog06}, we have:
\[f_{\beta\gamma\delta}(\Theta_{\beta\gamma}\otimes\Theta_{\gamma\delta})=\Theta_{\beta\delta}.\] Here
$\Theta_{\beta\delta}$ is a top grading generator of $CFL'(H_{\beta\delta},\mathfrak{s})$. 
Similarly, by the triangle calculation shown in Figure \ref{fig:cog07}, we have:
\[f_{\beta\gamma'\delta}(\Theta_{\beta\gamma'}\otimes\Theta_{\gamma'\delta})=\Theta_{\beta\delta}.\]
Notice that the generator $\Theta_{\beta\delta}$ in Figure \ref{fig:cog06} is
exactly the one shown in Figure \ref{fig:cog07}.

By associativity, we have:

\begin{align*} 
F_{B_2}\circ F_{B_1} &=
                       f_{\alpha\gamma\delta}(f_{\alpha\beta\gamma}(-\otimes 
                       \Theta_{\beta\gamma})\otimes
                       \Theta_{\gamma\delta})\\
&=
  f_{\alpha\beta\delta}(-\otimes f_{\beta\gamma\delta}(\Theta_{\beta\delta}\otimes\Theta_{\gamma\delta})) 
\\
&=f_{\alpha\beta\delta}(-\otimes \Theta_{\beta\delta})
\\
&= f_{\alpha\gamma'\delta}(f_{\alpha\beta\gamma'}(-\otimes 
                       \Theta_{\beta\gamma'})\otimes \Theta_{\gamma'\delta})
\\
 &=f_{\alpha\gamma'\delta}(f_{\alpha\beta\gamma'}(-\otimes 
                       \Theta_{\beta\gamma'})\otimes \Theta_{\gamma'\delta})\\
  &=F_{B_1}\circ F_{B_2}
\end{align*}

\end{proof}

\begin{lem}\label{sec:comm-betw-beta-2}
  Suppose $B_1$ and $B_2$ are
  two distinct $\beta$-bands on bipartite disoriented link
  $(\cal{L},\mathbf{O})$. If the bipartite disoriented link cobordism $B_2\circ B_1$  is
  isomorphic to $B_1\circ B_2$, then the two maps $F_{B_1}$ and
  $F_{B_2}$ commute with each other.
\end{lem}

\begin{proof}
  We postpone the proof to the end of the Section \ref{sec:com:abb}.
\end{proof}

\subsection{Commutation between band moves and quasi-stabilizations}\label{sec:comm-betw-band}
In this subsection, we will show the  commutation between the maps on
$CFL'$ associated to band moves and
quasi-stabilizations. The discussions about commutations are based
 on Manolescu's work in \cite[Section 5]{Manolescu2010} and
Zemke's work in \cite{Zemke2016} and \cite{Zemke2016a}. 

The bipartite link cobordism for the commutation between a
band move $B^\beta$ and quasi-stabilization $S^\beta$ contains two
critical points. One is the saddle critical point of the link
cobordism surface $F$, and the other one is the critical point for
one-manifold $\cal{A}_\Sigma$. We classify the cobordism of the
commutation into two cases:

\textbf{Case 1} (band move type changes): as shown in Figure \ref{Fig:cb-02}, the
critical point of $\cal{A}_\Sigma$ and the saddle critical point of
$F$ lie in the same component of $F_\beta$. Furthermore, the type of
band move changes from type I to type II after
quasi-stabilization. Notice that, this case can never happen for an
oriented link cobordism. For convenience, we call this a
\textbf{\textit{special commutation}} between $B^\beta$ and $S^\beta$.

\textbf{Case 2} (band move type does not change): we further classify
the cobordism into the following three subcases.

\begin{figure}
  \centering
  \includegraphics[scale=0.5]{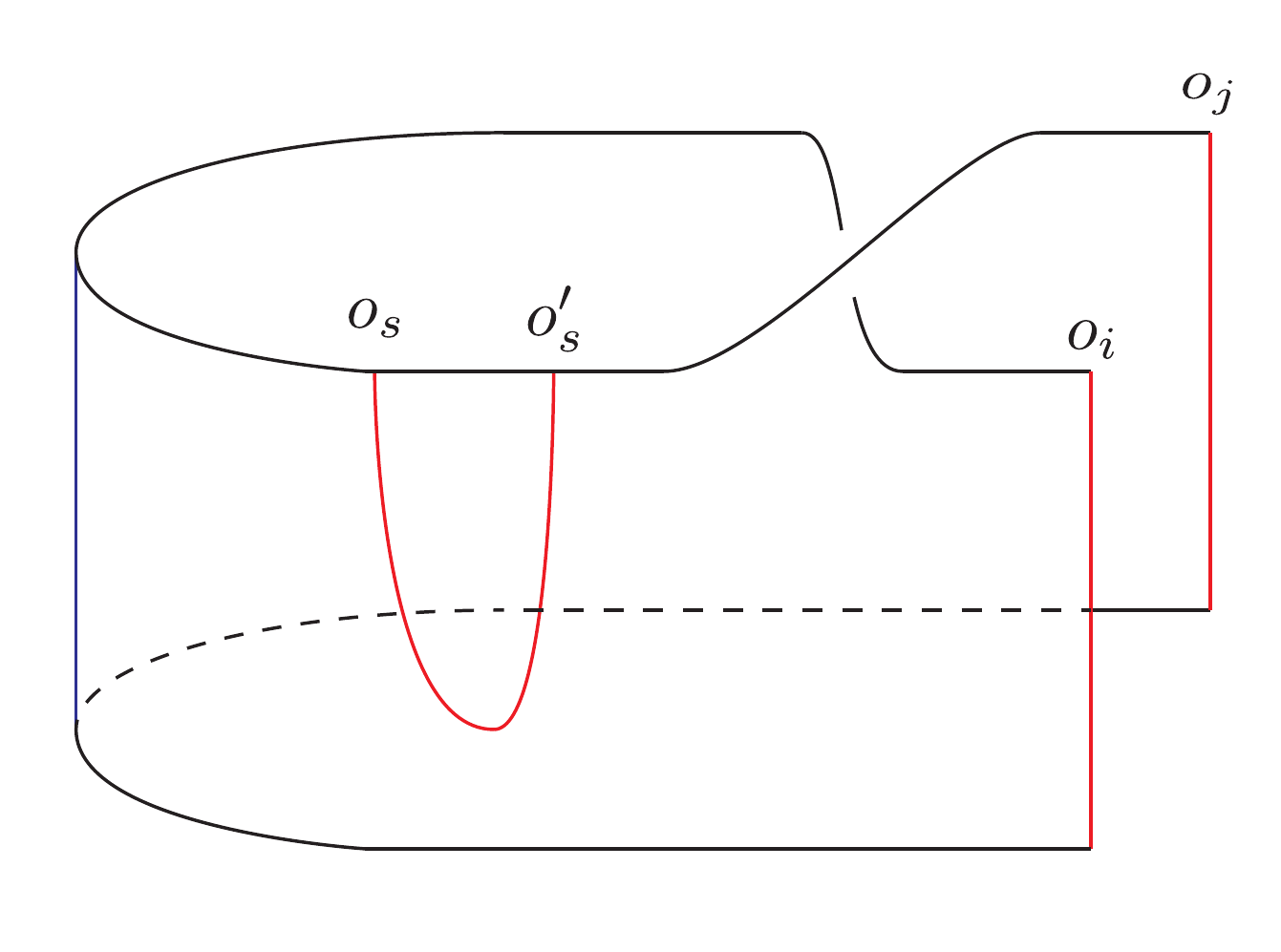}
  \caption{Case 1: Band move type changes.}
  \label{Fig:cb-02}
\end{figure}

\begin{figure}
  \centering
  \includegraphics[scale=0.5]{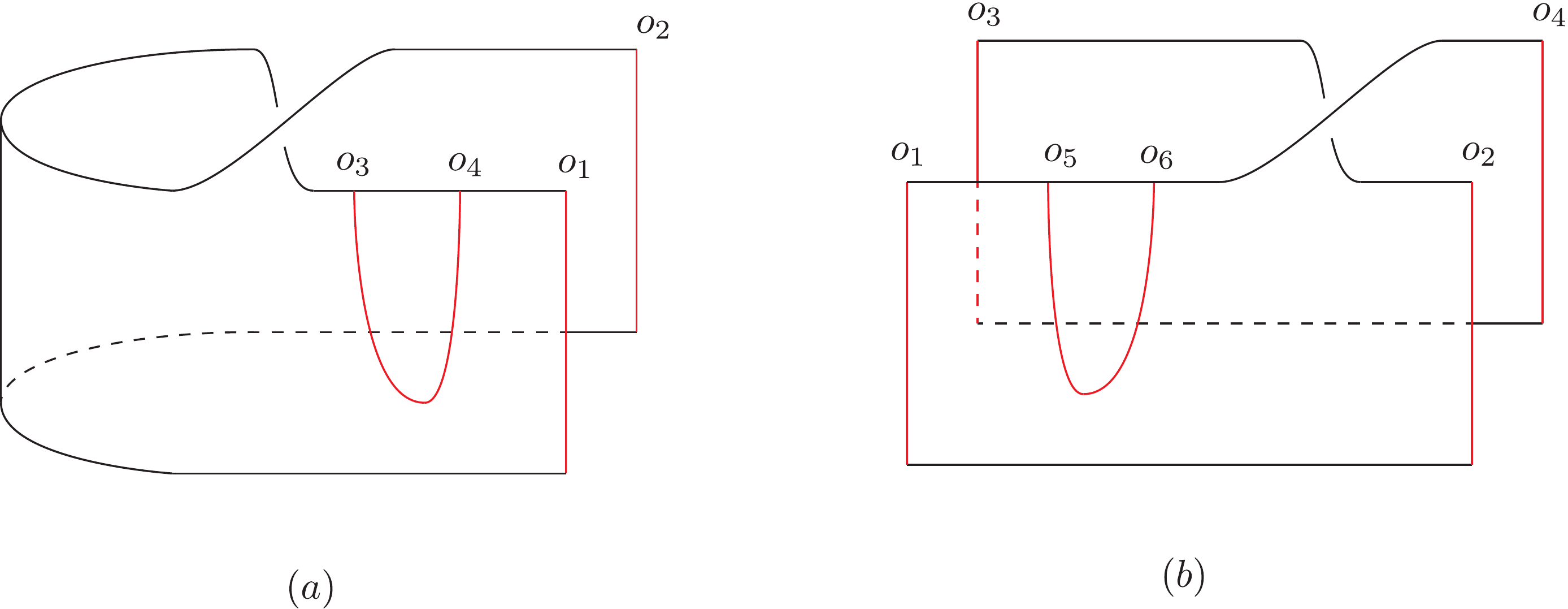}
  \caption{Case 2: Band move type does not change.}
  \label{Fig:cb-01}
\end{figure}

\begin{enumerate}[(a)]
\item  As shown in Figure \ref{Fig:cb-01} (a), the
critical point of $\cal{A}_\Sigma$ and the saddle critical point of
$F$ lie in the same component of $F_\beta$. The band move $B^\beta$ is
of type I.
\item  As shown in Figure \ref{Fig:cb-01} (b), the
critical point of $\cal{A}_\Sigma$ and the saddle critical point of
$F$ lie in the same component of $F_\beta$. The band move $B^\beta$ is
of type II.
\item  The
critical point of $\cal{A}_\Sigma$ and the saddle critical point of
$F$ lie in different components of $F_\beta$. 
\end{enumerate}

Next, we translate the cobordism data into Heegaard diagrams data.

\begin{defn}
  Suppose we have a bipartite link cobordism made of a band move $B^\beta$ and a
  quasi-stabilization $S^\beta$, such that $B^\beta\circ S^\beta\cong
  S^\beta\circ B^\beta$.
  We say that 
  $\cal{T}=(\Sigma,\bs{\alpha}\cup\{\alpha_s\},\bs{\beta})\cup
  \{\beta_s\},\bs{\gamma}\cup \{\gamma_s\},\mathbf{O}\cup\{o_s,o_s'\})$ is a
  \textbf{\textit{stabilized Heegaard triple subordinate to the commutation}}
 , if it satisfies
  the following diagram:
\[\xymatrix{
\overline{H}_{\alpha\beta}\ar[d]_{S^\beta} \ar[r]^{B^\beta} &\overline{H}_{\alpha\gamma}\ar[d]^{S^\beta} \\
H_{\alpha\beta}\ar[r]^{B^\beta} &H_{\alpha\gamma}}.\]
Here, the Heegaard diagrams above are described below:
\begin{itemize}
\item the triple
  $\overline{\cal{T}}=(\Sigma,\bs{\alpha},\bs{\beta},\bs{\gamma},\mathbf{O})$
  (or the $\cal{T}$ resp.)
  is subordinate to the band move $B^\beta$.
\item the stabilized Heegaard diagram $H_{\alpha\gamma}$ (or
  $H_{\alpha\beta}$ resp.) is
  subordinate to $S^\beta$.
\end{itemize}
\end{defn}

Next, we will construct a stabilized Heegaard triple for the
commutation between $B^\beta$ and $S^\beta$. This construction depends on
the type of $\beta$-band moves and the relative position between
$B^\beta$ and $S^\beta$. 

\begin{lem}\label{sec:comm-betw-band-2}
  Suppose $\cal{W}$ is a bipartite link cobordism for band move
  $B^\beta$ and quasi-stabilization $S^\beta$. 
  If the band move type
  changes after quasi-stabilization, then we can construct a stabilized
  Heegaard 
  triple subordinate to $\cal{W}$ as shown in Figure \ref{Fig:cbqs-07}.
If the band move type
  does not change, then we can construct a free stabilized Heegaard
  triple subordinate to $\cal{W}$ as shown in Figure \ref{Fig:cb-003} .  

\end{lem}

\begin{figure}
  \centering
  \includegraphics[scale=0.45]{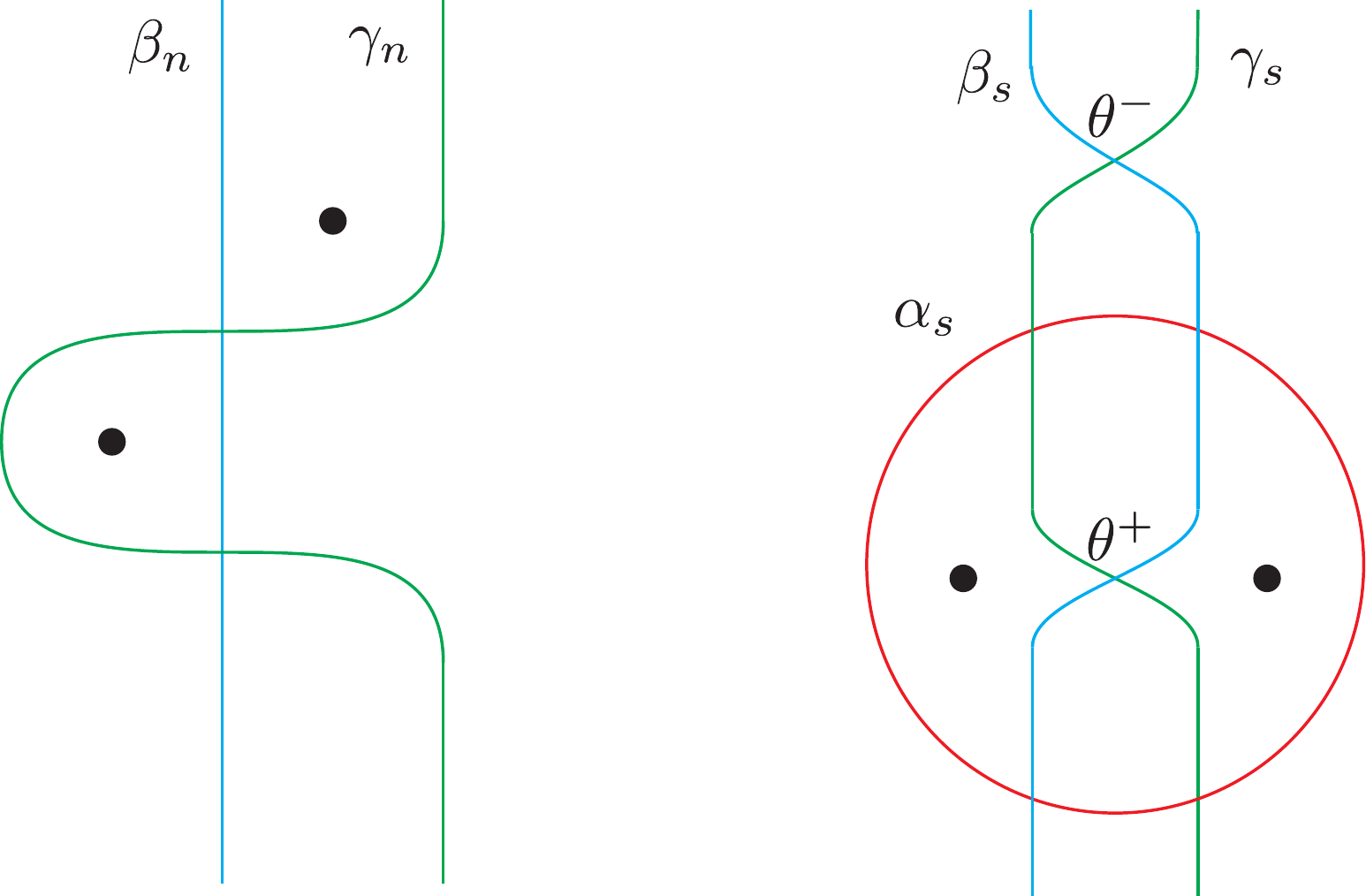}
  \caption{Stabilized Heegaard triple for Case 1.}
  \label{Fig:cbqs-07}
\end{figure}

\begin{figure}
  \centering  
  \includegraphics[scale=0.32]{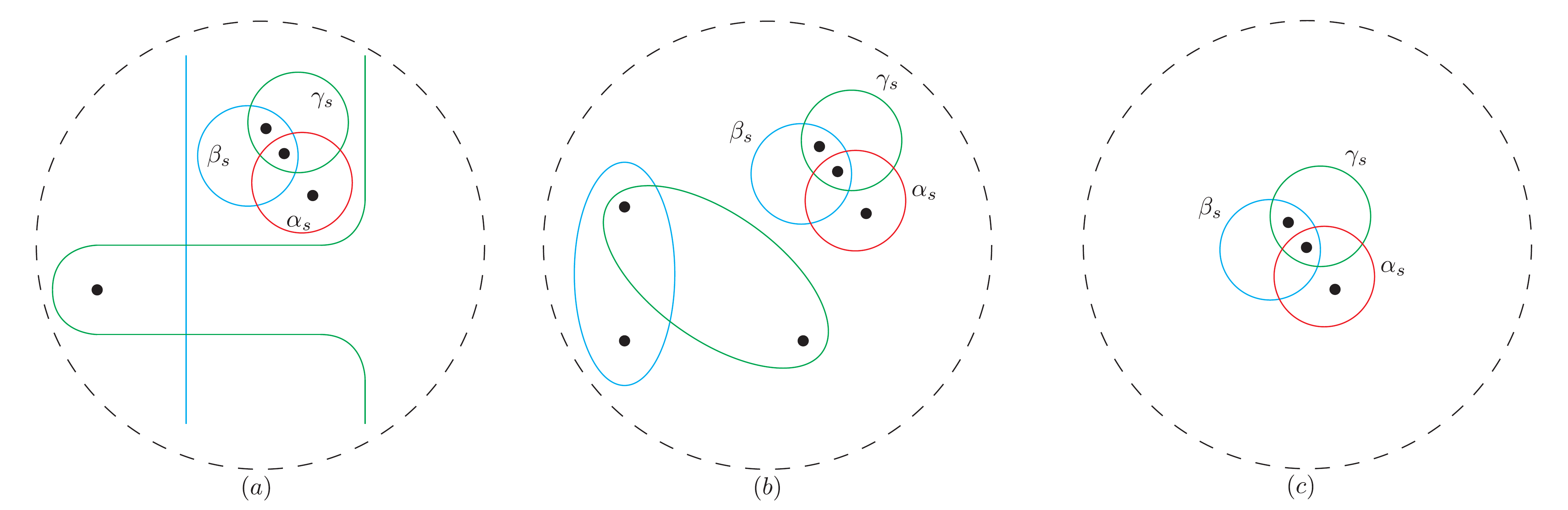}
  \caption{Free stabilized Heegaard triple for Case 2 (a)(b)(c). Note that, for simplicity we don't draw other alpha curves on the
picture.}
  \label{Fig:cb-003}

\end{figure}

\begin{proof}

The construction of the stabilized Heegaard triple relies on the
bipartite link cobordism. 

For case 1: Let $\overline{\cal{T}}$ be
a standard Heegaard triple subordinate to the band move
$B^\beta$. Suppose that $\gamma_n$ is a small Hamilitonian isotopy
crossing the basepoints $(o_i,o_j)$. As shown in Figure \ref{Fig:cb-05}, we
choose a parallel copy $\beta_s$ of $\beta_n$, and we set $\gamma_s$
be a small Hamiltonian isotopy of $\beta_s$ without crossing any
basepoints $\mathbf{O}$.

For case 2: Based on the construction for case 1, we can handleslide
$\beta_s$ such that $\beta_s$ is contractible on the Heegaard surface
$\Sigma$. See Figure \ref{Fig:cb-06} for the construction of case 2b. We show the
stabilized Heegaard triple for case 2a,2b,2c in Figure \ref{Fig:cb-003}.
 
\end{proof}

\begin{figure}
  \centering
  \includegraphics[scale=0.5]{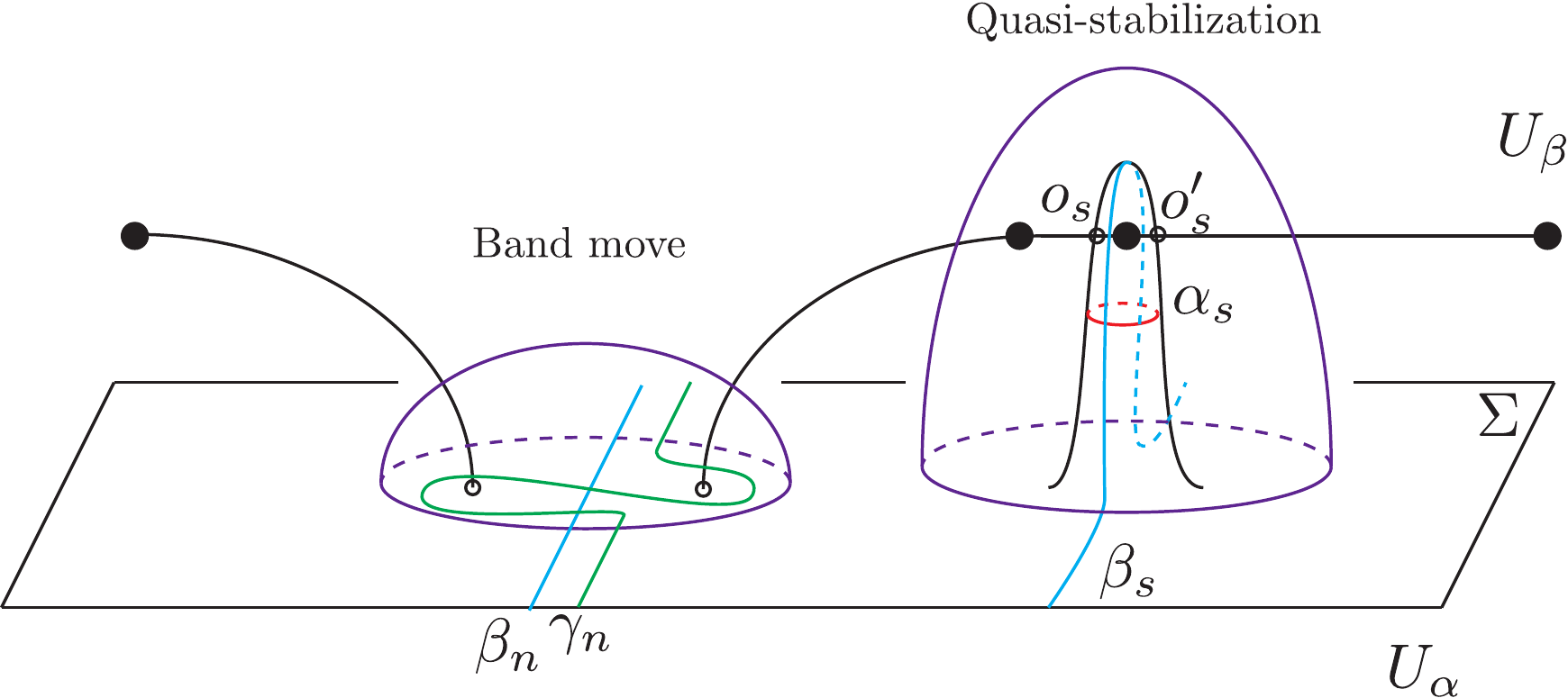}
  \caption{Type I band move and quasi-stabilization.}
  \label{Fig:cb-05}
\end{figure}

\begin{figure}
  \centering
  \includegraphics[scale=0.5]{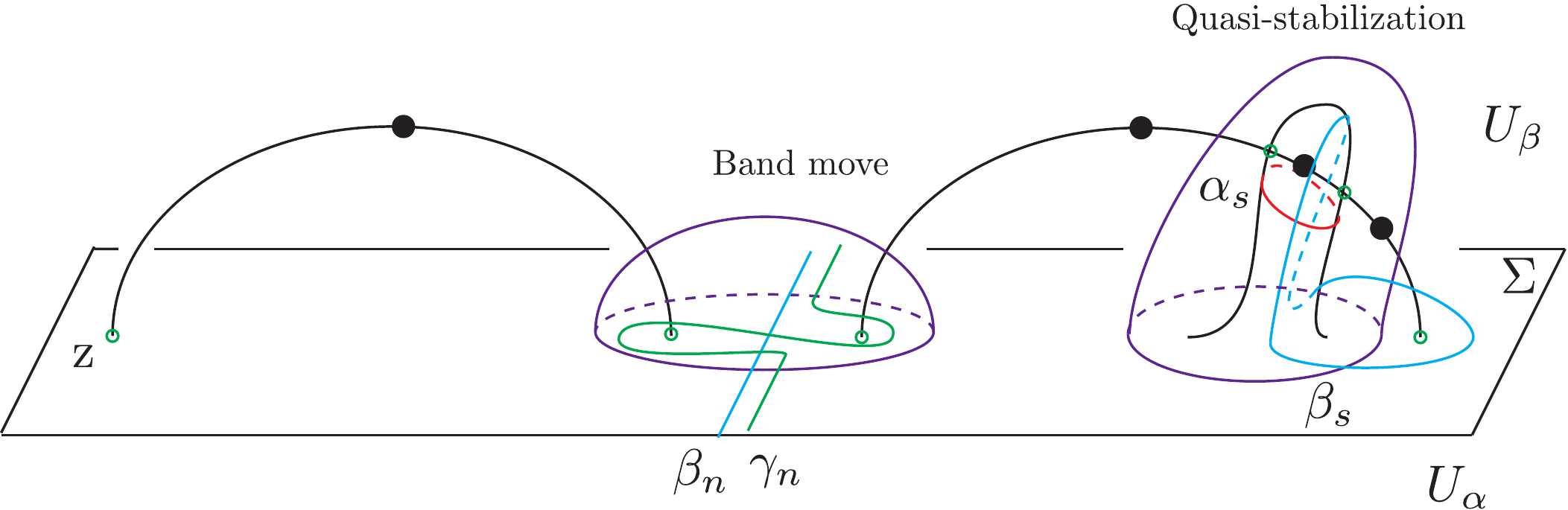}
  \caption{Type II band move and quasi-stabilization.}
  \label{Fig:cb-06}
\end{figure}

\begin{rem}\label{rem:stronglypositive}
By \cite[Lemma 6.3]{Zemke2016},
the Heegaard diagram
$(\Sigma,\bs{\beta}\cup\{\beta_s\},\bs{\gamma}\cup\{\gamma_s\})$,
which comes from the Heegaard triples we
constructed in Lemma \ref{sec:comm-betw-band-2}, are 
strongly positive. See \cite[Definition 6.2]{Zemke2016} for the
definition of strongly positive.
\end{rem}

Now we lift the bipartite link cobordism for the commutation between
the band move $B^\beta$ and quasi-stabilization $S^\beta$ to a bipartite
disoriented link cobordism. Consequently, the band move $B^\beta$ and
quasi-stabilization $S^\beta$ will also be lifted (they should be
adjacent to certain basepoints with respect to different lifting). 
Let's focus our discussion on the lifting for the commutation of case
1.   

For the commutation of case 1, we have two possible lifting: 
\begin{itemize}
\item the commutation $S^{\beta,o_i}\circ B^{\beta,o_i,o_j} \cong
  B^{\beta,o_j}\circ S^{\beta,o_i,o_j}$, as shown in the  Figure
  \ref{Fig:cb-04} (a);
\item the commutation $S^{\beta,o_j}\circ B^{\beta,o_i,o_j} \cong
  B^{\beta,o_i}\circ S^{\beta,o_i,o_j}$, as shown in the Figure
  \ref{Fig:cb-04} (b).
\end{itemize}

\begin{figure}
  \centering
  \includegraphics[scale=0.5]{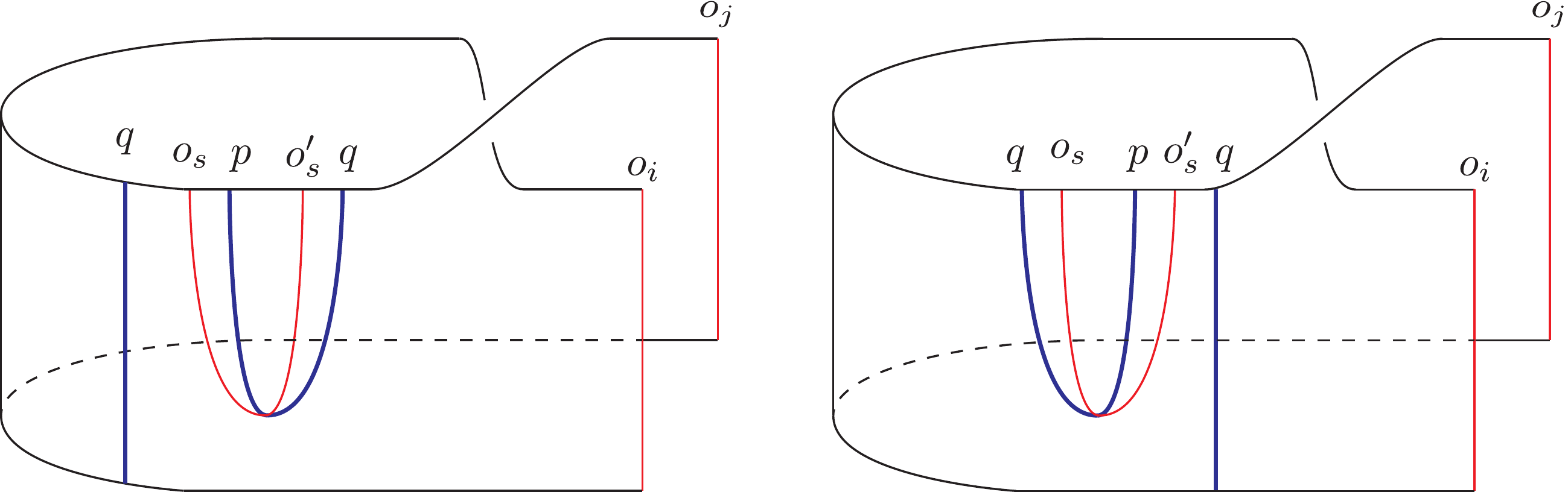}
  \caption{Two possible liftings to bipartite disoriented link cobordisms.}
  \label{Fig:cb-04}
\end{figure}

\begin{lem}\label{sec:comm-betw-band-1}
For the commutation of bipartite disoriented link cobordism $S^{\beta,o_i}\circ B^{\beta,o_i,o_j} \cong
  B^{\beta,o_i,o_j}\circ B^{\beta,o_j} $ in case 1, we have the following
  commutation on the unoriented link Floer chain complex up to
  chain homotopy:
\[F_{S^{\beta,o_i}}\circ F_{B^{\beta,o_i,o_j}} \simeq
 F_{B^{\beta,o_i,o_j}}\circ F_{S^{\beta,o_j}}.\]
A similar result holds for the commutation  $S^{\beta,o_j}\circ B^{\beta,o_i,o_j} \simeq
  B^{\beta,o_i,o_j}\circ S^{\beta,o_i}$.
\end{lem}

By Lemma \ref{sec:comm-betw-band-2}, we have a Heegaard triple (see
Figure \ref{Fig:cbqs-07})
subordinate to this commutation. We have to show the following
commutation diagram:

\[\xymatrix{
CFL'(\overline{H}_{\alpha\beta})\ar[d]_{S^{\beta,o_i}} \ar[r]^{B^{\beta,o_i,o_j}} &
CFL'(\overline{H}_{\alpha\gamma}) \ar[d]^{S^{\beta,o_j}} \\
CFL'(H_{\alpha\beta})\ar[r]^{B^{\beta,o_i,o_j}} & CFL'(H_{\alpha\gamma})}.\]
Here, we suppose that:
\begin{itemize}
\item the map $F_{S^{\beta,o_i}}$ send the generator $\x$ to
  $\x\times x^+$, where $x^+$ is an intersection of $\alpha_s$ and
  $\beta_s$ determined by $S^{\beta,o_i}$;
\item the map $F_{S^{\beta,o_j}}$ send the generator $\mathbf{z}$ to
$\mathbf{z}\times z^+$, where $z^+$ is an intersection of $\alpha_s$ and
  $\gamma_s$ determined by $S^{\beta,o_j}$;
\item the map $F_{B^{\beta,o_i,o_j}}$ send the generator $\x$ to
  $F_{\overline{\cal{T}}}(\x\otimes \Theta_{\beta\gamma})$, where the generator
  $\Theta_{\beta\gamma}$ is determined by $B^{\beta,o_i,o_j}$.
\item  the map $F_{B^{\beta,o_i,o_j}}$ send a generator of the form $\x\times x^+$ to
  $F_{\cal{T}}((\x\times x^+)\otimes (\Theta_{\beta\gamma}\times
  \theta^+))$, where $\theta^+$ is the intersection of $\beta_s$ and
  $\gamma_s$ determined by $B^{\beta,o_i,o_j}$.
\end{itemize}

Now we consider the triple $\cal{T}$ as the connected sum of the
triple $\cal{T}^1=(\Sigma^1,\bs{\alpha},\bs{\beta}\cup
\beta_s,\bs{\gamma}\cup\gamma_s)$ and the triple
$\cal{T}^2=(S^2,\alpha_s,\beta_s,\gamma_s)$ as shown in Figure \ref{Fig:cb-09}. 
Notice that, if we remove the $\beta_s$ and $\gamma_s$ curve form
$\cal{T}^1$ we will get the triple $\overline{\cal{T}}$.

\begin{figure}
  \centering
  \includegraphics[scale=0.5]{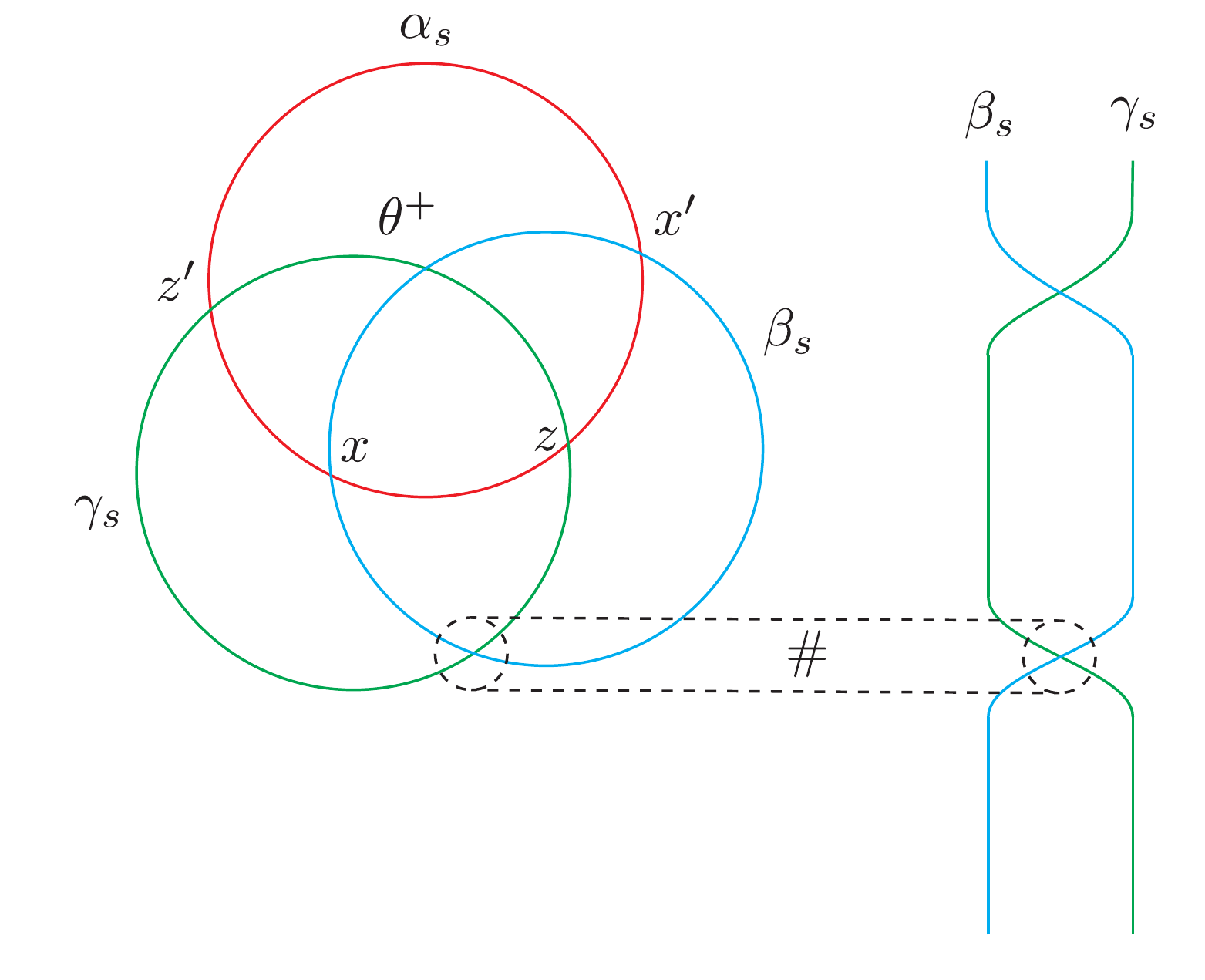}
  \caption{Connected sum of two Heegaard triples.}
  \label{Fig:cb-09}
\end{figure}

 Notice that
the intersection $\x\in
\mathbb{T}_{\bs{\alpha}\cup\{\alpha_s\}}\cap\mathbb{T}_{\bs{\beta}\cup\{\beta_s\}}$
for $H_{\alpha\beta}$,
can be uniquely written as $\x=\x^1\times \x^2$. Here the intersection $\x^1$ belongs
to $
\mathbb{T}_{\bs{\alpha}}\cap \mathbb{T}_{\bs{\beta}}$ for
$\overline{H}_{\alpha\beta}$, the intersection $\x^2$ belongs to $
\mathbb{T}_{\alpha_s}\cap \mathbb{T}_{\beta_s}$ for $(S^2,\alpha_s,\beta_s,\gamma_s)$.  

Similar to \cite[Lemma 6.9]{Manolescu2010}, we have a restriction map

\begin{lem}
  For the stabilized Heegaard triple $\cal{T}$ (as shown in Figure \ref{Fig:cbqs-07}) subordinate to a special
commutation, there is a well-defined restriction map:
\[\sigma:\pi_2(\mathbf{x},\mathbf{y},\mathbf{z})\rightarrow \pi_2(\mathbf{x}^2,\mathbf{y}^2,\mathbf{z}^2)\]
\end{lem}

Now we define an equivalence class for the triple
$(\phi^1,P_\beta,P_\gamma)$,  as follows.
We say that the triple $(\phi^1,P_\beta,P_\gamma)$ is equivalent to
$(\phi^{1'},P'_\beta,P'_\gamma)$, if we have the equality, 
\[\phi^1+P_\beta+P_\gamma=\phi^{1'}+P'_\beta+P'_\gamma.\]
Here, $\phi^1$ is in $\pi_2(\x^1,\y^1,\mathbf{z}^1)$, $P_\beta$ is in $
H_2(\Sigma^1,\bs{\beta}\cup\{\beta_s\})$, and $P_\gamma$ is in $
H_2(\Sigma^1,\bs{\gamma}\cup\{\gamma_s\})$.

\begin{lem}\label{sec:comm-betw-band-3}
  Suppose $\phi_2\in\pi_2(\x^2,\theta^+,\mathbf{z}^2)$ the restriction of
  $\phi\in \pi_2(\x,\Theta\times \theta^+,\mathbf{z})$ on
  $(S^2,\alpha_s,\beta_s,\gamma_s)$. Let $(\phi^1_1,
  P_\beta,P_\gamma)$ be the equivalence
  class determined by $\phi$. We have the following 
  Maslov index formula:
\[
\mu(\phi)=\mu(\phi^1)+\mu(P_\beta)+\mu(P_\gamma)-m_1(\phi)-m_2(\phi) +
\mu(\phi_2).
\]
 Here, the $m_1(\phi)$, $m_2(\phi)$, $a$ and $a'$are the multiplicities of the
regions of $\phi$ shown in Figure \ref{Fig:cb-08}.
\end{lem}

\begin{figure}
  \centering
  \includegraphics[scale=0.6]{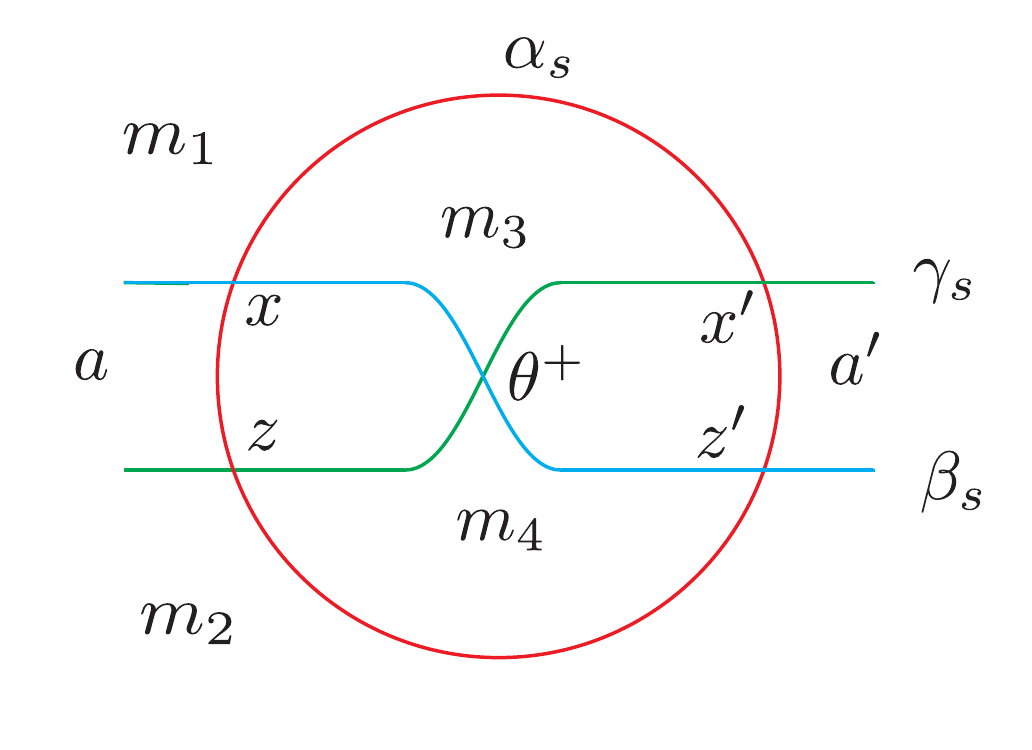}
  \caption{Local picture of Heegaard triple for Maslov index calculations.}
  \label{Fig:cb-08}
\end{figure}

\begin{proof}
By the Sarkar's Maslov index formula for holomorphic triangles
in \cite{Sarkar2006}, we have:
\[\mu(\phi)=\mu(\phi^1)+\mu(P_\beta)+\mu(P_\gamma)-\frac{1}{2}(m_1(\phi)+m_2(\phi)+a+a') +
\mu(\phi_2).\] On the other hand, as the intersection $\theta^-$
(shown in Figure \ref{Fig:cbqs-07}) does
not belong to the generator, so we have $m_1+m_2=a+a'$. Hence, we get
the Maslov index formula.
\end{proof}

\begin{lem}\label{sec:comm-betw-band-4}
  For a class $\phi_2\in \pi_2(\x^2,\theta^+,\mathbf{z}^2)$, where $\theta^+ $ is the
  intersection shown in Figure, we have the following Maslov index
  formula:
\[\mu(\phi_2)=m_1+m_2+m_3+m_4\]
\end{lem}

\begin{proof}
Straightforward. 
\end{proof}
\begin{proof}[Proof of Lemma \ref{sec:comm-betw-band-1}] 
We use Lipshitz's cylindrical formulation to count the holomorphic
triangles, see \cite{Lipshitz2006}.
  Let $\phi\in\pi_2(\x,\y,\mathbf{z})$ with Maslov index zero. By
  Lemma $\ref{sec:comm-betw-band-3}$ and Lemma \ref{sec:comm-betw-band-4}, we have:
\begin{align*}
\mu(\phi) &=\mu(\phi^1)+\mu(P_\beta)+\mu(P_\gamma)-m_1(\phi)-m_2(\phi) +
\mu(\phi_2)\\
&= \mu(\phi^1)+m_3+m_4+\mu(P_\beta)+\mu(P_\gamma) =0.
\end{align*}
Hence, we have $\mu(P_\beta)=\mu(P_\gamma)=0$, the multiplicity
$m_3=m_4=\mu(\phi^1)=0$ and $m_1=m_2=a=a'=m$.
 Therefore we have $\phi^1\in \pi^2(x,\theta,y)$ or
$\phi_1\in \pi_2(x',\theta,y')$. This is determined by the bipartite
disoriented link cobordism, shown in Figure \ref{Fig:cb-04}.

On the other hand, we consider the connected sum 
\[\Sigma(T)=(\Sigma^1-D^1)\#([-T-1,T+1]\times S^1)\#(S^2-D^2).\] Here, $D^1$ is
a small disk neighborhood of $p^1$ on $\Sigma^1$, and $D^2$ is a small disk
neighborhood of $p^2$ on $S^2$, see Figure \ref{Fig:cb-08}. We pick up an almost complex structure $J_1\in
\Sigma\times \Delta$ and $J_2$ on $\Sigma^2\times \Delta$.
 We
construct an almost complex structure $J(T)$ on $\Sigma(T)\times
\Delta$. 
By Remark \ref{rem:stronglypositive}, we know that, when $T\rightarrow \infty$,
the sequence of holomorphic
triangles will have a subsequence converges to some holomorphic objects
on $\Sigma^1$ and on $S^2$. On $S^2$, the objects are some broken
triangles; on $\Sigma$, the objects are some annoying $\beta$ degenerations, annoying $\gamma$ degenerations and broken triangles (see \cite[Lemma
6.3]{Manolescu2010} for the definition of annoying curves).
 As we have $m_3=m_4=0$,
the objects should be of the form $(\phi^1,0,0)$. Hence, for
sufficient large necklength, no boundary point of $S$
can be mapped to $p^2$ under the projection $\pi_\Delta\circ u$. Here
$u:S\rightarrow \Sigma$ is a holomorphic representative of $\phi$.
As in \cite[Proposition 6.15]{Manolescu2010}, we can identify
the moduli spaces $\cal{M}(\phi)$ with some fiber product
$\cal{M}(\phi^1)\times_{\text{Sym}^m\Delta}\cal{M}(\phi^2)$. The
counting of holomorphic triangle argument is similar to that in
the proof of \cite[Proposition 6.2]{Manolescu2010}. 
\end{proof}

\begin{prop}\label{sec:comm-betw-band-5}
 Suppose we have a bipartite disoriented link cobordism made of a band move $B^\beta$ and a
  quasi-stabilization $S^\beta$, such that $B^\beta\circ S^\beta\cong
  S^\beta\circ B^\beta$. Then the maps associated to  $B^\beta\circ S^\beta
 $ and $ S^\beta\circ B^\beta$ on unoriented link
 Floer chain complex $CFL'$ are 
 chain homotopic to each other.
\end{prop}

\begin{proof}
For special commutation, see Lemma \ref{sec:comm-betw-band-1}. For
the remaining cases, as the triples shown in Figure \ref{Fig:cb-003}
are free stabilizations, we can apply \cite[Proposition
6.2]{Manolescu2010} to show the commutation at the level of $CFL'$.  
\end{proof}

\subsection{The relation between $\alpha$ and $\beta$-band moves}

\label{sec:com:abb}

Let $B^{\beta,o_i,o_j}$ be a $\beta$-band move from a bipartite
disoriented link
$(\cal{L}^0,\mathbf{O})$ to a bipartite disoriented link
$(\cal{L}^1,\mathbf{O})$. We consider a
gradient flow $\phi_{t}$ induced by a Morse function
$f$ compatible with $(\cal{L}^0,\mathbf{O})$. There
exist a $t_0$ such that the band $\phi_{t_0}(B^{\beta,o_i,o_j})$ lies
in the $\alpha$-handlebody with respect to the Heegaard decomposition
induced by $f$. For convenience, we denote by $B^{\alpha,o_i,o_j}$ the
$\alpha$-band $\phi_{t_0}(B^{\beta,o_i,o_j})$. The band move
$B^{\alpha,o_i,o_j}$ changes the bipartite disoriented link
$(\cal{L}^0,\mathbf{O})$ to a bipartite disoriented link
$(\cal{L}^2,\mathbf{O})$.
 By an isotopy $\sigma$ of $(\cal{L}^2,\mathbf{O})$ supported only
on a three-ball $D$, we get a biparite disoriented link
 $(\sigma(\cal{L}^2),\sigma(\mathbf{O}))$ with the following properties:
 \begin{itemize}
 \item the disoriented link $\sigma(\cal{L}^2)$ agrees with
   $\cal{L}^1$,
 \item the (ordered) set of basepoints $\sigma(\mathbf{O})$ and
   $\mathbf{O}$ differ by switching $o_i$ and $o_j$. 
 \end{itemize}

\begin{figure}
  \centering
  \includegraphics[scale=0.7]{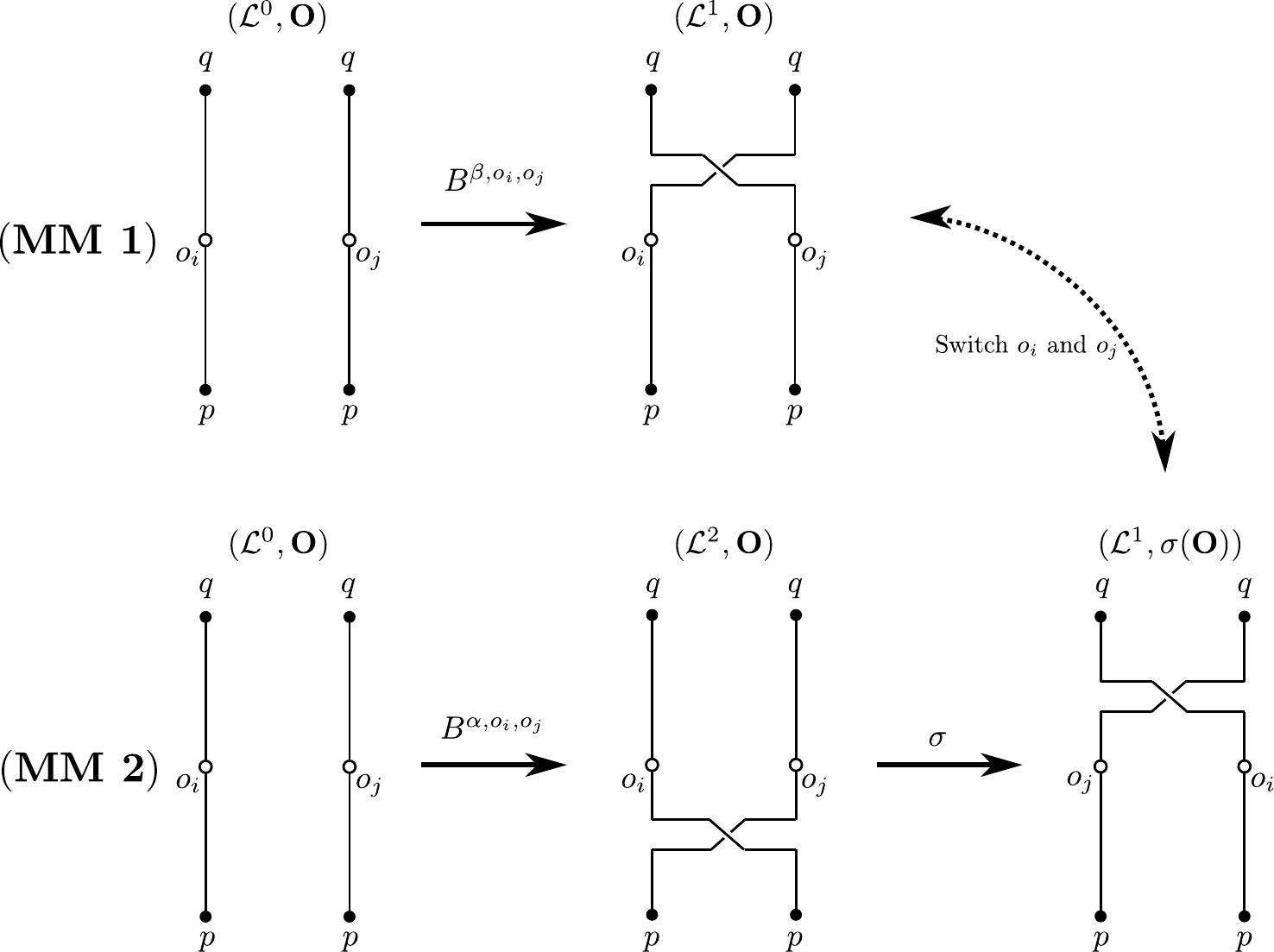}
  \caption{Two sequences of movie moves}
  \label{fig:mmabband}
\end{figure}

Similar to \cite[Proposition 7.8]{Zemke2016a}, we have the following relation between
the band move maps induced by the two band moves $B^{\beta,o_i,o_j}$, and $B^{\alpha,o_i,o_j}$.

\begin{prop}\label{thm:abband}
Let 
the band moves $B^{\beta,o_i,o_j}$ and
$B^{\alpha,o_i,o_j}=\phi_{t_0}(B^{\beta,o_i,o_j})$ be as described
above. We have the two sequences of moves between bipartite disoriented
links as shown in Figure \ref{fig:mmabband}: 
\begin{enumerate}[(\textbf{MM} 1):]
\item $(\cal{L}^0,\mathbf{O})\xrightarrow{B^{\beta,o_i,o_j}} (\cal{L}^1,\mathbf{O})$
\item $(\cal{L}^0,\mathbf{O}) \xrightarrow{B^{\alpha,o_i,o_j}}
  (\cal{L}^2,\mathbf{O})\xrightarrow{\sigma} (\sigma(\cal{L}^2)=\cal{L}^1,\sigma(\mathbf{O}))\xrightarrow{\text{switching
   } o_i \text{ and } o_j} (\cal{L}^1,\mathbf{O})$.
\end{enumerate}
At the level of unoriented link Floer chain complex $CFL'$, the maps induced
by the above two sequence of moves are chain homotopic to each
other. In other words, we have: 
  \[F_{B^{\beta,o_i,o_j}}\simeq \textnormal{Id}_{\textnormal{renum}}\circ
  \sigma_* \circ F_{B^{\alpha,o_i,o_j}}.\]
  Here the map $\operatorname{Id}_{\textnormal{renum}}$ on $CFL'$ is the identity map induced by
  switching the basepoints $o_i$ and $o_j$, and the map $\sigma_*$ is the
  canonical isomorphism induced by the isotopy $\sigma$.
\end{prop}

\begin{proof}

  By Theorem \ref{sec:exist-heeg-triple}, we consider a standard Heegaard Triple
  $\cal{T}_{\alpha\beta\gamma}$ subordinate to
  $B^{\beta,o_i,o_j}$. On the same Heegaard surface, we construct the $\delta$-curves
  such that,
\begin{enumerate}
\item the curves $\delta_1,\cdots,\delta_{n-1}$ are small Hamiltonian isotopies
  without crossing any baspoints,
\item the curve $\delta_n$ is a small Hamiltonian isotopy of $\alpha_n$ crossing
  the basepoints $o_i$ and $o_j$. The geometric intersection
  $|\delta_n\cap\alpha|=2=|\delta_n\cap\delta\cap D|$.
\end{enumerate} 
The Heegaard triples $T_{\alpha\beta\gamma}$ and
$T_{\delta\alpha\beta}$ are shown in Figure
\ref{fig:twoheegaardtriple}. It is easy to see that
$T_{\delta\alpha\beta}$ is subordinate to the band move
$B^{\alpha,o_i,o_j}$.

\begin{figure}
  \centering
  \includegraphics[scale=1.4]{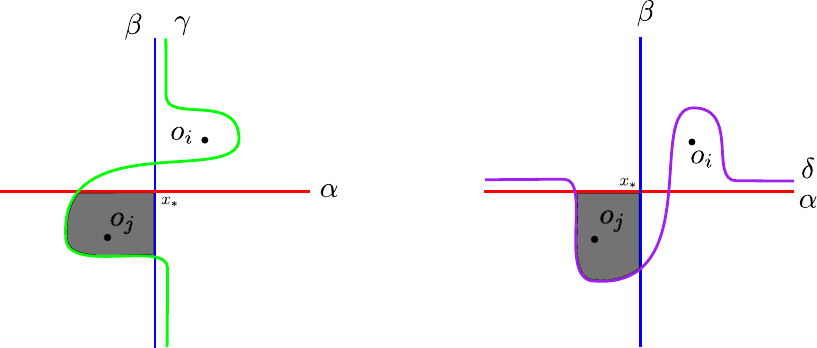}
  \caption{Heegaard Triple $\mathcal{T}_{\alpha\beta\gamma}$ and
    $\mathcal{T}_{\delta\alpha\beta}$}
  \label{fig:twoheegaardtriple}
\end{figure}

Without loss of generality, we can require the isotopy $\sigma$ to
satisfy the following: 

\begin{enumerate}
\item $\sigma$ preserves the Heegaard surface $\Sigma$ setwise,
\item $\sigma|_{(\Sigma\backslash D)}= \operatorname{Id}|_{(\Sigma\backslash D)}$,
\item $\sigma(o_i)=o_j$, $\sigma(o_j)=o_i$ and $\sigma(o_k)=o_k$ if
  $k\neq i,j$.
\item after switching the markings $o_i$ and $o_j$, the curves $\sigma(\delta)$ are small Hamiltonian isotopies of $\alpha$
  without crossing basepoints and the curves $\sigma(\beta)$ are small
  Hamiltonian isotopies of $\gamma$.  
\end{enumerate}
The two sequences of movie  moves in Figure
\ref{fig:mmabband} induce two sequences of Heegaard 
moves shown in Figure \ref{fig:isotopyhd}.

\begin{figure}
  \centering
  \includegraphics[scale=0.8]{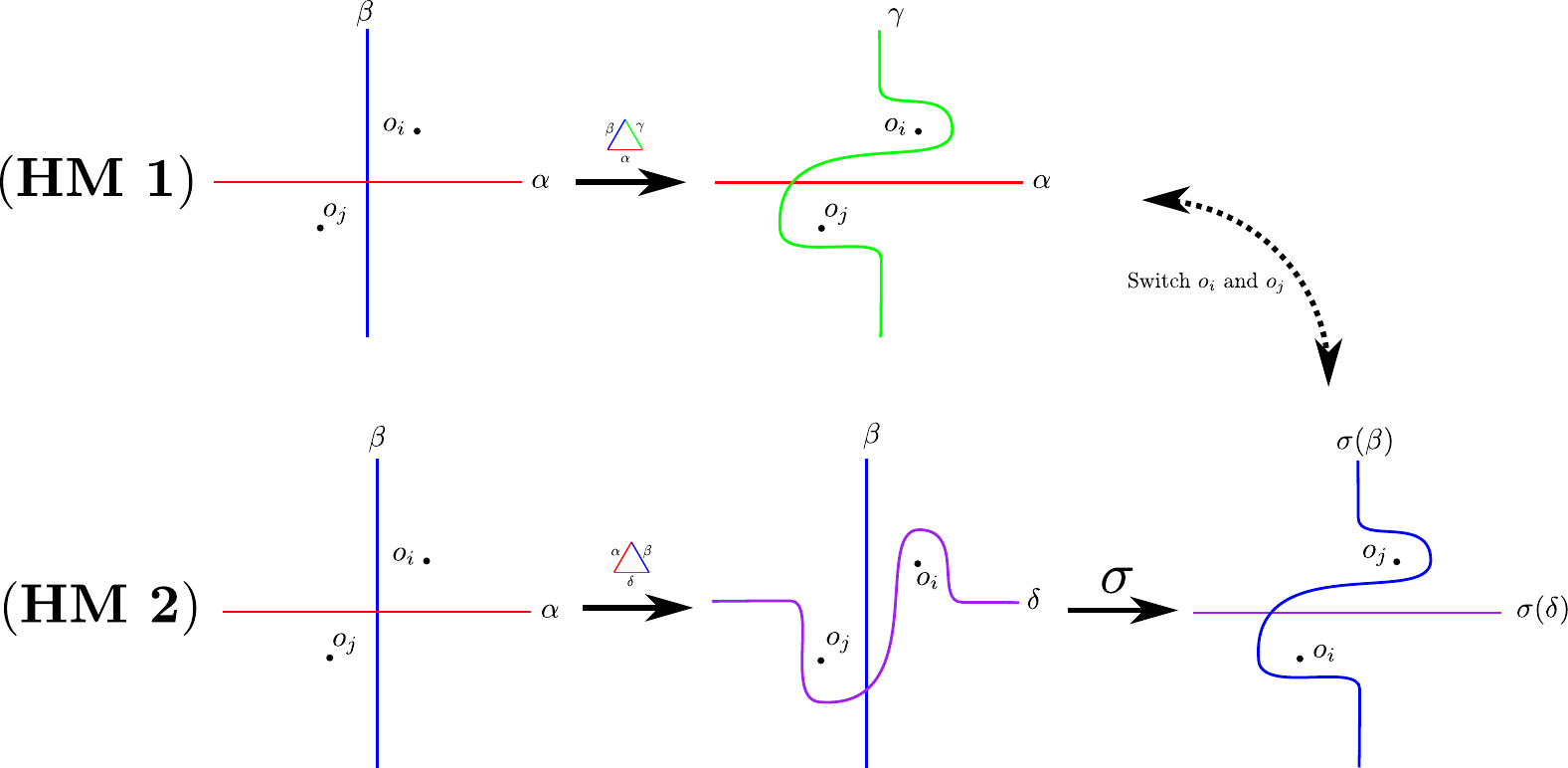}
  \caption{Sequence of Heegaard moves}
  \label{fig:isotopyhd}
\end{figure}

It suffices to show that,
\[\Phi_{\sigma(\delta)\sigma(\beta)}^{\alpha\gamma}\circ\sigma_{*}\circ
  F_{B^{\alpha,o_i,o_j}}\cong F_{B^{\beta,o_i,o_j}}.\]Here the map
$\Phi_{\sigma(\delta)\sigma(\beta)}^{\alpha\gamma}$ is induced
by small isotopy without crossing any basepoints.

For the Heegaard triple $\cal{T}_{\alpha\beta\gamma}$ we have the nearest
point map $N_{\alpha}^{\beta\gamma}$, see \cite{Ozsvath2008a}. It sends a generator
$\x=(x_1,\cdots,x_n)$ to the unique generator $\mathbf{z}=(z_1,\cdots,z_n)$
where the crossing $x_i$ and $y_i$ are vertices of a small triangle.
Notice from Figure \ref{fig:twoheegaardtriple} that the shaded small
triangles going over with basepoints
$o_i$ or $o_j$. Hence, in our case the map $N_{\alpha}^{\beta\gamma}$ is no longer a
chain map. Therefore, we define a chain map
$\Psi_{\alpha}^{\beta\gamma}$ as 
\[\Psi_{\alpha}^{\beta\gamma}(\x)=
\begin{cases}
U\cdot N_{\alpha}^{\beta\gamma}, &\text{if } x_{*}\in \x \\
N_{\alpha}^{\beta\gamma} &\text{if } x_{*}\notin \x
\end{cases}.\]

Similarly to the proof of Proposition \ref{thm:osseq}, the map $\Psi_{\alpha}^{\beta\gamma}$ is
chain homotopic to $F_{B^{\beta,o_i,o_j}}$.

On the other hand, for the Heegaard triple
$\cal{T}_{\delta\alpha\beta}$ we also have a nearest point map
$N^{\beta}_{\delta\alpha}$. It sends a generator
$\x=(x_1,\cdots,x_n)\in \mathbb{T}_{\alpha}\cap \mathbb{T}_{\beta}$ to the unique generator
$\mathbf{z}=(z_1,\cdots,z_n)\in \mathbb{T}_{\delta}\cap \mathbb{T}_{\beta}$
where the crossings $x_i$ and $z_i$ are vertices of a small triangle.

We also define a  chain map
$\Psi^{\beta}_{\delta\alpha}$ as 
\[\Psi^{\beta}_{\delta\alpha}(\x)=
\begin{cases}
U\cdot N^{\beta}_{\delta\alpha} &\text{if } x_{*}\in \x, \\
 N^{\beta}_{\delta\alpha} &\text{if } x_{*}\notin \x
\end{cases},\] which is chain homotopic
to  $F_{B^{\alpha,o_i,o_j}}$. 

By our construction, after switching the markings $o_i$ and $o_j$,
the curve $\sigma(\delta)$ and $\alpha$ (or $\sigma(\beta)$ and
$\gamma$ resp.) are
differed by a small isotopy without crossing any basepoints. Hence the
 map $\Phi_{\sigma(\delta)\sigma(\beta)}^{\alpha\gamma}$ is chain
homotopic to a composition of a sequence of nearest point map. 
Hence the nearest point maps
$\Psi_{\alpha}^{\beta\gamma}$ and
$\Phi_{\sigma(\delta)\sigma(\beta)}^{\alpha\gamma}\circ
\sigma_{*}\circ \Psi^{\beta}_{\delta\alpha}$ are chain homotopic. 
This implies that
 $\Phi_{\sigma(\delta)\sigma(\beta)}^{\alpha\gamma}\circ
 \sigma_{*}\circ F_{B^{\alpha,o_i,o_j}}$ is chain homotopic to $F_{B^{\beta,o_i,o_j}}$. 

\end{proof}

\begin{proof}[Proof of Lemma \ref{sec:comm-betw-beta-2}] 

We denote by $\mathbb{W}^1$ the bipartite disoriented link cobordism
$B^1\circ B^2$ and by $\mathbb{W}^2$ the bipartite disoriented link
cobordism $B_2\circ B_1$. Both cobordism $\mathbb{W}^1$ and
$\mathbb{W}^2$ are from the bipartite disoriented link
$(\cal{L}_1,\mathbf{O}_1)$ to $(\cal{L}_1,\mathbf{O}_2)$. Without loss
of generality, we assume $B_2=B_2^{\beta,o_{i'},o_{j'}}$ and the
bipartite disoriented link cobordisms $\mathbb{W}^1$ and
$\mathbb{W}^2$  are as shown in the top
of Figure \ref{fig:comtwobetaband}. 

We consider another two bipartite disoriented link cobordism
$\mathbb{W}^{1'}$ and $\mathbb{W}^{2'}$ from
$(\cal{L}_1,\mathbf{O}'_1)$ to $(\cal{L}_1,\mathbf{O}'_2)$. Here, the
bipartite disoriented links
$(\cal{L}_1,\mathbf{O}'_1)$ and $(\cal{L}_1,\mathbf{O}_1)$ are
differed by a baspoints-moving map $\sigma_1$ which
moves $o_{i'}$ and $o_{j'}$ but fix all the other
basepoints; the bipartite disoriented links
$(\cal{L}_1,\mathbf{O}_2)$ and $(\cal{L}_1,\mathbf{O}'_2)$ are
differed by switching $o_{i'}$ and $o_{j'}$ and a baspoints-moving map $\sigma_2$ which
moves $o_{i'}$ and $o_{j'}$ but fix all the other
basepoints. See the bottom of
Figure \ref{fig:comtwobetaband} for  $\mathbb{W}^{1'}$ and
$\mathbb{W}^{2'}$.
  
We denote by $B^{\alpha}$ the $\alpha$-band $B_1^{\alpha,o_{i'},o_{j'}}$. 
By Proposition \ref{thm:abband}, we know that
\begin{align*}
  &F_{\mathbb{W}^1}=\sigma_2 \circ F_{\mathbb{W}^{1'}}\circ \sigma_1 =
  \sigma_2\circ F_{B^1}\circ F_{B_2}^\alpha \circ \sigma_1\\
   &F_{\mathbb{W}^2}=\sigma_2 \circ F_{\mathbb{W}^{2'}}\circ \sigma_1 =
  \sigma_2\circ F_{B_2}^\alpha \circ F_{B^1}\circ\sigma_1.
\end{align*}
 As the band move maps
$F_{B_1}$ commute with $F^\alpha_{B_2}$ by Lemma
\ref{sec:comm-betw-alpha-1}, we conclude that $F_{\mathbb{W}^1}$
agree with the map $F_{\mathbb{W}^2}$.

\end{proof}

\begin{figure}
  \centering
  \includegraphics[scale=0.8]{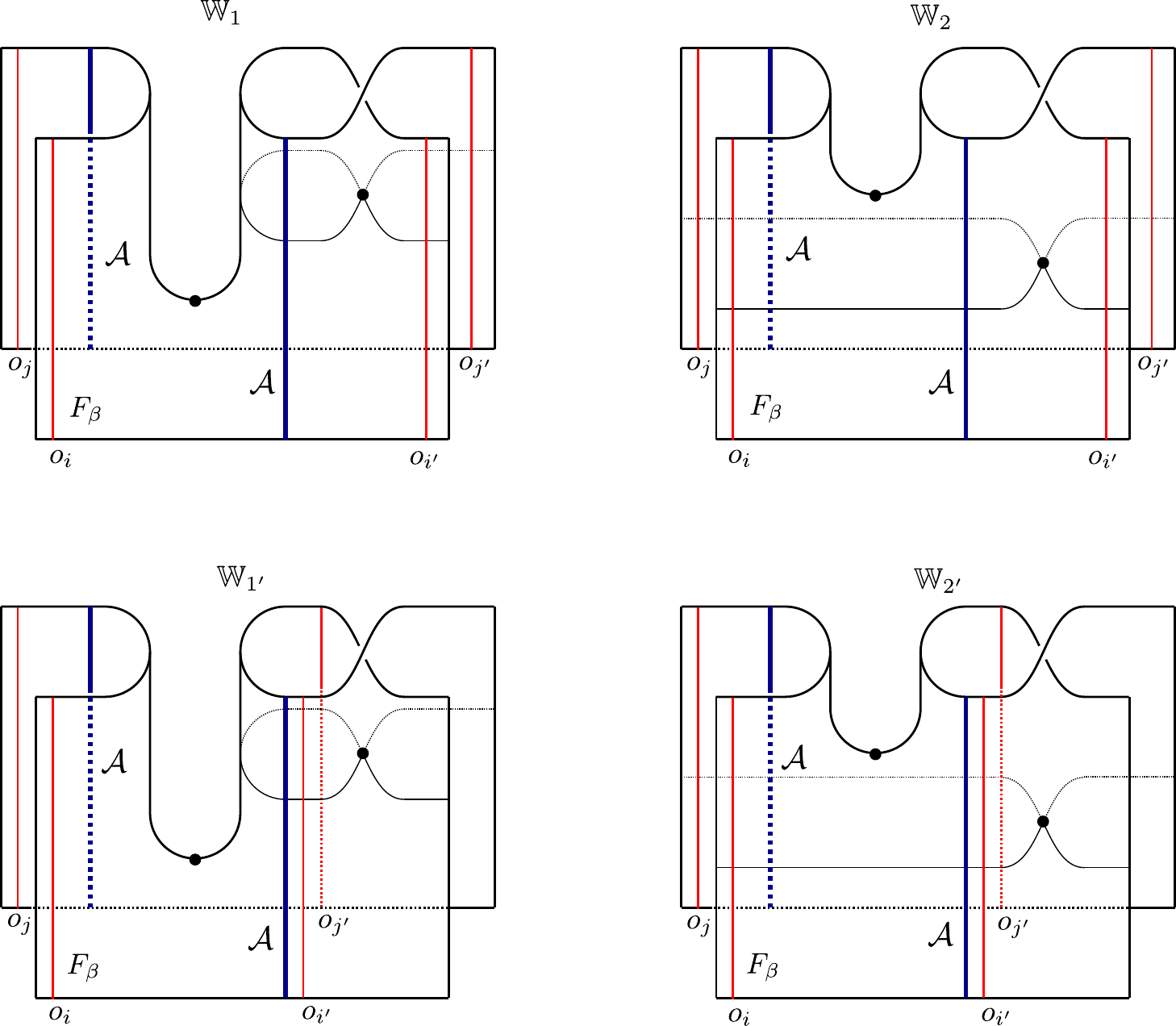}
  \caption{Commutation of two $\beta$-band moves.}
  \label{fig:comtwobetaband}
\end{figure}

\subsection{The relation between $\alpha$ and
  $\beta$-quasi-stabilizations} 
\label{sec:com:abq}

Let $S^{\beta,o_i}$ be a $\beta$-quasi-stabilization from
$(\cal{L}^0,\mathbf{O}^0)$ to $(\cal{L}^1,\mathbf{O}^1)$.
We also have a $\alpha$-quasi-stabilization $S^{\alpha,o_i}$,
which adds a pair of basepoints $(o,o')$ to the same bipartite disoriented link
$(\cal{L}^0,\mathbf{O}^0)$ and gets a bipartite disoriented link $(\cal{L}^2,\mathbf{O}^2)$. We provide a local picture for $S^{\alpha,o_i}$
and $S^{\beta,o_i}$ in Figure \ref{fig:relationabqs}. We can find an basepoint
moving map $\sigma$ (an isotopy preverses the link $L$), which satisfies the following,

\begin{itemize}
\item
the disoriented link $\sigma(\mathcal{L}^2)$ agrees with $\mathcal{L}^1$.
\item the (ordered) basepoint sets $\sigma(\mathbf{O}^2)$ and $\mathbf{O}^1$ differ by
  changing the ordering of the three basepoint $(o,o',o_i)$,
\item if a basepoint $o_l$ is not $o$, $o'$, or $o_i$ in
  $\mathbf{O}^2$, then $\sigma(o_l)=o_l$.
\end{itemize}

\begin{figure}
  \centering
  \includegraphics[scale=0.7]{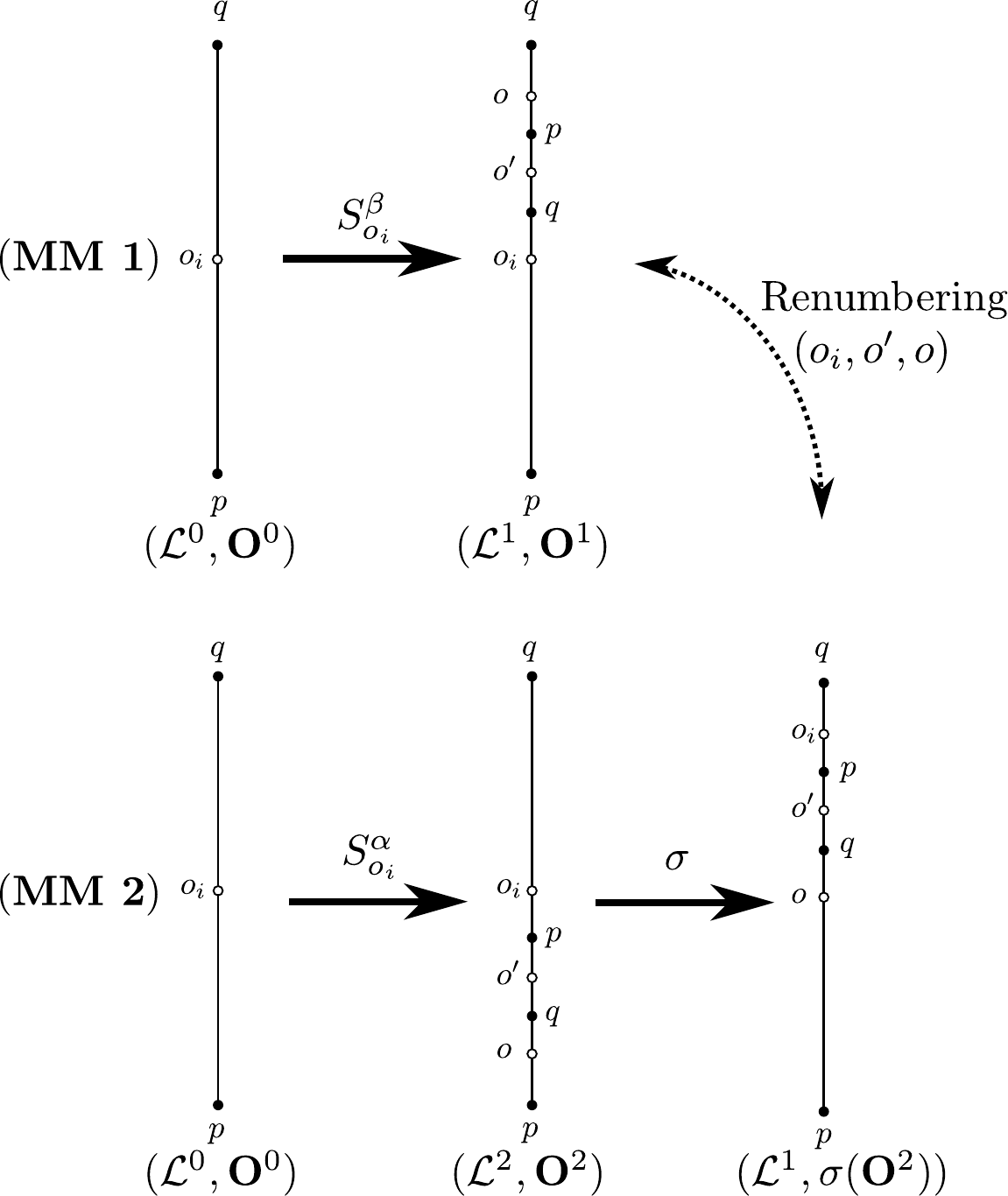}
  \caption{$\alpha$ and $\beta$-quasi-stabilizations}
  \label{fig:relationabqs}
\end{figure}

 At the level of $CFL'$, we have the following relation about the maps
induced by  $S^{\alpha,o_i}$ and  $S^{\beta,o_i}$. 

\begin{prop}\label{thm:abquasi}
  Let  $S^{\beta,o_i}$ and  $S^{\alpha,o_i}$ be two
  quasi-stabilizations described above. We consider the following two sequences of
  moves, as shown in Figure \ref{fig:relationabqs}:
\begin{enumerate}[(\textbf{MM} 1):]
  \item  $(\mathcal{L}^0,\mathbf{O}^0)\xrightarrow{S^{\beta,o_i}} (\mathcal{L}^1,\mathbf{O}^1)$,
  \item
    $(\mathcal{L}^0,\mathbf{O}^0)\xrightarrow{S^{\alpha,o_i}}(\mathcal{L}^2,\mathbf{O}^2)\xrightarrow
   {\sigma}(\sigma(\mathcal{L}^2)=\mathcal{L}^1,\sigma(\mathbf{O}^2)) \xrightarrow{\text{Renumbering
      }} (\mathcal{L}^1,\mathbf{O}^1)$.
\end{enumerate} At the level of $CFL'$, the maps induced by these two
sequence of movie moves are chain homotopic. In other words,
\[F_{S^{\alpha,o_i}}\cong \sigma\circ F_{S^{\beta,o_i}}.\] 
 If we reverse the two sequence of movie moves, similar results hold
 for quasi-destabilization. 
\end{prop}

\begin{proof} The proof is similar to the proof of \cite[Lemma
  3.23]{Zemke2016a}.  
  We compare the Heegaard diagrams shown in the left and right of
  Figure \ref{fig:relationabqs}. As we can require $\sigma$ to change the
  diagram $H_2$ in the right hand side to the Heegaard diagram $H_1$ in the
  left hand side (in fact, $\sigma(H_2)$ and $H_1$ differ by
  reordering $(o,o',o_i)$), it is easy to check the map $F_{S^{\alpha,o_i}}$
  agrees with $\sigma\circ F_{S^{\beta,o_i}}$. 
\end{proof}

\begin{figure}
  \centering
  \includegraphics[scale=0.7]{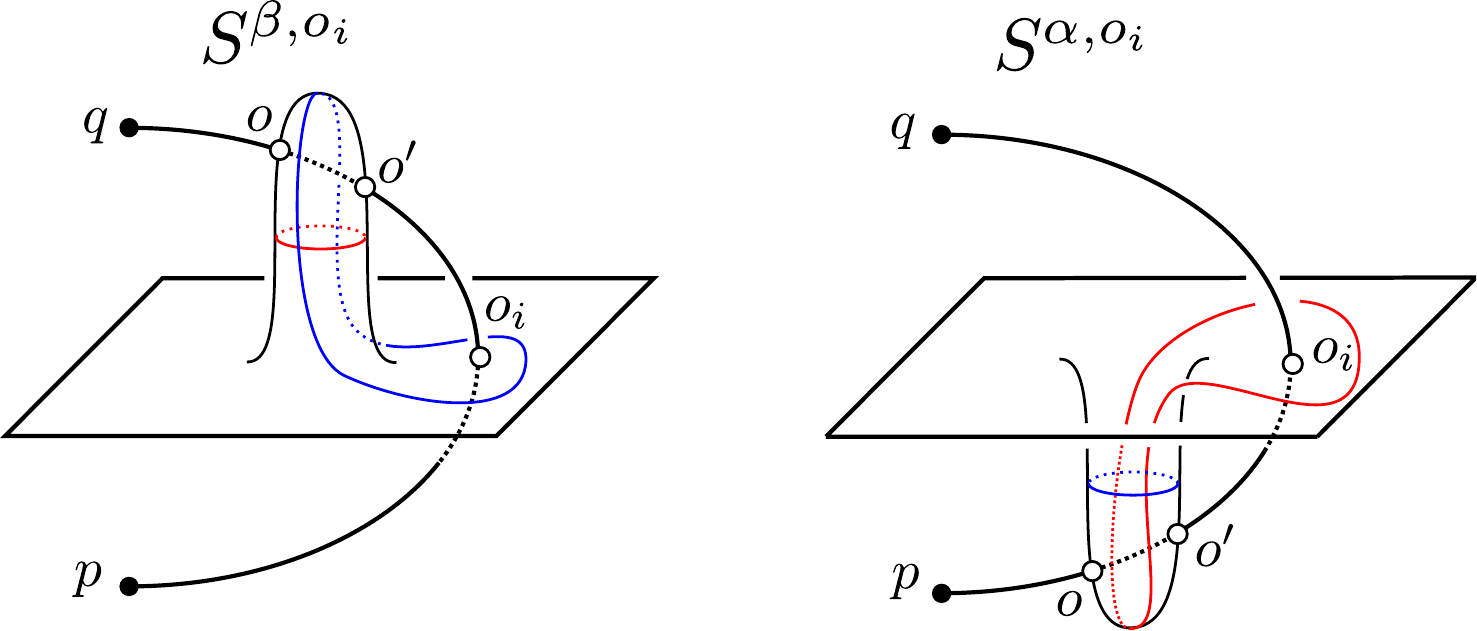}
  \caption{$\alpha$ and $\beta$-quasi-stabilizations}
  \label{fig:relationabqs}
\end{figure}

\section{Functoriality}

In this section, we assume that the
cobordism $W$ between three-manifolds is the product $Y\times I$. 
We will show that a disoriented link cobordism can be isotopied
to a regular form. We also show that any two regular forms of a given
disoriented link cobordism can be connected by certain moves. 
For the construction of the map, we isotope the disoriented link cobordism to a regular form
and then lift it to a bipartite
disoriented link cobordism. We construct the map from this disoriented
link cobordism.
Finally, we
prove the following:
\begin{itemize}
\item the map is independent of
liftings,
\item the link cobordism map is invariant under
the moves between regular forms of the disoriented link cobordism.
\end{itemize} These two results imply the invariance of our construction.

\subsection{Ambient isotopies of disoriented link cobordism}

In this section, we define an equivalence relation between disoriented
link cobordism analogous to the equivalence relation of knotted surface in
$\mathbb{R}^4$, see \cite[Chapter 2]{Carter1998}.

Recall that a disoriented link cobordism
$\mathfrak{W}=(\mathcal{W},\mathcal{F},\mathcal{A})$ from
  $(Y,\mathcal{L}_0)$ to $(Y,\mathcal{L}_1)$ contains the data of
  two maps:
\begin{itemize}
\item the embedding of an oriented one-manifold $\mathcal{A}:(A,\partial A)\hookrightarrow (F,\partial F)$,
\item the embedding of a surface $\mathcal{F}: (F,\partial F)\hookrightarrow (W,\partial W)$.
\end{itemize}

\begin{defn}

  We say that two disoriented link cobordism $\mathfrak{W}^i,i=0,1$
  are \textbf{\textit{equivalent}} or \textbf{\textit{ambient
      isotopic}}, if there exists a pair of smooth maps $h:(W\times
  I)\rightarrow W$ and $g:(F\times I)\rightarrow F$, such that
\begin{enumerate}[(1)]
\item the map $h_t=h(-,t): (W,\partial W) \rightarrow (W,\partial W)$
  is a diffeomorphism with $h_0=\operatorname{id}_{W}$ and $h_1\circ\mathcal{F}^0 = \mathcal{F}^1$.
\item the map $g_t=g(-,t): (F,\partial F) \rightarrow (F,\partial F)$
  is a diffeomorphism with $g_0=\operatorname{id}_{F}$ and $g_1\circ \mathcal{A}^0
  =\mathcal{A}^1$.
\end{enumerate}
\end{defn}

\begin{rem}
In fact, one can also define the equivalence relation between
disoriented link cobordism by using isotopy, i.e. a one-parameter
family of embeddings $(\mathcal{F}^t,\mathcal{A}^t)$. In our case, as the
manifolds are all compact, the isotopy extension theorem implies these
two definitions of equivalence of disoriented link cobordism are
 equivalent, see \cite[Chapter 8.1]{Hirsch1976}.
\end{rem}

As we require the cobordism $W$ to be a product $Y\times I$, there is a
natural projection $\pi:Y\times I \to I$. Recall that in section \ref{sec:lc:dl},
we said that a disoriented link cobordism is in regular form, if
$(\mathcal{F},\mathcal{A})$ satisfies the following:

\begin{enumerate}[(\textbf{R}1)]
\item The surface $\mathcal{F}(F,\partial F)$ is in generic position
  with respect to the natural projection $\pi$. In other words,
  $\pi\circ\mathcal{F}$ is a Morse function for $(F,\partial F)$.
\item The one-manifold $\mathcal{A}(A,\partial A)$ is in generic
  position with respect to the natural projection $\pi$. In other
  words, $\pi\circ\mathcal{F}\circ\mathcal{A}$ is a Morse function for
  $(A,\partial A)$. 
\item The index two/zero critical points of the Morse function
  $\pi\circ\mathcal{F}$ on $F$ are included in the image of index one critical
  points of $(A,\partial A)$ under the map $\mathcal{A}$.
\item The index one critical points of the Morse function
  $\pi\circ\mathcal{F}$ do not lie on the one-manifold
  $\mathcal{A}(A,\partial A)$.
\item The critical values of  $\pi\circ\mathcal{F}\circ\mathcal{A}$ and
  the critical values of  $\pi\circ\mathcal{F}$ corresponding to index
  one critical points are distinct on $I$. 
\item For a regular value $a\in I$, $(\pi^{-1}(a)\cap\mathcal{F}(F,\partial
  F),\pi^{-1}(a)\cap\mathcal{A}(A,\partial A))$ is a disoriented link
  in $\pi^{-1}(a)\cong Y$.
\end{enumerate}

\begin{thm}\label{thm:iso}
  For any disoriented link cobordism
  $(\mathcal{F},\mathcal{A})$, there exists a
  regular disoriented link cobordism
  $(\mathcal{F}',\mathcal{A}')$ which is isotopic to
  $(\mathcal{F},\mathcal{A})$. 
\end{thm}

\begin{proof}

  By linear perturbations of the embedding $\mathcal{F}$ in local
  charts, we can find an embedding $\mathcal{F}'$ which is close to
  $\mathcal{F}$ with respect to some metric and satisfies
  (\textbf{R}1). Similar to the proof in \cite[Proposition
  4.5]{Roseman2004}, we know that as $\mathcal{F}'$ is close to
  $\mathcal{F}$, they should be isotopic to each other. 
  
  Now we fix $\mathcal{F}$ and move the embedding $\mathcal{A}$ such
  that $\mathcal{A}(A,\partial A)$ go through all index zero and index
  two critical points of $\pi\circ\mathcal{F}$ on $(F,\partial F)$. As each piece of $F\backslash \cal{A}(A)$ has a boundary
  component in $\cal{A}(A)$, hence we move $\cal{A}(A)$ such that for
  each link component $L_i$  of $(\pi\circ \mathcal{F})^{-1}(a)$, the
  intersection $L_i\cap \mathcal{A}(A)$ is non-empty. By a further
  ambient isotopy which fixes the small neighborhood of all index two
  and index zero critical points of $\pi\circ\mathcal{F}$, we get a
  desired embedding $\mathcal{A}'$ satisfies   (\textbf{R}2) to
  (\textbf{R}6). See Figure \ref{fig:reg} for an example. 
\end{proof}

\begin{figure}
  \centering
  \includegraphics[scale=0.6]{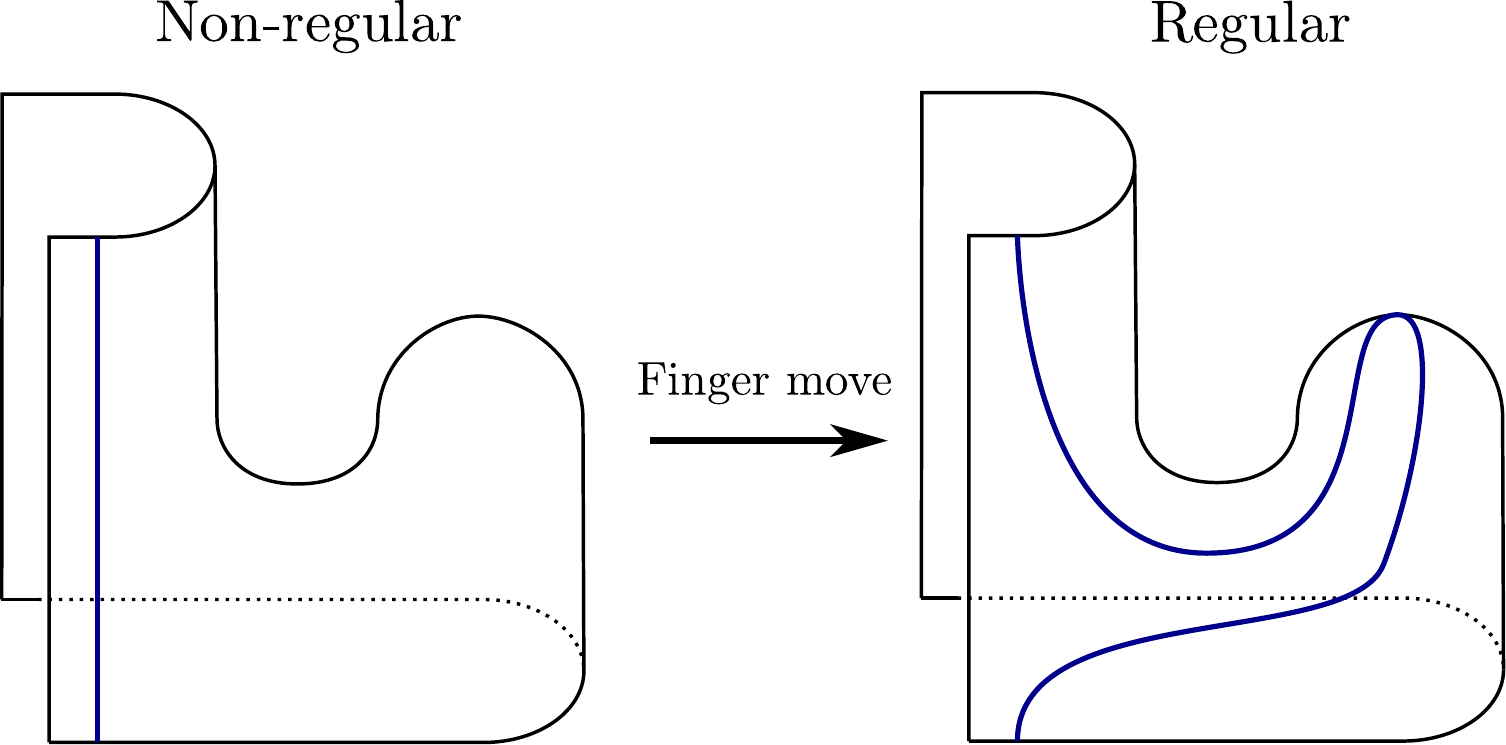}
  \caption{Isotopy to regular forms}\label{fig:reg}
\end{figure}

\subsection{Moves between regular forms of a disoriented link
  cobordism}

Recall that in Cerf theory, given a path connecting two Morse
functions $f_0$ and $f_1$, one can perturb and get a path of functions
$f_t$, such that $f_t$ is Morse function except at a finite set of $t\in
(0,1)$. Similarly, we will show that given a path of smooth embeddings
$(\cal{F}^t,\cal{A}^t)$ or equivalently an ambient isotopy connecting
two regular disoriented link cobordisms $(\cal{F}^0,\cal{A}^0)$ and
$(\cal{F}^1,\cal{A}^1)$, one can perturb the path of embeddings such
that $(\cal{F}^t,\cal{A}^t)$ is in regular form except a finite subset
$ E=\{t_1,t_2,\cdots,t_k\}$ of
$(0,1)$. We say that a path of pair of embedding satisfy the above
condition is \textbf{nice}.

To describe this path of embeddings, we pick up a finite subset of time
$R=\{a_1,\cdots,a_r\}\subset I\backslash E$ where $a_1=0$, $a_r=1$ and
$a_i<a_j$ if $i<j$, such that for all
component of $I\backslash E$ there exist at least one $a_i$ in $R$. 
Then, the path of embedding are represented by the
\textbf{\textit{move}} from the embedding $(\cal{F}^{a_i},\cal{A}^{a_i})$ to
$(\cal{F}^{a_{i+1}},\cal{A}^{a_{i+1}})$. 

We sort these move into 11 cases and provide examples of local
pictures of these moves in Figure \ref{fig:move1} and Figure \ref{fig:move2}.   

\begin{enumerate}[(\textbf{M }1):]
\item Ambient isotopies that do not change the level of all critical
  points of $\pi\circ\mathcal{F}$ and $\pi\circ\mathcal{A}$. 
\item Cancellation/birth of a pair of index zero/two and index one
  critical points of $\pi\circ\mathcal{F}$. 
\item Cancellation/birth of a pair of index one/zero critical points
  of $\pi\circ\mathcal{F}\circ \mathcal{A}$.
\item Ambient isotopies of $(A,\partial A)$ on $(F,\partial F)$ which
  go across an index two/zero critical point of $\pi\circ\mathcal{F}$
  and switch the height of two critical points of
  $\pi\circ\mathcal{F}\circ \mathcal{A}$. 
\item Ambient isotopies of $(A,\partial A)$ on $(F,\partial F)$ which
  go across an index one critical point of $\pi\circ\mathcal{F}$. 
\item Switching height of a pair of index one/zero critical points of
  $\pi\circ\mathcal{F}\circ \mathcal{A}$.
\item Switching the height of two index one critical points of
  $\pi\circ\mathcal{F}$. 
\item Switching the height of two index zero/two critical points of
  $\pi\circ\mathcal{F}$.
\item Switching the height of an index one critical points and a index
  two/zero critical point of $\pi\circ\mathcal{F}$.
\item Switching the height of an index zero/one critical points of
  $\pi\circ\mathcal{F}\circ \mathcal{A}$ and a index one critical
  point of $\pi\circ\mathcal{F}$.
\item Switching the height of an index zero/one critical points of
  $\pi\circ\mathcal{F}\circ \mathcal{A}$ and a distant index two/zero critical
  point of $\pi\circ\mathcal{F}$.

\end{enumerate}

\begin{figure}
  \centering
  \includegraphics[scale=0.6]{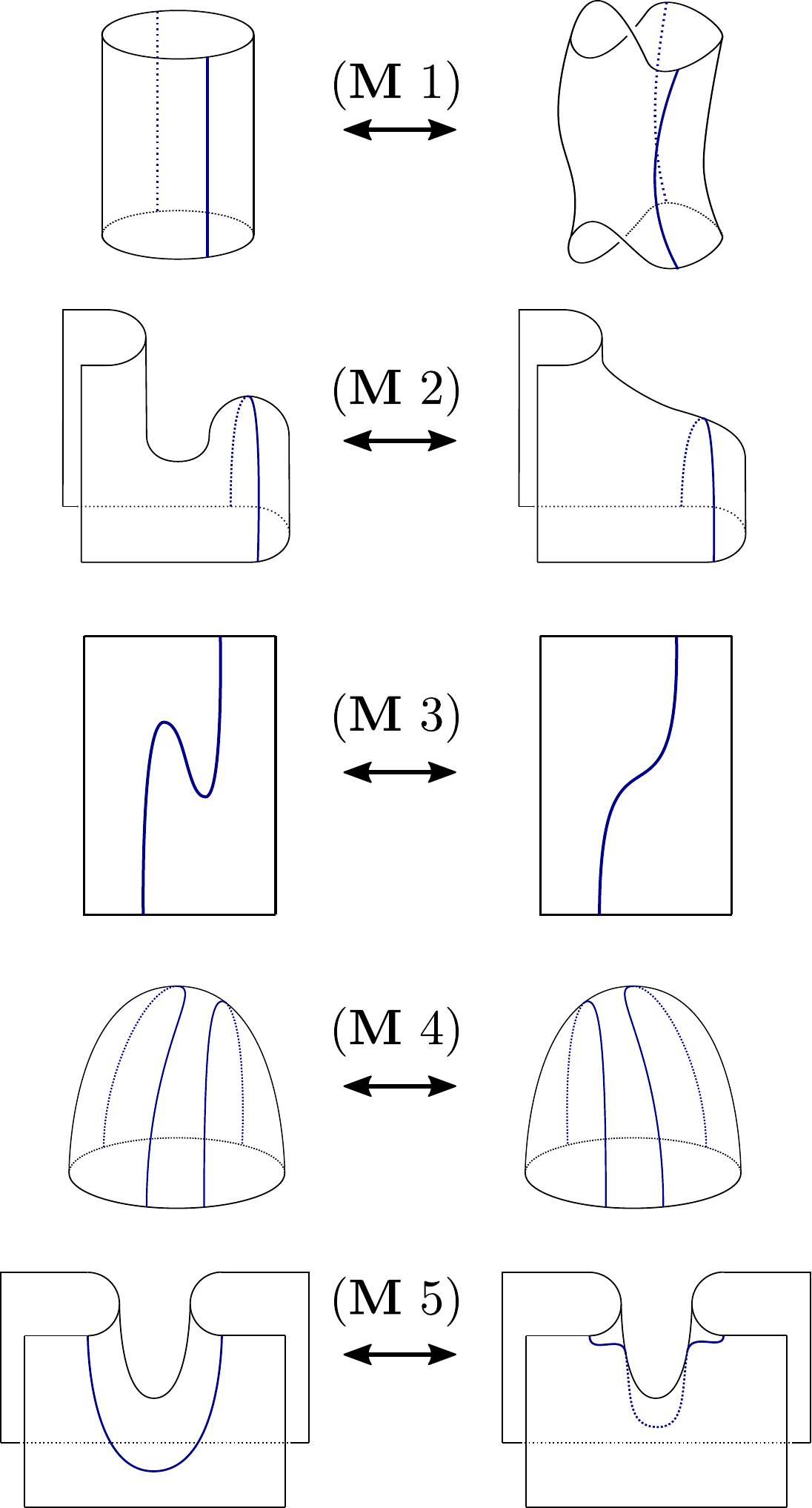}
  \caption{Moves between regular forms (Part I)}\label{fig:move1}
\end{figure}

\begin{figure}
  \centering
  \includegraphics[scale=0.6]{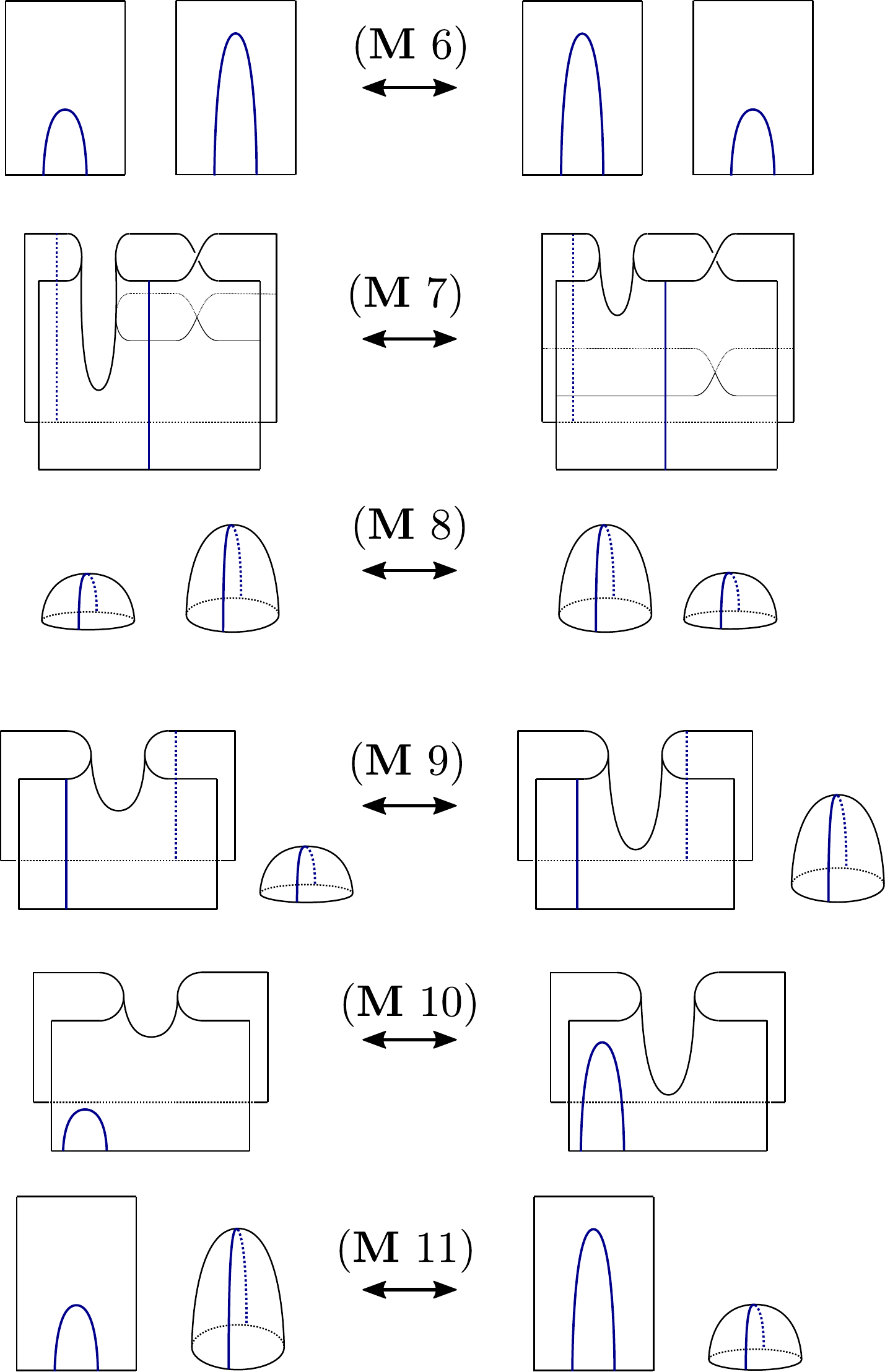}
  \caption{Moves between regular forms (Part II)}\label{fig:move2}
\end{figure}

\begin{thm}\label{thm:movesreg}
Any two regular forms of a disoriented link cobordism can be connected
by (\textbf{M} 1) to (\textbf{M} 11).
\end{thm}

\begin{proof}

  Let $(\mathcal{F}^0,\mathcal{A}^0)$ and
  $(\mathcal{F}^1,\mathcal{A}^1)$ be two regular forms of a
  disoriented link cobordism. By definition, there exists a path of
  embeddings $(\mathcal{F}^t,\mathcal{A}^t)$ connecting them. 

  Similar to the proofs in \cite[Proposition
  5.4]{Roseman2004},  the path of embeddings $\mathcal{F}^t$ can be
  approximated by another path $\cal{F}_{reg}^t$ such that
  $\pi\circ\cal{F}_{reg}^t$ is a 
  Morse function except a finite number of $t\in [0,1]$.
 The path of the pair of embeddings
 $(\cal{F}_{reg}^t,\cal{A}^0)$ may not be nice, but we can find
 another path of embedding $\cal{A}_{mid}^t$ starting from
 $\cal{A}^0$ such that the path of pairs of embeddings
 $(\cal{F}_{reg}^t,\cal{A}_{mid}^t)$ can be 
 represented by (\textbf{M} 1),(\textbf{M} 2),(\textbf{M} 7) 
 (\textbf{M} 8) and (\textbf{M} 9).

 The embedding $\cal{A}_{mid}^1$ may not be $\cal{A}^1$, therefore, we
 find the path of embedding $\cal{A}_{fin}^t$ from $\cal{A}_{mid}^1$ to $\cal{A}^1$ such
that the path of pair of embeddings $(\cal{F}^1,\cal{A}_{fin}^t)$ is
nice. This path can be represented by (\textbf{M} 1), (\textbf{M}
3), (\textbf{M} 4), (\textbf{M} 5), (\textbf{M} 6),  (\textbf{M} 10) and  
 (\textbf{M} 11).

Finally, we concatenate the two paths $(\cal{F}_{reg}^t,\cal{A}_{mid}^t)$
and  $(\cal{F}^1,\cal{A}_{fin}^t)$ to get a desired isotopy. See
Figure \ref{fig:egs} for an example of the isotopy.

\end{proof} 

\begin{figure}
  \centering
  \includegraphics[scale=0.6]{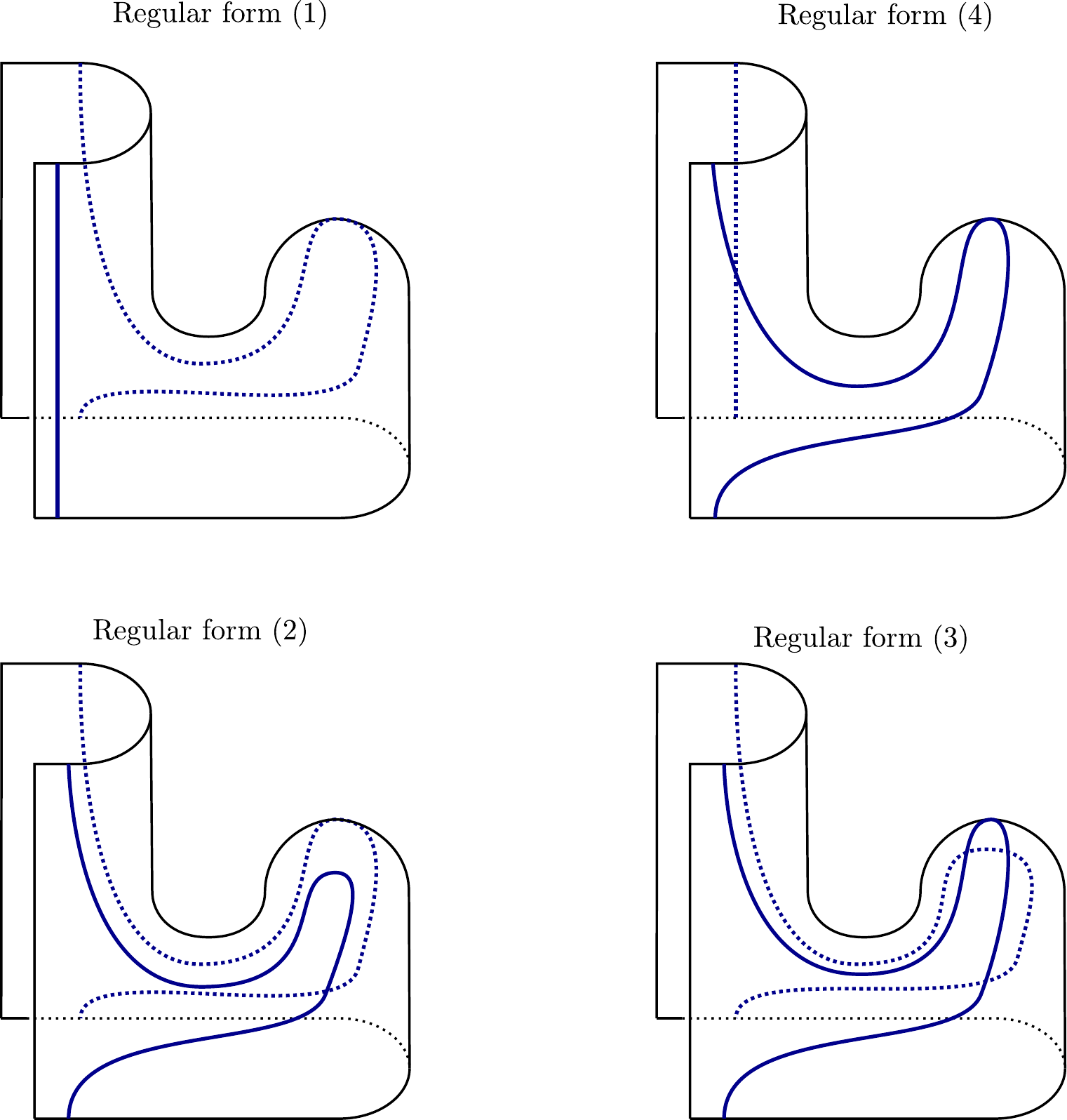}
  \caption{An example of a sequence of moves between regular
    forms. Regular form (1) to (2) and (3) to (4): (\textbf{M}
    1),(\textbf{M} 3),(\textbf{M} 6),(\textbf{M} 10),(\textbf{M}
    11); Regular form (2) to (3): (\textbf{M} 4).}\label{fig:egs} 
\end{figure}

\subsection{Construction and invariance of disoriented link cobordism maps}\label{sec:constr-invar-disor}

\begin{proof}[Proof of Theorem \ref{sec:introduction}] The proof include six steps: Regular form, lifting, elementary pieces,
  composition of maps, independence of liftings and moves between
  regular forms.  

  \emph{Step 1-Regular Form:} By theorem \ref{thm:iso} we can perturb
  $\mathfrak{W}$ and get a disoriented link cobordism
  $\mathfrak{W}_{reg}$ in regular form.  

\emph{Step 2-Lifting:} By the discussion in Section
\ref{sec:relat-betw-three}, we can lift $\mathfrak{W}_{reg}$ to a
bipartite disoriented link cobordism $\mathbb{W}_{reg}$. 

\emph{Step 3-Elementary pieces:} 
Then we cut $\mathbb{W}_{reg}$ into
elementary pieces $\mathbb{W}_{reg}^{i}$. These elementary pieces are categorified into four
types: isotopies, disk-stablization and destabilizations,
quasi-stabilization and destabilizations, band moves. 
We constructed a map $F_{\mathbb{W}_{reg}^{i}}$ on the unoriented link Floer
homology for each type of these elementary cobordism. 
For disk-stabilization/destabilization see the disscussion in the
begining of Section \ref{sec:constr-invar-disor}. For the definition
of band move maps, see Theorem \ref{sec:bipart-disor-link}. For the
definition of maps induced by quasi-stabilization/destabilizations,
see Section \ref{sec:quasi-stabilization}.

\emph{Step 4-Composition of maps:} We define the map $F_{\mathbb{W}}$ to be the
composition of the map $F_{\mathbb{W}_{reg}^{i}}$.  
The decomposition corresponds to a division of the interval
$I$. If we have another decomposition of $\mathbb{W}$, we can find a
common subdivision of $I$ for the two decomposition. By the invariance
of the maps induced by elementary pieces, we know
that subdivision induce the same map. Hence the map $F_{\mathbb{W}}$
is independent of the decomposition. 

\emph{Step 5-Independence of liftings:}
Recall that in the process of lifting $\mathfrak{W}_{reg}$, we lift each
elementary pieces $\mathfrak{W}_{i}$ to $\mathbb{W}_i$ and glue them
together. For an isotopy, a disk-stabilization or a disk-destabilization, we have
a unique way to lift. For a band move or a
quasi-stabilization/destabilization, we have two ways to
lift. Therefore it suffices to check that for a band move or a
quasi-stabilization/destabilization, the map on unoriented link Floer
homology we defined is independent of different liftings. By the
result in
Proposition \ref{thm:abband} and Proposition \ref{thm:abquasi}, we
get a well-defined map $F_{\mathfrak{W}_{reg}}=F_{\mathbb{W}_{reg}}$.

\emph{Step 6-Moves between regular forms:}
As one may perturb $\mathfrak{W}$ to different regular forms
$\mathfrak{W}_{reg1}$ or $\mathfrak{W}_{reg2}$ (by an small ambient
isotopy), we need to verify the map
is independent of the perturbations. By Theorem \ref{thm:movesreg}, we
know that the two regular forms $\mathfrak{W}_{reg1}$ and
$\mathfrak{W}_{reg2}$ are connected by a sequence of moves. It suffices
to
check the $F_{\mathfrak{W}_{reg}}$ is invariant under the moves. 

By definition, a move between regular forms can be realized as an ambient
isotopy $g_t$ of $\mathcal{F}$, where $g_t$ is a path of
diffeomorphism of $Y\times I$.  In \cite{Zemke2016a}, from a path of
Morse functions 
$\pi\circ g_t$ of $Y\times I$, we get a sequence of moves between
parametrized Kirby decompositions. We have the following table comparing
the moves between regular forms and the moves between parametrized
Kirby decompostion described in \cite{Zemke2016a}. 
\begin{center}
\begin{tabular}[5pt]{c|c}
 Regular forms & Parametrized Kirby decompositions \\
\hline
  (\textbf{M} 1) & \textbf{Move}(1)\\
  (\textbf{M} 2) & \textbf{Move}(7)\\
  (\textbf{M} 3) & \textbf{Move}(10)\\
 (\textbf{M} 4) (\textbf{M} 5) & \textbf{Move}(14) \\
  (\textbf{M} 6) & \textbf{Move}(11)\\
  (\textbf{M} 7) (\textbf{M} 8) (\textbf{M} 9) & \textbf{Move}(12)\\
  (\textbf{M} 10)(\textbf{M} 11) & \textbf{Move}(13)\\
\end{tabular} 
\end{center}
If the surface $F$ is orientable, the moves between regular forms can
be described by the moves in \cite{Zemke2016a}. Hence we only need to
verify (\textbf{M} 5)  (\textbf{M} 7)  (\textbf{M} 9) and (\textbf{M}
10) which can involve an unoriented band move. 

For(\textbf{M} 4), it is easy to verify the invariance by constructing Heegaard
diagrams for the two cobordisms (both are composition of a
disk-stabilization and a quasi-stabilization, so the Heegaard diagrams
can be very simple) and keep track of the top grading
generator. 
For (\textbf{M} 5), one can still apply the results in
\cite{Zemke2016a} (invariance of map under \textbf{Move}(14)) to unorientable
band moves and show that (\textbf{M} 5) does not change the map
$F_{\mathcal{W}}$. Actually, we can also lift the two disoriented link cobordisms in
(\textbf{M} 5) to bipartite disoriented link cobordisms. By the
discussion in Section \ref{sec:comm-betw-band} we can construct a Heegaard
triple subordinate to the band moves and show that the maps for the two
bipartite disoriented link cobordisms are the same. Simlarly, the invariance
under (\textbf{M} 9) can be directly varified by constructing certain
Heegaard triples as in Section \ref{sec:comm-betw-band} or in \cite{Zemke2016a}. For (\textbf{M} 7), we 
lift $\mathfrak{W}_{reg}$ to bipartite disoriented link cobordisms (see
Figure \ref{fig:com01} for an example) and 
 apply the results in Lemma \ref{sec:comm-betw-alpha-1} (the
commutation between a $\alpha$ and a $\beta$-band move)  and Lemma
\ref{sec:comm-betw-beta-2} 
(the commutation between two $\beta$-band moves) to show the
invariance.      
For (\textbf{M} 10), we lift the two disoriented link cobordism to bipartite disoriented link
cobordisms (see Figure \ref{Fig:cb-04} for an example) and apply Proposition \ref{sec:comm-betw-band-5}.

\end{proof}



\bibliographystyle{amsalpha}
\bibliography{ulc.bib}

\end{document}